\documentclass[11pt]{amsart}
\usepackage[margin=1in]{geometry}
\usepackage{amssymb}
\usepackage{amsthm}
\usepackage{amsmath}
\usepackage{mathrsfs}
\usepackage{amsbsy}
\usepackage{dynkin-diagrams}
\usepackage{bm}
\usepackage{tikz}
\usepackage{array}
\usepackage[normalem]{ulem}
\usepackage{enumerate}
\usepackage{bbm}
\usepackage{comment}
\usepackage{mathtools}
\usepackage{makecell} 
\usepackage{colortbl}
\usepackage{xcolor}
\usetikzlibrary{positioning,calc}
\usetikzlibrary{decorations.markings,arrows.meta}
\usetikzlibrary{fit} 
\usepackage{booktabs}
\usepackage{csquotes}
\usepackage{framed}
\usepackage{longtable}
\usepackage{pdflscape}
\usepackage{ragged2e}
    \definecolor{barclr}{HTML}{1A6FBF}
    \newcommand{\ov}[1]{\textcolor{barclr}{\overline{#1}}}
    \newcommand{\pbsep}{\mkern3mu\allowbreak}      
    \newcommand{\pbgap}{\mkern10mu\allowbreak}     
    \newcommand{\Xword}[1]{$#1$}                   
\usepackage{multirow} 
\usetikzlibrary{decorations.text}

\usepackage{mathbbol}

\DeclareSymbolFontAlphabet{\amsmathbb}{AMSb}
 
\newcolumntype{R}[1]{>{\RaggedRight\arraybackslash\footnotesize$}p{#1}<{$}}
 
  {\par\smallskip\noindent\textbf{#1:}\par\nopagebreak
   \footnotesize\RaggedRight $}%
  {$\par\smallskip}

\makeatletter
\def\paircount#1_#2{\csname count#1_#2\endcsname}
\makeatother

\DeclareFontFamily{U}{mathx}{}
\DeclareFontShape{U}{mathx}{m}{n}{<-> mathx10}{}
\DeclareSymbolFont{mathx}{U}{mathx}{m}{n}
\DeclareMathAccent{\widecheck}{0}{mathx}{"71}

\usepackage{microtype}
\usepackage{multicol}
\usepackage{paracol}
\usepackage{etoolbox}
\usepackage{caption}
\usepackage[nobreak,noadjust]{cite}
\usepackage{hyperref}
\hypersetup{colorlinks=true, citecolor=blue, linkcolor=blue,urlcolor=blue} 
\usepackage{esvect}

\DeclareFontEncoding{OT2}{}{} 
\newcommand{\textcyr}[1]{{\fontencoding{OT2}\fontfamily{wncyr}\fontseries{m}\fontshape{n}\selectfont #1}}\newcommand{\Sha}{{\mbox{\textcyr{Sh}}}}

\usepackage[noabbrev,capitalise,nameinlink]{cleveref}
\crefname{conjecture}{Conjecture}{Conjectures}
 
\Crefname{remark}{Remark}{Remarks}

\newtheorem{theorem}{Theorem}[section]
\newtheorem{proposition}[theorem]{Proposition}
\newtheorem{corollary}[theorem]{Corollary}
\newtheorem{conjecture}[theorem]{Conjecture}

\newtheorem{lemma}[theorem]{Lemma}

\theoremstyle{definition}
\newtheorem{definition}[theorem]{Definition}
\newtheorem{remark}[theorem]{Remark}

\newtheorem{example}[theorem]{Example}
\AtBeginEnvironment{example}{%
  \pushQED{\qed}%
}
\AtEndEnvironment{example}{\popQED}

\newcommand{\vword}{\protect\vv{\beta}}

\newcommand{\inv}{\mathrm{inv}}

\newcommand{\Dist}{\mathcal{D}}

\newcommand{\NC}{\mathrm{NC}}
\newcommand{\Sub}{\mathrm{Sub}}
\newcommand{\Subb}{\mathrm{sub}}
\newcommand{\Des}{\mathrm{Des}}
\newcommand{\Park}{\mathrm{Park}}
\newcommand{\PF}{\mathrm{PF}}

\newcommand{\rev}{\mathrm{rev}}
\newcommand{\Inv}{\mathrm{inv}}

\renewcommand{\t}{\mathbf{t}}
\renewcommand{\S}{\mathbf{S}}
\newcommand{\s}{\mathbf{s}}
\renewcommand{\u}{\mathbf{u}}
\newcommand{\su}{\mathsf{w}}

\usepackage[most]{tcolorbox} 
\newcommand{\Cat}{\mathrm{Cat}} 
\newcommand{\Pop}{\mathrm{Pop}}
\newcommand{\Conj}{\mathrm{Conj}}
\newcommand{\B}{\mathbf{B}}
\renewcommand{\P}{\mathbb{P}}
\renewcommand{\H}{\mathrm{H}}
\newcommand{\R}{\mathbf{T}}
\newcommand{\PR}{\mathbb{T}}
\newcommand{\bc}{\mathbf{B}(c)}
\newcommand{\vvl}[1]{\reflectbox{$\vv{\reflectbox{$#1$}}$}}
\newcommand{\vect}{\protect\vv{t}}
\newcommand{\vectt}{\protect\vv{\tt}}
\newcommand{\vece}{\protect\vv{\epsilon}}

\DeclareDocumentCommand{\T}{e{_^}}{%
  \mathcal{T}%
  \IfValueT{#1}{_{#1\vphantom{{}^{-1}}}}%
  \IfValueT{#2}{^{#2}}%
}
\newcommand{\HOMFLY}{\text{\sc homflypt}}
\newcommand{\homfly}{\text{\sc homflypt}}

\newcommand{\Orb}{\mathrm{Orb}}

\DeclareRobustCommand{\fold}{\raisebox{0.3ex}{\resizebox{\width}{0.5\height}{\rotatebox{90}{\textsf{w}}}}}

\usepackage[textsize=tiny]{todonotes}
\newcommand{\dfn}[1]{\textcolor{blue}{\emph{#1}}}
\newcommand{\defn}[1]{\textcolor{blue}{\emph{#1}}}

\newcommand{\os}{\overline{s}}

\newcommand{\RR}{T}

\newcommand{\Hur}{\mathrm{Hur}} 
\newcommand{\HT}[1]{\H(#1)}
\newcommand{\woc}{{\sf w}_\circ(c)}

\newcommand{\bbeta}{\boldsymbol{\beta}}     
\newcommand{\brho}{\boldsymbol{\rho}}
 
\renewcommand{\tt}{\mathbbm{t}}
\newcommand{\uu}{\mathbbm{u}}
\newcommand{\cc}{\mathbbm{c}}

\newcommand{\XX}{\cc_{p,\mathbf{g}}}

\def\mycycle{1,4,7,10,11,9,8,6,5,3,2}

\newcounter{clen}
\foreach \val in \mycycle {
    \expandafter\xdef\csname pos_\val\endcsname{\theclen}
    \expandafter\xdef\csname val_\theclen\endcsname{\val}
    \stepcounter{clen}
}

\makeatletter
\def\part{\@startsection{part}{0}%
  \z@{2\linespacing\@plus\linespacing}{1\linespacing}%
  {\normalfont\Huge\bfseries\centering\def\@secnumfont{\bfseries}}}
\makeatother

\title{Decategorification of Knot Invariants}
\author{Colin Defant}
\address{Harvard University}
\author{Nathan Williams}
\address{The University of Texas at Dallas}

\date{\today}
\usepackage{environ}

\usepackage{tikz}\usetikzlibrary{decorations.text}
\providecommand{\circabstractbody}{}
\RenewEnviron{abstract}{\global\let\circabstractbody\BODY}

\ExplSyntaxOn
\fp_new:N \l_disc_r_fp   \fp_new:N \l_disc_bls_fp
\fp_new:N \l_disc_y_fp   \fp_new:N \l_disc_hw_fp
\int_new:N \l_disc_n_int \tl_new:N \g_disc_ps_tl
\cs_new_protected:Npn \circledisc #1 {
  \fp_set:Nn \l_disc_r_fp   { \dim_to_fp:n {#1} }
  \fp_set:Nn \l_disc_bls_fp { \dim_to_fp:n {\baselineskip} }
  \int_set:Nn \l_disc_n_int { \fp_to_int:n { floor( 2*\l_disc_r_fp / \l_disc_bls_fp ) } }
  \tl_gclear:N \g_disc_ps_tl
  \int_step_inline:nnnn {1}{1}{\l_disc_n_int} {
    \fp_set:Nn \l_disc_y_fp  { \l_disc_r_fp - (##1 - 0.5)*\l_disc_bls_fp }
    \fp_set:Nn \l_disc_hw_fp { sqrt( max(0, \l_disc_r_fp*\l_disc_r_fp - \l_disc_y_fp*\l_disc_y_fp) ) }
    \tl_gput_right:Nx \g_disc_ps_tl
      { \fp_eval:n{ \l_disc_r_fp - \l_disc_hw_fp }pt~ \fp_eval:n{ 2*\l_disc_hw_fp }pt~ }
  }
  \parshape \int_use:N \l_disc_n_int\space \g_disc_ps_tl
  \noindent
}
\ExplSyntaxOff
\newlength{\rtext}

\providecommand{\circabstractbody}{}
\RenewEnviron{abstract}{\global\let\circabstractbody\BODY}

\renewcommand{\maketitle}{%
  \begin{center}
  \begin{tikzpicture}
    \def\R{6cm}
    \def\RR{6.4cm}
    \path[decorate,decoration={text along path, text align=fit to path,
        text={|\large\bfseries|Noncrossing Combinatorics, the Full Twist, and Decategorification of Knot Invariants}}]
      (180:\R) arc (180:0:\R);
    \path[decorate,decoration={text along path, text align=fit to path,
        text={|\scshape|Colin Defant and Nathan Williams}}]
      (225:\R) arc (225:315:\R);
    \path[decorate,decoration={text along path, text align=fit to path,
        text={|\footnotesize\itshape|Harvard University and The University of Texas at Dallas}}]
      (225:\RR) arc (225:315:\RR);
    \node[align=justify, text width=8.6cm, font=\footnotesize] at (-0.4,.4){\parbox{2\rtext}{\footnotesize
        \tolerance=9999 \emergencystretch=2em
        \circledisc{\rtext}{\scshape Abstract.}\enspace\circabstractbody\par}};
  \end{tikzpicture}
  \end{center}
  \vspace{1.5em}
} 

\hypersetup{
  pdftitle={Noncrossing Combinatorics, the Full Twist, and Decategorification of Knot Invariants},
  pdfauthor={Colin Defant and Nathan Williams}
}

\begin{document}

\begin{abstract}
Much work in knot theory has consisted of categorifying, and thereby strengthening, knot invariants.  We take the opposite approach: decategorification, more commonly called combinatorics.
We introduce a technique that relates the dual braid group generators, the Hecke images of pure braids, and factorization problems in reflection groups to knot invariants.  We prove that the $(a,z{=}0)$-HOMFLYPT polynomial can be computed as a solution to such a problem.
This technique was motivated by Coxeter--Catalan combinatorics.  For example, we give a new proof of EL-shellability of the noncrossing partition lattice using the image of the full twist in the Hecke algebra; in a surprising sort of combinatorial reciprocity, its inverse computes the homotopy type.  Similarly, noncrossing partitions naturally arise from our construction applied to \emph{positive} powers of the full twist, while cluster complexes come from the same construction applied to \emph{negative} powers.  In crystallographic type, we exploit the conjugacy of all Coxeter elements to give the first reflection subword models for rational noncrossing Catalan objects.  Our reciprocity gives two models: one generalizing noncrossing partitions, and one generalizing clusters.  Applying the same method produces two (rational) noncrossing parking {\hphantom{444444}models}. 
\end{abstract}
\maketitle

\section{Introduction}

Let $(W,S)$ be a finite irreducible Coxeter system of rank $r$ with Coxeter number $h$. Fix a standard Coxeter element $c$ (a product of the simple reflections in $S$ in some order) and a reduced $S$-word ${\sf c}$ for $c$. Let $\P_W$, $\B_W$, and $\H_W$ be the pure braid group, the braid group, and the Hecke algebra (over $\mathbb{Z}[q^{\pm 1}]$) of type~$W$, respectively. Let $T=T_W$ be the set of reflections of $W$, and let $N=|T|$. For more details, see~\Cref{sec:background}.

\subsection{Braid Groups and Hecke Algebras}
As we understand it, the central theme of Garside theory is that computations in the infinite group $\B_W$ rigidly reduce to rather manageable computations in the finite group $W$.  In \emph{standard} Garside theory~\cite{garside1969braid,deligne1972immeubles,dehornoy2015foundations}, one uses a set $\S$ of lifts of the simple reflections and the long element $w_\circ$ (opposite of the base chamber labeling the identity) to successively strip away factors lying in $W$ from a braid in $\B_W$; in \emph{dual} Garside theory for $\B_W$~\cite{birman1998new,brady2000artin,brady2001partial,brady2002k,bessis2003dual,bessis2015finite}, one uses a set $\R$ of lifts of all reflections and the standard Coxeter element $c$ to strip away factors lying in the noncrossing partition lattice $\NC(W,c)$. 
In either case, the moral is that even though the braid group $\B_W$ is {\bf large}, its multiplication is relatively {\bf simple}.

\medskip

This story is reversed in the Hecke algebra $\H_W$---the quotient of the group algebra $\mathbb{Z}[q^{\pm 1}][\B_W]$ by the relations $\s^2=(q-1)\s+q$ for all $s \in S$, where $\s$ denotes the Artin lift of $s$.  In contrast to $\B_W$, the dimension of $\H_W$ is {\bf small} (namely, $|W|$), but its multiplication is {\bf complicated}---all standard Garside factors of a braid in $\B_W$ are crushed together by the quotient into a deformation of the group algebra of $W$.  This makes $\H_W$ an excellent source of rich invariants for $\B_W$.\footnote{It is maybe even reasonable to conjecture that the natural projection $\mathbb Z[q^{\pm 1}][\B_W]\to\H_W$ is injective on $\B_W$ (decategorifying~\cite{rouquier2006categorification,jensen20172}).}

Furthermore, the story for $\H_W$ is incomplete, since we do not have a useful \emph{dual} Hecke theory.  While the standard presentation yields a full basis of $\H_W$ (taking the images of the Artin lifts of elements of $W$), the dual presentation only provides lifts for the relatively few noncrossing partitions in $\NC(W,c)$.

\medskip

In this paper, we show that some of the \emph{dual} structural rigidity of $\B_{W}$ survives as a leading term in the quotient to $\H_W$ for noncrossing partitions.  We divide this paper into two parts: 
\begin{itemize}
\item \Cref{part:technique} deals with the background and general framework. 
\item \Cref{part:applications} gives applications to algebraic combinatorics and knot theory. 
\end{itemize}

\subsection{\texorpdfstring{\Cref{part:technique}}{Part I}. Technique}
Let $\{\T_w\}_{w\in W}$ be the standard basis of the Hecke algebra $\H_W$. Given a braid $\bbeta \in \B_W$, we consider its Hecke image $\HT{\bbeta} \in \H_W$. For $\mathrm{x}\in\H_W$ and $v\in W$, let $[\T_v]\mathrm{x}$ denote the coefficient of $\T_v$ in the expansion of $\mathrm{x}$ in the standard Hecke algebra basis. 

\Cref{part:technique} is devoted to using noncrossing combinatorics to compute $[\T_\pi]\H(\brho)$, where $\brho\in\P_W$ is a pure braid and $\pi$ is a noncrossing partition.  More precisely,~\Cref{thm:main} proves that a leading term of noncrossing coefficients of the image of a pure braid in the Hecke algebra can still be computed using the dual Garside theory of $W$. 

\subsubsection{Rewriting pure braids and the full twist}
For $t\in T$, let us write $\t$ for the $c$-dual lift of $t$ to $\B_W$. The set $\R:=\{\t : t \in T\}$ forms a generating set for $\B_{W}$ in Bessis's dual presentation.  Let $\tt=\t^2$. The set $\PR:=\{\tt: t \in T\}$ generates the pure braid group $\P_W$ (see~\Cref{sec:noncrossing_shards} and~\cite{salvetti1987topology,bessis2003dual,defant2022pop}).

We begin by writing $\brho$ as a product of the generators in $\PR$.  Fixing any spanning tree of weak order and the Artin lifts of $w \in W$ as the coset representatives of $\B_W/\P_W$, the Reidemeister--Schreier rewriting process for $\P_W \subset \B_W$ generically results in some horrifically unreduced word for the pure braid $\brho$ in a large set of generators (indexed by shards, as described in~\Cref{sec:shards} and~\cite{defant2022pop}).
For \emph{some} pure braids, however, we can do much better (it is sometimes best to just guess the correct answer, as in~\Cref{sec:rational}).  Let $\B(w)\in\B_W$ denote the Artin lift of an element $w\in W$. In~\Cref{sec:pure}, we use the inversion set of $w$ to give elegantly simple reduced factorizations for pure braids of the form \[\B(w)\B(w^{-1}),\] uniformly reproving some results on Mikado braids along the way.  The factorizations we produce in \Cref{thm:BB=Z} will provide a useful tool for our applications in~\Cref{part:applications}.   In the very special case when $w=w_\circ$ is the long element, we recover the lovely fact that the \defn{full twist} $\Delta$\footnote{This also immediately follows from the commutation equivalence of the \emph{words} $\mathsf{c}^h$ and ${\sf w}_\circ(c){\sf w}_\circ(w_\circ cw_\circ)$.  See~\cite{stump2025cataland,defant2022pop}.} can be factored as a product (in reverse order) over the inversion sequence of ${\sf w}_\circ(c)$, the $c$-sorting word for $w_\circ$:
\begin{equation}\label{eq:full_twist}
\Delta^2:=\B(w_\circ)\B(w_\circ)=\prod_{t \in \inv({\sf w}_\circ(c))}^{\longleftarrow} \tt.
\end{equation}

\subsubsection{Hecke images of pure braids}
Having rewritten a pure braid $\brho$ as a product of elements $\tt$, we find that the defining relations for $\H_W$ give the image of $\brho$ in the Hecke algebra as \begin{equation}\label{eq:expand}\HT{\brho}=\prod \HT{\tt}^{\pm 1}=\prod \HT{\t}^{\pm 2}=\prod \begin{cases} (q-1) \HT{\t}+q & \text{ positive exponent} \\ -q^{-2}(q-1)\HT{\t} + (1-q^{-1}+q^{-2})& \text{ negative exponent.}\end{cases}\end{equation}
Expanding this product looks very much like:
\begin{itemize}
\item[($+$)] ``for each factor $\tt^{+1}$, choose either $(q-1)\HT{\t}$ or $q$'' and
\item[($-$)] ``for each factor $\tt^{-1}$ choose either $-q^{-2}(q-1)\HT{\t}$ or $(1-q^{-1}+q^{-2})$.''
\end{itemize}
The image of $\brho$ in $\H_W$ can also be expanded in the standard basis $\{\T_w\}_{w\in W}$ of the Hecke algebra, and we can try to isolate the coefficient $[\T_w]\H(\brho)$ as a (signed) combinatorial object encoding the expansion of each term in the product from~\eqref{eq:expand} ``either choose the term with $\HT{\t}$, or not.''  The trouble is that this expansion becomes difficult to track---the $\HT{\t}$ terms themselves expand into sums over various standard Hecke basis elements $\T_v$, and as more and more $\HT{\t}$ factors pile up, we get cascading collapses among these standard Hecke basis elements.

\tikzset{
  ed/.style={draw=gray!65, line width=0.9pt},
  vtx/.style={inner sep=2pt, fill=white, font=\small},
  coef/.style={font=\footnotesize, inner sep=2pt, fill=white, fill opacity=0.9, text opacity=1, rounded corners=2pt},
  num/.style={font=\large, inner sep=2pt},
}

\def\hexcoefstyle{coef}
\newcommand{\weakhex}[7]{%
  \begin{scope}[shift={#1}]
  \def\R{2.7}
  \coordinate (e) at (270:\R); \coordinate (o1) at (210:\R);
  \coordinate (o2) at (330:\R); \coordinate (o12) at (150:\R);
  \coordinate (o21) at (30:\R); \coordinate (o121) at (90:\R);
  \draw[ed] (e)--(o1) (e)--(o2) (o1)--(o12) (o2)--(o21) (o12)--(o121) (o21)--(o121);
  \node[vtx] (Ve) at (e) {$\T_e$};      \node[vtx] (V1) at (o1) {$\T_1$};
  \node[vtx] (V2) at (o2) {$\T_2$};     \node[vtx] (V12) at (o12) {$\T_{12}$};
  \node[vtx] (V21) at (o21) {$\T_{21}$};\node[vtx] (V121) at (o121) {$\T_{121}$};
  \node[\hexcoefstyle, below=4pt of Ve]   {$#2$};
  \node[\hexcoefstyle, left=4pt of V1]    {$#3$};
  \node[\hexcoefstyle, right=4pt of V2]   {$#4$};
  \node[\hexcoefstyle, left=4pt of V12]   {$#5$};
  \node[\hexcoefstyle, right=4pt of V21]  {$#6$};
  \node[\hexcoefstyle, above=4pt of V121] {$#7$};
  \end{scope}}

\newcommand{\absorder}[7]{%
  \begin{scope}[shift={#1}]
  \def\X{1.8}
  \def\Y{2.0}
  
  \coordinate (e) at (0,-\Y);
  
  \coordinate (o1)   at (-2.5*\X,0);
  \coordinate (o121) at (0,0);
  \coordinate (o2)   at (2.5*\X,0);
  
  \coordinate (o12) at (-1.5*\X,\Y);
  \coordinate (o21) at (1.5*\X,\Y);
  
  \draw[ed] (e)--(o1) (e)--(o121) (e)--(o2);
  \draw[ed] (o1)--(o12) (o1)--(o21);
  \draw[ed] (o121)--(o12) (o121)--(o21);
  \draw[ed] (o2)--(o12) (o2)--(o21);
  
  \node[vtx] (Ve)   at (e)    {$\T_e$};      
  \node[vtx] (V1)   at (o1)   {$\T_1$};
  \node[vtx] (V121) at (o121) {$\T_{121}$};
  \node[vtx] (V2)   at (o2)   {$\T_2$};     
  \node[vtx] (V12)  at (o12)  {$\T_{12}$};
  \node[vtx] (V21)  at (o21)  {$\T_{21}$};
  
  \node[\hexcoefstyle, below=4pt of Ve]       {$#2$};
  \node[\hexcoefstyle, below=4pt of V1]        {$#3$};
  \node[\hexcoefstyle, below=4pt of V2]       {$#4$};
  \node[\hexcoefstyle, above=4pt of V12] {$#5$};
  \node[\hexcoefstyle, above=4pt of V21]{$#6$};
  \node[\hexcoefstyle, above=4pt of V121]     {$#7$};
  \end{scope}}

\begin{example}
We illustrate this collapse with a small example.  Fix the symmetric group $W=\mathfrak{S}_3$ and $c=s_1s_2=(123)$, where $s_i=(i\,\,i+1)$. Then $\B_W=\B_3$ and $\P_W=\P_3$ are the usual $3$-strand braid group and pure braid group, respectively.  Write \[\B_3=\Big\langle \s_1,\s_2 : \s_1\s_2\s_1=\s_2\s_1\s_2\Big\rangle = \Big\langle \t_{(12)},\t_{(13)},\t_{(23)} : \t_{(12)}\t_{(23)}=\t_{(13)}\t_{(12)}=\t_{(23)}\t_{(13)} \Big\rangle\] for the standard (Artin) and dual (Birman--Ko--Lee) presentations of the 3-strand braid group.  Identifying $\s_1$ with $\t_{(12)}$ and $\s_2$ with $\t_{(23)}$, we can express the usual Artin generators of $\P_3$ as  
\[\tt_{(12)} = \t_{(12)}^2=\s_1^2,\quad \tt_{(13)} =\t_{(13)}^2= \s_1\s_2^2 \s_1^{-1},\quad \tt_{(23)} = \t_{(23)}^2= \s_2^2,\]
so that \[\P_3 = \Big\langle \tt_{(12)},\tt_{(13)},\tt_{(23)} : \tt_{(23)}\tt_{(13)}\tt_{(12)}=\tt_{(13)}\tt_{(12)}\tt_{(23)}=\tt_{(12)}\tt_{(23)}\tt_{(13)}\Big\rangle.\]
Then in the Hecke algebra $\H_3$, we have $\HT{\t_{(12)}}=\T_1$, $\HT{\t_{(23)}}=\T_2$, and
\[\HT{\tt_{(12)}}=(q-1)\T_1+q, \quad \HT{\tt_{(13)}}= q^{-1}(q-1)\T_{121} -q^{-1}(q-1)^2\T_{12} + q, \quad \HT{\tt_{(23)}}=(q-1)\T_2+q.\]

Now fix the pure braid $\brho=\tt_{(12)}\tt_{(13)}\tt_{(23)} \in \P_3$ (note that $\brho$ is \emph{not} the full twist). 
The expansion of the Hecke image of $\brho$ in the standard basis of $\H_3$ is given in~\Cref{fig:ex3}; we suggest to our human readers that it is not immediate to determine these coefficients without explicitly working out the product. \label{ex:3}
\end{example}

\begin{figure}[htbp]
\begin{tikzpicture}
  \def\hexcoefstyle{coef}
  \absorder{(0,0)}
    {-q(q^2-q+1)(q^2-3q+1)}
    {-(q-1)(q^4-3q^3+3q^2-3q+1)}
    {-(q-1)(q^4-3q^3+3q^2-3q+1)}
    {-q^{-1}(q-1)^4(q^2+1)}
    {(q-1)^2(q^2+1)}
    {q^{-1}(q-1)(q^4-q^3+q^2-q+1)}
\end{tikzpicture}
\caption{For $c=s_1s_2$, the Hecke expansion of $\tt_{(1,2)}\tt_{(1,3)}\tt_{(2,3)}$ into the standard basis $\{\T_w\}_{w \in \mathfrak{S}_3}$, where each $w\in \mathfrak{S}_3$ is represented by the indices of a reduced word in simple reflections. }
\label{fig:ex3}
\end{figure}
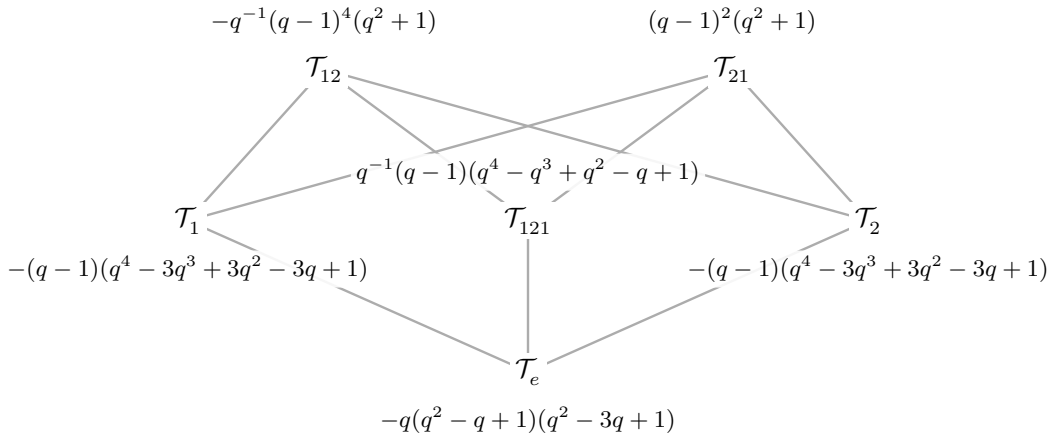

\subsubsection{Hecke coefficients and subwords}
It turns out that the collapse of the product in~\eqref{eq:expand} in the Hecke algebra is not \emph{completely} unmanageable.  Because each $\HT{\t}$ carries with it a factor of $(q-1)$ in the expansion~\eqref{eq:expand}, we maintain some rudimentary control: the leading term of the coefficient of $\T_w$ at $q=1$ is forced to be divisible by at least $(q-1)^{\ell_T(w)}$  (where $\ell_T(w)$ is the reflection length of $w$; see~\Cref{sec:groups}).  In general, however, this order of vanishing can be higher.

\begin{example}\label{ex:4}
We continue with~\Cref{ex:3} with $c=s_1s_2=(123) \in \mathfrak{S}_3$ and the pure braid $\brho=\tt_{(12)}\tt_{(13)}\tt_{(23)} \in \P_3$.  \Cref{fig:ex4} gives the evaluated coefficients $(q-1)^{-\ell_T(w)}\,[\T_w]\H(\brho)$ at $q=1$.  For each element of $\mathfrak{S}_3$---except for $c$ itself---this evaluated coefficient is just the number of subwords of the word $\big((12),(13),(23)\big)$ that are reduced $T$-words for $w$.  That is, reflection subwords perfectly predict the evaluated coefficients for elements of $\NC(\mathfrak{S}_3,c^{-1})$.  For example, the coefficient of $\T_{c^{-1}}=\T_{21}$ is explained by the reflection subwords $\big((12),(13),\textcolor{lightgray}{(23)}\big)$ and $\big(\textcolor{lightgray}{(12)},(13),(23)\big)$.

On the other hand, the coefficient of $\T_{c}$ in $\H(\brho)$ is divisible by $(q-1)^4$, so $(q-1)^{-\ell_T(c)}[\T_c]\H(\brho)=(q-1)^{-2}[\T_c]\H(\brho)$ vanishes upon specialization.  We can understand this vanishing as follows.  We have the reflection subword $\big((12), \textcolor{lightgray}{(13)}, (23)\big)$, which corresponds to the perfectly reasonable subword \[\HT{\tt_{(12)}}\cdot \textcolor{lightgray}{\HT{\tt_{(13)}}} \cdot \HT{\tt_{(23)}} =((q-1)\T_1\textcolor{lightgray}{+q})\cdot (\textcolor{lightgray}{(q-1)\T_{\t_{13}}}+q)\cdot ((q-1)\T_2\textcolor{lightgray}{+q}) = q(q-1)^2\T_{12}\textcolor{lightgray}{+\cdots}\] (coming from choosing the factors of $(q-1)\T_1$ and $(q-1)\T_2$ from $\HT{\tt_{(12)}}$ and $\HT{\tt_{(23)}}$ and from choosing the factor $q$ in the expansion of $\HT{\tt_{(13)}}$). 
The issue is that $\HT{\tt_{13}}$ has a nonzero coefficient for $\T_{12}$.  Choosing this coefficient of $\T_{12}$ along with the $q$ from the expansions of $\HT{\tt_{(12)}}$ and $\HT{\tt_{(23)}}$, \[(\textcolor{lightgray}{(q-1)\T_1}+q)\cdot (-q^{-1}(q-1)^2\T_{12}\textcolor{lightgray}{+q^{-1}(q-1)\T_{121} + q})\cdot (\textcolor{lightgray}{(q-1)\T_2}+q)=-q(q-1)^2\T_{12}\textcolor{lightgray}{+\cdots},\] perfectly cancels with the previously-found subword leaving only other terms that vanish to high order.\end{example}

\begin{figure}[htbp]
\begin{tikzpicture}
  \def\hexcoefstyle{num}
  \absorder{(0,0)}
    {1}{1}{1}
    {0}
    {2}{1}
\end{tikzpicture}
\caption{The Hecke expansion of $\brho=\tt_{(12)}\tt_{(13)}\tt_{(23)}$, where each coefficient $[\T_w]\H({\brho})$ has been evaluated to $(q-1)^{-\ell_T(w)}\,[\T_w]\H(\brho)\big|_{q=1}$.  For every element $w \in \NC(W,c^{-1})$---that is, every element of $\mathfrak{S}_3$ except for $c=s_1s_2$---the coefficient of $T_w$ is just the number of subwords of the word $\vect=\big((12),(13),(23)\big)$ that are reduced $T$-words for~$w$.}
\label{fig:ex4}
\end{figure}
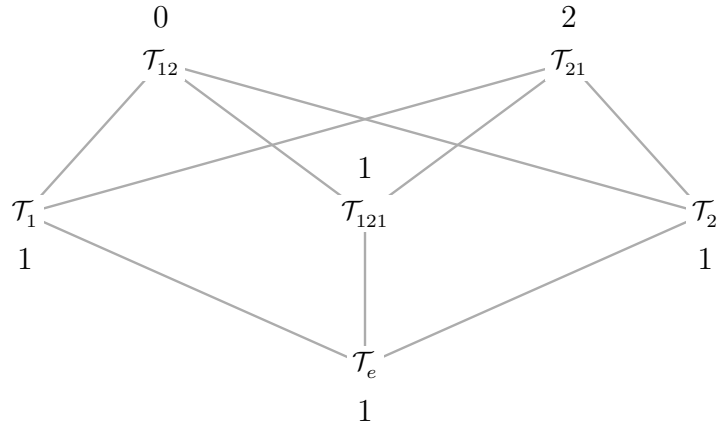

Recall that we wrote our braid $\brho$ as $\tt_1\cdots\tt_M$, where $\vect=(t_1,\ldots,t_M) \in T^M$.  Ignoring the difference between $c$ and $c^{-1}$ for a moment, if we were still in the group algebra of the dual braid group $\mathbb{Z}[q^{\pm 1}][\B_{W}]$ instead of the Hecke algebra (one could also imagine working in $\H_W$ using a non-standard basis coming from \emph{some} extension of the image of $\NC(W,c)$), then the analogue of~\eqref{eq:expand} would be the product \[\prod_{i=1}^M \begin{cases} (q-1) \t_i+q & \text{ positive exponent } \\ -q^{-2}(q-1)\t_i + (1-q^{-1}+q^{-2})& \text{ negative exponent }.\end{cases}\]
Using dual Garside theory to allow ourselves to compute in absolute order and the noncrossing partition lattice $\NC(W,c)$ instead of the dual braid group $\B_{W}$, the coefficient of the Garside element $\mathbf{c}=\B(c)$ at $q=1$ (after dividing by $(q-1)^{r}$) would quite simply just be the signed sum over subwords of $\vect$ that are reduced $T$-words for $c$ (and similarly for any $\pi \in \NC(W,c)$). 

Remarkably---and now paying attention to the difference between $c$ and $c^{-1}$---for any noncrossing partition $\pi \in \NC(W,c^{-1})$, this leading term at $q=1$ actually survives the quotient from $\B_{W}$ to $\H_W$. Suppose $\protect\vv{t}=(t_1,\ldots,t_M)\in T^M$, $\protect\vv\epsilon=(\epsilon_1,\ldots,\epsilon_M)\in\{\pm 1\}^M$, and $\pi\in\NC(W,c^{-1})$. Our main result (\Cref{thm:main}) concerns the set $\Sub_T(\vect,\pi)$ of subwords of $\vect$ that are reduced $T$-words for $\pi$. Let \[\Subb_T(\vect,\pi)=|\Sub_T(\vect,\pi)|.\] If we choose a sign vector $\protect\vv\epsilon=(\epsilon_1,\ldots,\epsilon_M)\in\{\pm 1\}^M$, then we can assign each subword $\sf w$ of $\vect$ the sign $\mathrm{sgn}^{\protect\vv\epsilon}(\mathsf w)$ obtained by multiplying the signs corresponding to the indices of the reflections used in ${\sf w}$. We then consider 
\begin{equation}\label{eq:sub}\Subb_\RR^{\protect\vv\epsilon}(\protect\vv{t},\pi):=\sum_{{\sf w}\in\Sub_T(\vect,\pi)}\mathrm{sgn}^{\protect\vv\epsilon}({\sf w})=\sum_{\substack{1 \leq i_1 < i_2 < \cdots < i_{\ell_T(\pi)} \leq M \\ t_{i_1}t_{i_2}\cdots t_{i_{\ell_T(\pi)}}=\pi}} \prod_{j=1}^{\ell_T(\pi)} \epsilon_{i_j},\end{equation}
the signed count over all subwords of $\vect$ that are reduced $T$-words for $\pi$. The following result is new even when $\protect\vv\epsilon=(1,\ldots,1)$, in which case $\Subb_T^{\protect\vv\epsilon}(\vect,\pi)=\Subb_T(\vect,\pi)$.

{
\renewcommand{\thetheorem}{\ref{thm:main}}
\begin{theorem}
Fix a sequence $\vect=(t_1,\ldots,t_M)\in T^M$ of reflections and the corresponding sequence ${\vectt=(\tt_1,\ldots,\tt_M)\in\mathbb T^M}$ of squares of $c$-dual lifts. Fix a sign vector $\vece=(\epsilon_1,\ldots,\epsilon_M)\in\{\pm 1\}^M$.
For each $\pi\in\NC(W,c^{-1})$, we have 
\[(q-1)^{-\ell_{\RR}(\pi)}[\T_\pi]\H(\tt_1^{\epsilon_1}\cdots \tt_M^{\epsilon_M})\Big|_{q=1}=\Subb_\RR^{\protect\vv\epsilon}(\protect\vv{t},\pi).\] 
\end{theorem}
\addtocounter{theorem}{-1}
}

\begin{example}
\Cref{thm:main} actually explains the vanishing of the evaluated coefficient of $\T_c$ in the Hecke algebra image of the pure braid $\brho$ from~\Cref{ex:3,ex:4} in a combinatorial way, since $c$ is still a noncrossing partition for the inverse of the standard Coxeter element $c^{-1}$.  In more detail, if we use the dual braid generators for $c^{-1}=s_2s_1$, then we obtain the dual presentation
\[\B_{3}=\Big\langle \u_{(12)},\u_{(13)},\u_{(23)} : \u_{(23)}\u_{(12)}=\u_{(12)}\u_{(13)}=\u_{(13)}\u_{(23)} \Big\rangle,\]
which, after identifying $\u_{(12)}=\s_1$ and $\u_{(23)}=\s_2$, gives us the dual pure braid group generators \[\uu_{(12)} = \s_1^2,\quad \uu_{(13)} = \s_2\s_1^2 \s_2^{-1},\quad  \uu_{(23)} = \s_2^2.\]
Since $\tt_{(13)}=\uu_{(12)} \uu_{(13)} \uu_{(12)}^{-1}$, the pure braid $\brho$ from~\Cref{ex:3,ex:4} can be expressed in these new generators as
\[\brho = \uu_{(12)}\uu_{(12)} \uu_{(13)} \uu_{(12)}^{-1} \uu_{(23)}.\]

If we now count \emph{signed} reduced reflection subwords for $c$ (which now is a noncrossing partition in $\NC(\mathfrak{S}_3,c)$) in the corresponding word $\vect=\big((12),(12),(13),(12),(23)\big)$ with signs $\vece=(1,1,1,-1,1)$, we obtain the two positive subwords \[\big((12),\textcolor{lightgray}{(12)},\textcolor{lightgray}{(13)},\textcolor{lightgray}{(12)},(23)\big) \quad \text{and}\quad \big(\textcolor{lightgray}{(12)},(12),\textcolor{lightgray}{(13)},\textcolor{lightgray}{(12)},(23)\big),\] which perfectly cancel with the two negative subwords \[\big(\textcolor{lightgray}{(12)},\textcolor{lightgray}{(12)},(13),(12),\textcolor{lightgray}{(23)}\big) \quad \text{and}\quad \big(\textcolor{lightgray}{(12)},\textcolor{lightgray}{(12)},\textcolor{lightgray}{(13)},(12),(23)\big),\] combinatorially explaining the vanishing of the coefficient of $\T_{12}$ in~\Cref{fig:ex4}. 
\end{example}

\subsection{\texorpdfstring{\Cref{part:applications}}{Part II}. Applications}

\Cref{part:applications} explores applications of~\Cref{thm:main}, which evaluates the coefficients of the Hecke images of pure braids at $q=1$ using noncrossing combinatorics.
A central feature of these applications is a sort of \emph{combinatorial reciprocity}~\cite{beck2018combinatorial}, arising from the identity\footnote{In the interest of keeping our dual generators constant for a fixed standard Coxeter element $c$ (so that we can take the coefficient of $\T_{c^{-1}}$ in the Hecke algebra), we will often use $\B(c^{-1})$ in the place of $\B(c)$ in the braid $\bbeta$.}
\begin{align}\tag{\Cref{lem:coeff-delta}}\label{eq:reciprocity}
\underbrace{[\T_e]\H(\Delta^2 \bbeta)}_{\text{positive models}} &= \underbrace{(-1)^{\mathrm{wr}(\bbeta)} q^{N+\mathrm{wr}(\bbeta)} [\T_e]\H(\bbeta^{-1})}_{\text{negative models}} \text{ for any } \bbeta \in \B_W.\end{align}

Given a braid $\bbeta \in \B_W$ whose image in $W$ is a noncrossing partition $\pi \in \NC(W,c)$, we can write $\bbeta=\brho\,\B(\pi)$, where $\brho \in \P_W$. Combining~\Cref{thm:main} with the identity of coefficients
\begin{align}\tag{\Cref{lem:coeff-identity}}[\T_e]\H(\brho\, \B(\pi)) = q^{\ell(\pi)} [\T_{\pi^{-1}}]\H(\brho),\end{align} we find that the two sides of~\Cref{lem:coeff-delta} give two combinatorial reflection subword models---a positive model (coming from the left side of~\Cref{lem:coeff-delta}, after factoring $\brho$ into squares of the dual generators), and a negative model (coming from the right)---for the price of analyzing the one pure braid $\brho$.   

\Cref{sec:fuss} concerns a particularly amusing example of this reciprocity for the braid $\bbeta=\Delta^{2m} \B(c)$ (for which the corresponding pure braid is $\brho=\Delta^{2m}$)---the positive model is interpreted as $m$-noncrossing partitions (\Cref{eq:intropos}), while the negative model gives $m$-clusters (\Cref{eq:introneg}). More generally, \Cref{prop:Catmh+1,prop:Clusters_m}
show that the canonical $c$-sorting-word factorizations may be replaced
by arbitrary suitable factorizations of $\Delta^{2(m+1)}$ and
$\Delta^{2m}$ into $c$-dual pure generators.

\subsubsection{Noncrossing partition lattices}  In~\Cref{sec:noncrossing_lattice}, we apply our technique to study the $c^{-1}$-noncrossing partition lattice. 
\Cref{thm:EL} proves that \emph{any} reduced $\PR$-word for the full twist (for example, \eqref{eq:full_twist} provides one such word) gives an EL-shelling of $\NC(W,c^{-1})$. 

{
\renewcommand{\thetheorem}{\ref{thm:EL}}
\begin{theorem}
Let $\vect=(t_1,\ldots,t_N)$ be a total order on $T$. If $\tt_1\cdots\tt_N=\Delta^2$, then $(t_1,\ldots,t_N)$ is an EL-shelling order of $\NC(W,c^{-1})$. 
\end{theorem}
\addtocounter{theorem}{-1}
}

The preceding result was first proven by C.~Athanasiadis, T.~Brady, and C.~Watt in the special case when the resulting permutation of reflections is commutation equivalent to the $c$-sorting word for $w_\circ$~\cite{ABW}.  Our new, uniform proof shows that the unique increasing chain condition of EL-shelling is equivalent to the condition that the coefficient of a noncrossing partition in the expansion of the full twist evaluates to 1: \[(q-1)^{-\ell_T(\pi)} [\T_\pi] \H(\Delta^2)\Big|_{q=1} = 1 \text{ for } \pi \in \NC(W,c^{-1}).\]  This coefficient is easily evaluated---the case when $\pi=c^{-1}$ was already highlighted in V.~Deodhar's 1985 paper~\cite[§4]{deodhar1985} as a particularly simple example of a Richardson variety.

In~\Cref{thm:homotopy_type}, we similarly reduce the computation of the homotopy type of the noncrossing partition lattice to the coefficient
\[(q-1)^{-r}[\T_{e}] \H(\Delta^{-2}\B(c))\Big|_{q=1}. 
\]  Namely, we show that this evaluated coefficient equals the M\"obius function value $\mu_{\NC(W,c^{-1})}(e,c^{-1})$. 
More generally, we show that the braid $\Delta^2\B(c^{-1})$ can be seen as a $q$-analogue of the zeta function of absolute order. 
{
\renewcommand{\thetheorem}{\ref{thm:zeta}}
\begin{theorem}  For $m \in \mathbb{Z}$,
\[(q-1)^{\ell_T(w)-r}[\T_w] \H(\Delta^{2m}\B(c^{-1})) \Big|_{q=1} = \begin{cases}\zeta^{(m)}_{\NC(W,c^{-1})}(w,c^{-1}) &\text{ if } w \in \NC(W,c^{-1}) \\ 0 &\text{otherwise}.\end{cases}\] 
\end{theorem}
\addtocounter{theorem}{-1}
}

\subsubsection{Clusters and noncrossing partitions}
In~\Cref{sec:fuss}, we use our framework to explain the existence of reflection subword models for clusters and noncrossing partitions.  The \emph{reason} these models are so simple is that the braid $\B(c)^{mh\pm1} = \Delta^{2m} \B(c)^{\pm 1}$ is very close to a power of the full twist, for which we have the simple factorization from~\eqref{eq:full_twist}.   The positive model from~\Cref{lem:coeff-delta} applied to the trace of the braid $\bbeta=\B(c)^{mh+1}$ gives $m$-noncrossing partitions, while the negative model gives $m$-clusters~\cite{stump2025cataland}.   

In more detail, given an integer $k$, let $[k]_q=\frac{1-q^k}{1-q}$.  Recall that the \defn{Fuss--Catalan numbers} for a finite Coxeter group $W$ of rank $r$ are defined for positive integers $m$ by \[\Cat_{mh+ 1}(W;q):=\prod_{i=1}^r \frac{[mh+ 1+e_i]_q}{[d_i]_q} \quad\text{and}\quad\Cat_{mh+ 1}(W):=\Cat_{mh+ 1}(W;1),
\]
where $d_1,\ldots,d_r$ are the degrees and $e_1,\ldots,e_r$ are the exponents. Similarly, the \defn{Fuss--Dogolon numbers} are defined by \[\Cat_{mh- 1}(W;q):=\prod_{i=1}^r \frac{[mh- 1+e_i]_q}{[d_i]_q} \quad\text{and}\quad\Cat_{mh- 1}(W):=\Cat_{mh- 1}(W;1).
\]
The positive subword model from~\Cref{lem:coeff-delta} for $\bbeta=\B(c)^{mh+1}$ gives the $m$-noncrossing partition subword model from~\cite{stump2025cataland} by applying~\cref{thm:main} to the computation
\begin{equation}\label{eq:intropos}[\T_e]\H(\Delta^2 \B(c)^{mh+1})=q^r[\T_{c^{-1}}]\H((\Delta^2)^{m+1})=q^N (q-1)^r \Cat_{mh+1}(W;q)\end{equation}
from the trace evaluations in~\cite{galashin2024rational}, and factoring $\Delta^2=\B(c)^h$. As in~\eqref{eq:full_twist}, one suitable sequence of reflections we can use for $(\Delta^2)^{m+1}$ is $\inv({\sf w}_\circ(c^{-1}))^{m+1}$. 

The negative model from~\Cref{lem:coeff-delta} for $\bbeta=\B(c^{-1})^{mh+1}$ gives the $m$-cluster subword model from~\cite{stump2025cataland} by applying \Cref{thm:main} to the computation

\begin{align}\label{eq:introneg}\nonumber[\T_e]\H(\B(c^{-1})^{-(mh+1)})&=q^r[\T_{c^{-1}}]\H((\Delta^2)^{-m}(\B(c)\B(c^{-1}))^{-1}))\\&=(-1)^rq^{-r(mh+1)}(q-1)^r \Cat_{mh+1}(W;q),\end{align}
where we note that $\B(c)\B(c^{-1}) = \tt_r\cdots\tt_1$ if $\inv(c)=(t_1,\ldots,t_r)$.

Nearly identical computations for the braid $\bbeta=\B(c^{-1})^{mh-1}$ give the Fuss--Dogolon full-support models from~\cite{stump2025cataland}. 

\subsubsection{Rational models}
The \defn{rational Catalan numbers} of type $W$ are defined for a positive integer $p$ coprime to the Coxeter number $h$ as \[\Cat_p(W;q):=\prod_{i=1}^r \frac{[p+(p e_i \!\!\!\mod h)]_q}{[d_i]_q} \quad\text{and}\quad\Cat_p(W):=\Cat_p(W;1),\]
where $p e_i \!\!\!\mod h$ is the least nonnegative residue of $pe_i$ modulo $h$. 
The trace formulas
\[[\T_e] \H(\Delta^2 \B(c)^{p}) = q^{N}(q-1)^r \Cat_p(W;q) \quad\text{and}\quad [\T_e] \H(\B(c)^{-p}) = (-1)^rq^{-pr}(q-1)^r \Cat_p(W;q)\]
were proven in \cite{galashin2024rational}. 
To produce the long-sought-after (see~\cite[Section 8]{stump2025cataland}) subword models for rational Catalan numbers, we are led to study the braids $\B(c)^{\pm p}$ (generalizing the Fuss--Catalan and Fuss--Dogolon cases, where $p\equiv 1 \pmod h$ or $p\equiv -1\pmod h$).  If we naively mimic the computations in~\eqref{eq:intropos} and \eqref{eq:introneg} to get the coefficient \[[\T_e] \H(\Delta^2 \B(c)^p) = q^{\ell(c^p)} [\T_{c^{-p}}] \H(\Delta^2 \B(c)^p \cdot \B(c^p)^{-1})\] or \[[\T_e] \H(\B(c^{-1})^{-p}) = q^{\ell(c^{p})} [\T_{c^{-p}}] \H(\B(c^{-1})^{-p} \cdot \B(c^{p})^{-1}),\] then we can expand the pure braids $\B(c)^p \cdot \B(c^p)^{-1}$ and $\B(c^{-1})^{-p} \cdot \B(c^{p})^{-1}$ as products of elements of $\mathbb T$. However, we run into the serious problem that $c^{-p}$ is not generally a \emph{standard} Coxeter element (nor a noncrossing partition for a standard Coxeter element, since $\ell_T(c^{-p})=r$), which means that we cannot directly use \Cref{thm:main} to compute the coefficient of $\T_{c^{-p}}$.

In crystallographic type, however, both $c^p$ and $c^{-p}$ are conjugate to
the standard Coxeter element $c$. Choose elements $g_+,g_-\in W$ such that $g_+c^pg_+^{-1}=g_-c^{-p}g_-^{-1}=c$,
and choose braid lifts $\mathbf g_+,\mathbf g_-\in\B_W$. Define the pure
braids
\[
\mathbbm c_{p,\mathbf g_+}
:=
\mathbf g_+\B(c)^p\mathbf g_+^{-1}\B(c)^{-1}
\quad\text{and}\quad
\mathbbm c_{-p,\mathbf g_-}
:=
\mathbf g_-\B(c)^{-p}\mathbf g_-^{-1}\B(c)^{-1}.
\]
The trace is invariant under conjugation, and $\Delta^2$ is central. Hence,
by \Cref{lem:coeff-identity}, the positive-trace model (the analogue of
noncrossing partitions) comes from
\begin{align*}
[\T_e]\H\bigl(\Delta^2\B(c)^p\bigr)
&=
[\T_e]\H\bigl(\Delta^2\mathbf g_+\B(c)^p\mathbf g_+^{-1}\bigr)\\
&=
[\T_e]\H\bigl(\Delta^2\mathbbm c_{p,\mathbf g_+}\B(c)\bigr)\\
&=
q^r[\T_{c^{-1}}]\H\bigl(\Delta^2\mathbbm c_{p,\mathbf g_+}\bigr),
\end{align*}
whereas the negative-trace model (the analogue of clusters) comes from
\begin{align*}
[\T_e]\H\bigl(\B(c)^{-p}\bigr)
&=
[\T_e]\H\bigl(\mathbf g_-\B(c)^{-p}\mathbf g_-^{-1}\bigr)\\
&=
[\T_e]\H\bigl(\mathbbm c_{-p,\mathbf g_-}\B(c)\bigr)\\
&=
q^r[\T_{c^{-1}}]\H\bigl(\mathbbm c_{-p,\mathbf g_-}\bigr).
\end{align*}
Combining these identities with the trace evaluations above gives
\[
(q-1)^{-r}
[\T_{c^{-1}}]\H\bigl(\Delta^2\mathbbm c_{p,\mathbf g_+}\bigr)
\Big|_{q=1}
=
(-1)^r(q-1)^{-r}
[\T_{c^{-1}}]\H\bigl(\mathbbm c_{-p,\mathbf g_-}\bigr)
\Big|_{q=1}
=
\Cat_p(W).
\]
We can then apply
\Cref{thm:main} by expressing the relevant pure braids as products of elements
$\P_c(t)^{\pm1}$ with $t\in\RR$. The exponents in these factorizations are not
guaranteed to all be positive, so a priori this produces signed combinatorial
models. If $\mathbbm c_{p,\mathbf g_+}$ admits a positive factorization, then
its writhe forces that factorization to contain exactly $r(p-1)/2$ pure dual
generators.\footnote{It is well known from nonnesting combinatorics that
$r(p-1)/2$ counts the positive roots at heights $\lceil jh/p\rceil$ for
$j=1,2,\ldots,p-1$. These are not, however, the roots appearing in our
factorizations of $\mathbbm c_{p,\mathbf g_+}$, nor do we see how to modify
them naturally to obtain the factorizations in \cref{sec:rational}.}

We stress that---although our strategy is well motivated and uniform for
crystallographic type---our execution is type-by-type. Furthermore, it does
not apply to noncrystallographic types, since $c^p$ and $c$ are not generally
conjugate (see \cref{sec:noncrystallographic} for a discussion of our attempts to extend the
method to noncrystallographic types). In every crystallographic type except
$E_7$ and $E_8$, we were able to choose $\mathbf g_+$ so that
$\mathbbm c_{p,\mathbf g_+}$ has a positive factorization with $r(p-1)/2$
factors. In type $E_7$ when $p=5$ and in type $E_8$ for the cases not already
covered in \cref{sec:fuss}, our factorizations are signed, although we
strongly suspect that unsigned models can be found with additional effort.
The factorizations needed for the negative-trace model are obtained by
applying the same constructions to the residue of $-p$ modulo $h$.

{
\renewcommand{\thetheorem}{\ref{thm:rational-subword-models}}
\begin{theorem}Let $W$ be a Weyl group of rank $r$, let $c$ be a standard Coxeter
element, and let $p>0$ be coprime to the Coxeter number $h$. Choose
$g_+,g_-\in W$ such that $g_+c^pg_+^{-1}=g_-c^{-p}g_-^{-1}=c$, and choose
a braid lift $\mathbf g_+$ of $g_+$ and a braid lift $\mathbf g_-$ of $g_-$. Define
\[
\mathbbm c_{p,\mathbf g_+}
=
\mathbf g_+\B(c)^p\mathbf g_+^{-1}\B(c)^{-1}
\quad\text{and}\quad
\mathbbm c_{-p,\mathbf g_-}
=
\mathbf g_-\B(c)^{-p}\mathbf g_-^{-1}\B(c)^{-1}.
\]
Suppose we have factorizations 
\[
\mathbbm c_{p,\mathbf g_+}
=
\prod_{i=1}^{M_+}\P_c(t_i)^{\epsilon_i},
\qquad
\mathbbm c_{-p,\mathbf g_-}
=
\prod_{j=1}^{M_-}\P_c(v_j)^{\delta_j},
\qquad
\Delta^2
=
\prod_{k=1}^{N}\P_c(u_k),
\]
where\[
\vect=(t_1,\ldots,t_{M_+})\in T^{M_+},\qquad
\protect\vv{v}=(v_1,\ldots,v_{M_-})\in T^{M_-},\qquad
\protect\vv{u}=(u_1,\ldots,u_N)\in T^N,
\]
\[\protect\vv{\epsilon}=(\epsilon_1,\ldots,\epsilon_{M_+})\in\{\pm 1\}^{M_+},\qquad\protect\vv{\delta}=(\delta_1,\ldots,\delta_{M_-})\in\{\pm 1\}^{M_-}.\]
Let $(1^N,\protect\vv{\epsilon})\in\{\pm 1\}^{M_++N}$ be obtained by concatenating the vector $(1,\ldots,1)\in\{\pm 1\}^N$ with $\protect\vv{\epsilon}$. 
Then
\[
\Subb_T^{(1^N,\protect\vv{\epsilon})}
(\vv{u\vphantom{t}}\vect,c^{-1})
=
\Cat_p(W)
=
(-1)^r
\Subb_T^{\protect\vv\delta}
(\protect\vv{v},c^{-1}). 
\] 
\end{theorem}
}

In \cref{thm:rational-subword-models}, the first model comes from the positive braid
$\Delta^2\B(c)^p$ and uses a factorization of
$\mathbbm c_{p,\mathbf g_+}$, while the second comes from the negative
braid $\B(c)^{-p}$ and uses a factorization of
$\mathbbm c_{-p,\mathbf g_-}$. Explicit choices of factorizations of
$\mathbbm c_{p,\mathbf g_+}$ are given in type~$A$ in
\cref{subsec:Type_A}, in type~$B$ (and hence type~$C$) in \cref{subsec:Type_B}, in
type~$D$ in \cref{subsec:Type_D}, and in the exceptional crystallographic
types in \cref{tab:exceptional}. Applying the same constructions to the residue of
$-p$ modulo $h$ gives the factorizations required for the negative
model. Note that the cases
$p\equiv 1\pmod h$ and $p\equiv-1\pmod h$ are the Fuss--Catalan and
Fuss--Dogolon models of \Cref{sec:fuss-cat,sec:dog}, respectively, and
arbitrary $p$ is reduced to $1\leq p<h$ by
\Cref{prop:p-range}. 

Recorded in~\cite[Section 8]{stump2025cataland} are the results of exhaustively searching for rational Catalan subword models by computer (where the sequence of reflections is taken to be a prefix of ${\sf c}^\infty$).  This met with limited success---the authors proposed some constructions in the classical types, but there were no proofs, there was no understanding of {\emph why} the constructions should work, and the exceptional types were not addressed.  The difficulty of finding a ``correct'' lift ${\bf g}$ of the element $g \in W$ satisfying $g c^p g^{-1}=c$ (so that $\XX$ uses only positive powers of the generators in $\PR$) formalizes and explains the difficulties encountered in~\cite{stump2025cataland}.

\subsubsection{Parking models}
For $p>0$ coprime to $h$, the \defn{rational parking numbers} are defined to be
\[\Park_p(W;q):=[p]_q^r \quad\text{and}\quad\Park_p(W):=p^r.\]
It follows from~\cite{galashin2024rational,trinh2026partial} that 
\begin{align*}[\T_e] \sum_{w \in W} q^{-\ell(w)} \H\!\left(\B(w)\B(w^{-1})\cdot  \B(c)^p\right) &= (q-1)^r\Park_p(W;q);\\
[\T_e] \sum_{w \in W} q^{-\ell(w)} \H\!\left(\B(w)\B(w^{-1}) \cdot \B(c)^{-p}\right) &= (-q)^{-rp}(q-1)^r\Park_p(W;q).
\end{align*}
From this, we are able to easily derive subword models for rational parking objects, generalizing the cluster parking functions from~\cite{douvropoulos2025cluster} and the noncrossing parking functions from~\cite{edelman1980chain,armstrong2015parking,rhoades2014parking}.  For $w\in W$, let $\inv_c(w)$ denote the sequence of inversions of $w$, taken in the order they appear in $\inv({\sf w}_\circ(c))$. For $\protect\vv\epsilon=(\epsilon_1,\ldots,\epsilon_M)\in\{\pm 1\}^M$ and $k\geq 0$, let 
\[
(1^k,\protect\vv\epsilon)
=
(\underbrace{1,\ldots,1}_{k\text{ times}},
\epsilon_1,\ldots,\epsilon_M).
\] 

{
\renewcommand{\thetheorem}{\ref{thm:rat_parking}}
\begin{theorem}
Let $W$ be a Weyl group of rank $r$, and fix a standard Coxeter element $c$. Let
$p>0$ be coprime to $h$. Choose an element $g_+\in W$ such that
$g_+c^pg_+^{-1}=c$,
and choose a braid lift $\mathbf g_+\in\B_W$ of $g_+$. Define the
pure braid
\[
\mathbbm c_{p,\mathbf g_+}
=
\mathbf g_+\B(c)^p\mathbf g_+^{-1}\B(c)^{-1}.
\]
Suppose that $\vect=(t_1,\ldots,t_M)\in T^M$ and
$\protect\vv\epsilon=(\epsilon_1,\ldots,\epsilon_M)\in\{\pm1\}^M$
satisfy
\[
\mathbbm c_{p,\mathbf g_+}
=
\P_c(t_1)^{\epsilon_1}\cdots\P_c(t_M)^{\epsilon_M}.
\]
Then
\[
\sum_{w\in W}
\Subb_T^{(1^{\ell(w)},\protect\vv\epsilon)}
\bigl(\rev(\inv_c(w))\vect,c^{-1}\bigr)
=
\Park_p(W).
\]
\end{theorem}
\addtocounter{theorem}{-1}
} 

{
\renewcommand{\thetheorem}{\ref{thm:rat_parking_negative}}
\begin{theorem}
Let $W$ be a Weyl group of rank $r$, and fix a standard Coxeter element $c$. Let
$p>0$ be coprime to $h$. Choose an element $g_-\in W$ such that
\(
g_-c^{-p}g_-^{-1}=c,
\)
and choose a braid lift $\mathbf g_-\in\B_W$ of $g_-$. Define the
pure braid
\[
\mathbbm c_{-p,\mathbf g_-}
=
\mathbf g_-\B(c)^{-p}\mathbf g_-^{-1}\B(c)^{-1}.
\]
Suppose that $\vect=(t_1,\ldots,t_M)\in T^M$ and
$\protect\vv\epsilon=(\epsilon_1,\ldots,\epsilon_M)\in\{\pm1\}^M$
satisfy
\[
\mathbbm c_{-p,\mathbf g_-}
=
\P_c(t_1)^{\epsilon_1}\cdots\P_c(t_M)^{\epsilon_M}.
\]
Then
\[
(-1)^r
\sum_{w\in W}
\Subb_T^{(1^{\ell(w)},\protect\vv\epsilon)}
\bigl(\rev(\inv_c(w))\vect,c^{-1}\bigr)
=
\Park_p(W).
\]
\end{theorem} 
\addtocounter{theorem}{-1}
}

\subsubsection{Knots} Our final application is to knot theory. Recall that the HOMFLYPT polynomial is determined by
\begin{align*}
\HOMFLY(\bigcirc;a,z)&=1,\\
a\,\HOMFLY(L_+;a,z)-a^{-1}\,\HOMFLY(L_-;a,z)
&=
z\,\HOMFLY(L_0;a,z),
\end{align*}
Under the substitution $z=q^{1/2}-q^{-1/2}$, the specialization $z=0$
corresponds to $q=1$. We will show that $\HOMFLY(K;a,z=0)$ admits a reflection subword model. In~\cite[Theorem~1.8]{trinh2026partial}, the
$a$-coefficients of the HOMFLYPT polynomial were expressed in terms of
leading Hecke coefficients of braid wedges. Combining this result with
\Cref{thm:main} turns these leading coefficients into signed counts of
reduced reflection subwords.

Let $c=(12\cdots n)\in \mathfrak{S}_n$, and set $\mathfrak{S}_{n,k} = \bigl\{w\in \mathfrak{S}_n:\Des(w)=\{s_1,\ldots,s_k\}\bigr\}.$  For $w\in \mathfrak{S}_n$, write $\inv_c(w)$ for its inversion sequence in
$c$-sorting order. 

{
\renewcommand{\thetheorem}{\ref{thm:knots}}
\begin{theorem}
Let $\bbeta_K\in\B_n$ be a braid whose closure is a knot $K$. Let $\beta_K$ be the image of $\bbeta_K$ in $\mathfrak S_n$. Choose $g\in\mathfrak S_n$
such that $g\beta_Kg^{-1}=c$,
and choose a braid lift $\mathbf g\in\B_n$ of $g$. Factor the pure braid $\brho_K
:=
\mathbf g\bbeta_K\mathbf g^{-1}\B(c)^{-1}$ as
\[
\brho_K
=
\P_c(t_1)^{\epsilon_1}\cdots\P_c(t_M)^{\epsilon_M}
\]
for some $\vect=(t_1,\ldots,t_M)\in T^M$ and
$\protect\vv\epsilon=(\epsilon_1,\ldots,\epsilon_M)
\in\{\pm1\}^M$.
We have
\[
(-1)^{n-1}
\Subb_T^{\protect\vv\epsilon}(\vect,c^{-1})
=
\left.[a^{-\mathrm{wr}(\bbeta_K)-n+1}]
\homfly(K;a,z)\right|_{z=0}.
\]
More generally, we recover the full HOMFLYPT polynomial at $z=0$ by
\[
\homfly(K;a,z)\Big|_{z=0}
=
\sum_{k=0}^{n-1}
(-1)^{n-1+k}
\sum_{u\in \mathfrak S_{n,k}}
\Subb_T^{(1^{\ell(u)},\protect\vv\epsilon)}
\bigl(\rev(\inv_c(u))\vect,c^{-1}\bigr)
a^{-\mathrm{wr}(\bbeta_K)-n+1+2k}.
\]
\end{theorem}
\addtocounter{theorem}{-1}
}

As explained in~\Cref{sec:knot}, the same construction extends to links
by replacing the long cycle $c$ with the standard noncrossing
permutation having the cycle type of the link.

\newpage
\part{Technique}\label{part:technique}

In this part of the paper, we review background and introduce our methods.
\begin{itemize}
    \item In~\Cref{sec:background}, we recall standard concepts from combinatorics, Coxeter theory, and Coxeter--Catalan combinatorics.
    \item In~\Cref{sec:pure}, we recall the relationship between the shard generators of $\P_W$ and the dual braid generators of $\B_{W}$, and we discuss how to rewrite pure braids in terms of these generators.  Of special interest are pure braids of the form $\B(w)\B(w^{-1})$, which will occur in our work on rational Catalan models and noncrossing parking functions.  We also recover results on Mikado braids from~\cite{gobet2020dual,digne2017dual,baumeister2017simple,licata2017braid} in a new and uniform way.
    \item In~\Cref{sec:subwords}, we prove our main~\Cref{thm:main}, which interprets the leading term of $[\T_\pi]\H(\bbeta)$ (for $\bbeta \in \P_W$ and $\pi \in \NC(W,c^{-1})$) as a combinatorial model consisting of signed subwords that are reduced $T$-words for $\pi$.
\end{itemize}

\section{Background}\label{sec:background}

\subsection{General Groups}\label{sec:groups}
Let $G$ be a group. We denote the identity element of a group by $e$. Given a subset $U\subseteq G$, we define a \dfn{$U$-word} to be a tuple $(u_1,\ldots,u_M)$ such that the entries $u_1,\ldots,u_M$ all belong to $U$. We say the word $(u_1,\ldots,u_M)$ \dfn{represents} the element $u_1\cdots u_M$ of $G$. For $g\in G$, let $\ell_U(g)$ be the smallest length of a $U$-word that represents $g$. A \dfn{reduced $U$-word} for $g$ is a $U$-word that represents $g$ and has length $\ell_U(g)$. 

Given a word $(g_1,\ldots,g_M)\in G^M$ and an index $i\in[M-1]$, we let 
\[\Hur_i(g_1,\ldots,g_M)=(g_1,\ldots,g_{i-1},g_ig_{i+1}g_i^{-1},g_i,g_{i+2},\ldots,g_M).\] Note that $\Hur_i$ is an invertible operator on $G^M$ whose inverse is given by 
\[\Hur_{i}^{-1}(g_1,\ldots,g_M)=(g_1,\ldots,g_{i-1},g_{i+1},g_{i+1}^{-1}g_ig_{i+1},g_{i+2},\ldots,g_M).\] The operators $\Hur_i$ and $\Hur_{i}^{-1}$ are called \dfn{Hurwitz moves}. The set of all words that can be obtained from $(g_1,\ldots,g_M)$ via a sequence of Hurwitz moves is the \dfn{Hurwitz orbit} of $(g_1,\ldots,g_M)$. Note that two words in the same Hurwitz orbit necessarily represent the same element. 

A \dfn{commutation move} is an operation that changes a word by swapping two consecutive entries that commute with each other. Two words are \dfn{commutation equivalent} if one can be obtained from the other via a sequence of commutation moves. For our purposes, two words that are commutation equivalent will often behave in the same manner, so we will commonly work with words up to commutation equivalence. 

Given $\protect\vv{g}=(g_1,\ldots,g_M)\in G^M$, we let $\mathrm{rev}({\protect\vv g})=\vvl{g}:=(g_M,\ldots,g_1)$. We use the notation 
\[\prod_{a\in \protect\vv{g}}a=g_1\cdots g_M\quad\text{and}\quad\prod_{a\in\protect\vv{g}}^{\longleftarrow}a=g_M\cdots g_1.\] A \dfn{subword} of $\protect\vv{g}$ is a word ${\sf g}=(\overline{g}_1,\ldots,\overline g_M)$ such that $\overline g_i\in\{g_i,e\}$ for all $i$. If $\overline g_1\cdots\overline g_M=g$, then we call ${\sf g}$ a \dfn{$g$-subword} of $\protect\vv{g}$. If we assume in addition that none of $g_1,\ldots,g_M$ is the identity element $e$, then we can use a subword ${\sf g}$ to partition the index set $[M]$: we say an index $i$ is a \dfn{skip} of ${\sf g}$ if $\overline g_i=e$; otherwise we say $i$ is a \dfn{take} of ${\sf g}$. Let $\mathrm{skip}({\sf g})$ denote the number of skips of ${\sf g}$.

\subsection{Coxeter Groups}\label{subsec:coxeter}

Let $(W,S)$ be a finite irreducible Coxeter system with Coxeter presentation 
\[W=\langle S\mid s^2=(ss')^{m_{ss'}}=e\text{ for all distinct }s,s'\in S\rangle.\]
Let $r=|S|$ be the rank of this Coxeter system. The \dfn{Coxeter graph} of $(W,S)$ is an edge-labeled graph with vertex set $S$ in which two vertices $s,s'$ are adjacent if and only if $m_{ss'}\geq 3$; the label of the edge between $s$ and $s'$ is $m_{ss'}$. Let $w_\circ$ be the long element of $W$. For $J\subseteq S$, the \dfn{standard parabolic subgroup} $W_J$ is the subgroup of $W$ generated by $J$. A \dfn{parabolic subgroup} of $W$ is a subgroup that is conjugate to a standard parabolic subgroup. Since $W_J$ is a Coxeter group whose simple reflections are the elements of $J$, it has a long element, which we denote by $w_\circ(J)$. Let $\ell=\ell_S$ denote the usual Coxeter length function. When we refer to a \dfn{reduced word} for an element of $W$, we mean a reduced $S$-word. Let \[N=\ell(w_\circ).\] 

Write \[\RR=\RR_W=\{wsw^{-1}:w\in W,\, s\in S\}\] for the set of reflections in $W$. Note that $|T|=N$. The \dfn{Coxeter number} of $W$ is 
\[h=\frac{2N}{r}.\] 
Define the \dfn{inversion sequence} of an $S$-word $\mathsf{w}=(s_1,\ldots,s_M)$ to be the $T$-word $\inv(\mathsf{w})=(t_1,\ldots,t_M)$, where $t_i=s_1\cdots s_{i-1}s_is_{i-1}\cdots s_1$. 

For $\vect=(t_1,\ldots,t_M)\in T^M$ and $w\in W$, we let $\Sub_T(\vect,w)$ denote the set of subwords ${\sf w}$ of $\vect$ such that the word obtained from ${\sf w}$ by deleting all occurrences of $e$ forms a reduced $T$-word for $w$. By a slight abuse of terminology, we will often say the subwords in $\Sub_T(\vect,w)$ \emph{are} reduced $T$-words for $w$, even though this is technically wrong because such words could contain occurrences of $e$. 

A \dfn{descent} of an element $w\in W$ is a simple reflection $s\in S$ such that $\ell(ws)<\ell(w)$. Let $\Des(w)$ denote the set of descents of $w$. A \dfn{cover reflection} of $w$ is a reflection $wsw^{-1}$ such that $s\in\Des(w)$. 
 
Suppose $\protect\vv\beta=(s_1,\ldots,s_k)$ is an $S$-word. Fix $v\in W$. Consider a subword ${\sf w}=(\overline s_1,\ldots,\overline s_k)$ of $\protect\vv\beta$. We say a subword ${\sf w}$ of $\protect\vv{\beta}$ is \dfn{$v$-distinguished} if $\overline s_i=s_i$ whenever $s_i$ is a descent of $v\overline s_1\cdots \overline s_{i-1}$. We can consider the sequence \[v,\,v\overline s_1,\,v\overline s_1\overline s_2,\ldots,v\overline s_1\overline s_2\cdots\overline s_k\] as a walk in $W$ that starts at $v$; at the $i$th step, we either multiply by $s_i$ (if $\overline s_i=s_i$) or stay still (if $\overline s_i=e$). Saying ${\sf w}$ is $v$-distinguished means that, at each step, if multiplying by the simple reflection would cause the Coxeter length to decrease, then we must multiply by the simple reflection. Let $\mathcal D^v(\protect\vv\beta,w)$ be the set of $v$-distinguished $w$-subwords of $\protect\vv\beta$. A subword is called \dfn{distinguished} if it is $e$-distinguished.  Distinguished subwords encode a natural decomposition of Richardson varieties (and, more generally, braid varieties) into pieces that are products of tori times affine spaces~\cite{deodhar1985}.

Let $V$ be the standard geometric representation of $W$. A reflection $t$ of $W$ acts via the (geometric) reflection through a hyperplane $H_t$ in the Coxeter arrangement $\mathcal H_W$. Thus, $\mathcal H_W=\{H_t:t\in T\}$. Let $\Phi\subseteq V^*$ be the associated root system, and let $\Phi^+$ be the set of positive roots. For each root $\alpha\in\Phi$, let $t_\alpha\in \RR$ be the reflection through the hyperplane orthogonal to $\alpha$. For $u\in \RR$, let $\alpha_u\in\Phi^+$ be the positive root such that $u=t_{\alpha_u}$. 

Let $V^{\mathbb C}=V\otimes\mathbb C$ and $\mathcal H_W^{\mathbb C}=\{H^{\mathbb C}:H\in\mathcal H_W\}$, where $H^{\mathbb C}=H\otimes\mathbb C\subseteq V^{\mathbb C}$. A \dfn{regular vector} is a
vector in
\(
V^{\mathbb C}\setminus\bigcup_{H\in\mathcal H_W}H^{\mathbb C}
\).  An element of $W$ is \dfn{$d$-regular} if it has a regular eigenvector whose associated eigenvalue is a primitive $d$th root of unity.  A \defn{Coxeter element} is an $h$-regular element.

A \defn{standard Coxeter element} is an element of $W$ that can be written as a product of the simple reflections of $W$, with each simple reflection used exactly once. All standard Coxeter elements are Coxeter elements.  If $c$ is a standard Coxeter element, then $\ell(c)=\ell_T(c)=r$. The standard Coxeter elements are all conjugate to each other~\cite{humphreys1992reflection}. When $W$ is crystallographic, the set of all Coxeter elements is a single conjugacy class~\cite[Theorem 4.2, Proposition 4.7]{springer1974regular}.  In noncrystallographic types, however, the Coxeter elements form more than one conjugacy class, so there are Coxeter elements that are not conjugate to any standard Coxeter elements (for example, in $I_2(5)$ with $S=\{s_1,s_2\}$, the element $(s_1s_2)^2$ is a Coxeter element, but it is not conjugate to $s_1s_2$ or $s_2s_1$).  See~\cite{reiner2017non} for further details. 

Consider a standard Coxeter element $c$ of the Coxeter system $(W,S)$, and let ${\sf c}$ be a reduced word for $c$. Given an element $w\in W$, we define the \dfn{$c$-sorting word} of $w$, denoted ${\sf w}(c)$, to be the lexicographically first subword of the infinite word ${\sf c}^\infty={\sf c}{\sf c}{\sf c}\cdots$ that is a reduced $S$-word for $w$. The word ${\sf w}(c)$  depends on the particular choice of the reduced word ${\sf c}$. However, we will often abuse notation and ignore this technicality; doing so is justified by the fact that different reduced words for $c$ give rise to commutation-equivalent $c$-sorting words for $w$.  We similarly write $\inv(c)$ instead of $\inv({\sf c})$, which is justified by the fact that different reduced words for $c$ have commutation-equivalent inversion sequences. The $c$-sorting word $\woc$ for $w_\circ$ is of particular interest~\cite{reading2007clusters,speyer2009powers}. 

\subsection{Posets} 
We assume the reader is familiar with basic notions from the theory of posets and lattices, which can be found in \cite[Chapter 3]{Stanley2012}.   Let $P$ be a finite poset, and let \[\mathcal E(P)=\{(x,y)\in P\times P:x\lessdot y\}\] be the set of edges in the Hasse diagram of $P$. Let $\Lambda$ be a set. A map $\lambda\colon\mathcal E(P)\to\Lambda$ is called an \dfn{edge-labeling} of $P$. 

Let $L$ be a finite lattice with meet operation $\wedge$ and join operation $\vee$. Recall that an element $j\in L$ is \dfn{join-irreducible} if every set with join $j$ contains $j$. Equivalently, $j$ is join-irreducible if it covers exactly one element of $L$. Let $\mathcal J_L$ be the set of join-irreducible elements of $L$. 

We say $L$ is \dfn{semidistributive} if the following implications hold for all $x,y,z\in L$: 
\[\begin{aligned}
x\vee y&= x\vee z\Longrightarrow x\vee y= x\vee (y\wedge z)\\ 
x\wedge y&= x\wedge z\Longrightarrow x\wedge y= x\wedge (y\vee z).\end{aligned}\]
Assume $L$ is semidistributive. For each cover relation $x\lessdot y$ in $L$, this assumption ensures that the set $\{z\in L:z\vee x=y\}$ has a unique minimal element $j_{x,y}$; moreover, this element is necessarily join-irreducible. 

When $L$ is the poset of regions of a simplicial hyperplane arrangement, Reading found a natural bijection between join-irreducible elements of $L$ and certain polyhedral pieces of hyperplanes called \emph{shards}. For this reason, we will sometimes abuse terminology by referring to the join-irreducible elements of a general semidistributive lattice as \dfn{shards}. The \dfn{shard labeling} of $L$ is the edge-labeling $\mathcal E(L)\to\mathcal J_L$ given by $(x,y)\mapsto j_{x,y}$. 

The main semidistributive lattice that will interest us is the \dfn{(right) weak order} on the Coxeter group $W$, which is the partial order $\leq$ on $W$ defined so that $x\leq y$ if and only if $\ell(y)=\ell(x)+\ell(x^{-1}y)$. A reduced word $\sf w$ for an element $w$ can be seen as a saturated chain in the weak order from $e$ to $w$; let $\Sha({\sf w})$ be the sequence of shard labels of the edges of this chain, read from bottom to top. More precisely, if ${\sf w}=(s_1,\ldots,s_k)$ and we write $x_i=s_1\cdots s_i$, then \[\Sha({\sf w})=(j_{x_0,x_1},j_{x_1,x_2},\ldots,j_{x_{k-1},x_k}).\] 

Another important partial order on $W$ is the \dfn{absolute order} $\leq_{T}$, which is defined so that $x\leq_{T}y$ if and only if $\ell_T(y)=\ell_T(x)+\ell_T(x^{-1}y)$. The \dfn{$c$-noncrossing partition lattice}, denoted $\NC(W,c)$, is the interval between $e$ and $c$ in the absolute order. If $x\lessdot_{T} y$ is a cover relation
in $\NC(W,c)$, then $x^{-1}y$ is a reflection. The map that sends the edge $x\lessdot_{T}y$ to $x^{-1}y$ is the \dfn{natural reflection labeling} of $\NC(W,c)$. 

A shard is called \dfn{$c$-noncrossing} if it appears in $\Sha({\sf w}_\circ(c))$. Note that there is a natural one-to-one correspondence between $c$-noncrossing shards and reflections under which the $i$th shard in $\Sha({\sf w}_\circ(c))$ corresponds to the $i$th reflection in $\inv({\sf w}_\circ (c))$. In the geometric picture where a shard is a polyhedral piece of a reflecting hyperplane, this one-to-one correspondence comes from the fact that each hyperplane in $\mathcal H_W$ contains a unique $c$-noncrossing shard and the fact that the map $t\mapsto H_t$ is a bijection from $T$ to $\mathcal H_W$. An element $w\in W$ is \dfn{$c$-sortable} if for every element $x$ that is covered by $w$ in the weak order, the shard label $j_{x,w}$ is $c$-noncrossing. 

Given a $c$-sortable element $w\in W$, let $\mathrm{nc}_c(w)=r_1\cdots r_k$, where $r_1,\ldots,r_k$ are the cover reflections of $w$, listed in the order they appear in $\inv({\sf w}(c))$. N.\ Reading proved the following fundamental result. 

\begin{theorem}[{\cite[Theorem 6.1]{reading2007clusters}}]\label{thm:Reading_bijection}
The map $\mathrm{nc}_c$ is a bijection from the set of $c$-sortable elements of $W$ to the set $\NC(W,c)$ of $c$-noncrossing partitions.  
\end{theorem} 

For $x\in W$, let
\[\mathrm{Inv}(x)=\{u\in T:\ell(ux)<\ell(x)\}
\]
be the inversion set of $x$. Let $\inv_c(x)$ be the $T$-word obtained by restricting $\inv({\sf w}_\circ(c))$ to $\mathrm{Inv}(x)$. We will use the following terminology and facts from Reading
\cite[Sections~3--4]{reading2007clusters}. Reading defined the \dfn{$c$-orientation} on irreducible rank-2 parabolic subgroups using the ordering of $T$ given by $\inv({\sf w}_\circ(c))$. He then defined an element $x\in W$ to be \dfn{$c$-aligned} if in every irreducible rank-2 parabolic subgroup $W'$ with canonical generators $\rho_1,\rho_M$ oriented $\rho_1\to_c \rho_M$, we have the implication 
\[
\mathrm{Inv}(x)\cap \bigl((T\cap W')\setminus\{\rho_M\}\bigr)\neq \emptyset
\implies \rho_1\in\mathrm{Inv}(x).\]
Reading proved that $x$ is $c$-sortable if and only if $x$ is $c$-aligned \cite[Theorem~4.1]{reading2007clusters}; this uses the fact that $\mathrm{Inv}(x)\cap (T\cap W')$ is an initial segment or a final segment in the rank-2 reflection order \cite[Lemma~1.6]{reading2007clusters}.

\begin{lemma}\label{lem:join-irreducible}
Let $j$ be a $c$-sortable join-irreducible element of the weak order, and let $t$ be the unique cover reflection of $j$. If $r\in \mathrm{Inv}(j)\setminus\{t\}$, then
$tr\notin \NC(W,c^{-1})$.
\end{lemma}

\begin{proof}
This is a consequence of Reading's alignment characterization. Suppose $W'$ is an irreducible rank-2 parabolic subgroup, and write the reflections of $W'$ in the order induced by the $c$-sorting word for $w_\circ$ as $\rho_1,\rho_2,\ldots,\rho_M$,
where $\rho_1$ and $\rho_M$ are the canonical generators of $W'$ and $\rho_1\to_c \rho_M$. Then the reduced $T$-words for the $c^{-1}$-noncrossing partitions in $W'$ are precisely those of the form $(\rho_i,\rho_{i+1})$, where the indices are read modulo $M$. 

Now suppose $r\in \mathrm{Inv}(j)\setminus\{t\}$. Let $W''$ be the largest rank-2 parabolic subgroup of $W$ containing $t$ and $r$. Since $j$ is $c$-sortable, it is $c$-aligned. Hence, Reading's rank-2 description of inversion sets forces the inversions of $j$ in $W''$ to lie on the $c$-aligned side of the rank-2 reflection order. Since $t$ is the unique cover reflection of $j$, it is the unique removable reflection from this rank-2 inversion set. It follows that the ordered pair $t,r$ is not one of the rank-2 adjacent pairs listed above for the reversed orientation. Therefore, $tr$ is not the rank-2 $c^{-1}$-noncrossing Coxeter element in $W''$. It follows that $tr\notin \NC(W,c^{-1})$.
\end{proof}

\subsection{Braid Groups and Hecke Algebras}\label{subsec:braid_hecke} 
 Let $V^{\mathrm{reg}}:=V^{\mathbb C}\setminus \bigcup\mathcal H_W^{\mathbb C}$. The \dfn{braid group} $\B_W$ is the fundamental group $\pi_1(V^{\mathrm{reg}}/W)$, and the \dfn{pure braid group} $\P_W$ is the fundamental group $\pi_1(V^{\mathrm{reg}})$. The center of $\P_W$ is generated by the \dfn{full twist}, a loop that winds once around all of the complexified hyperplanes. 

Each simple reflection $s\in S$ has an associated \dfn{Artin lift} $\s\in\B_W$. Let ${\bf S}=\{\s:s\in S\}$. The \dfn{Artin presentation} of the braid group of $W$ is 
\[\B_W=\langle{\bf S}:\,[\mathbf{w}_\circ]\rangle,\]
where $[\mathbf{w}_\circ]$ is the relation that sets equal all expressions of the form $\s_1\cdots\s_N$ such that $(s_1,\ldots,s_N)$ is a reduced $S$-word for $w_\circ$. For $w\in W$, let \begin{equation}\label{def:artin_lift}\B(w):=\s_1\cdots\s_k,\end{equation} where $(s_1,\ldots,s_k)$ is a reduced $S$-word for $w$. The element $\B(w)\in\B_W$ is the \dfn{Artin lift} of $w$; it is well defined (i.e., does not depend on the chosen reduced word). The full twist is $\Delta^2$, where $\Delta=\B(w_\circ)$. 

Recall that the \defn{writhe} of a braid $\bbeta \in \B_W$ is the quantity $\mathrm{wr}(\bbeta)$ defined as the sum of the exponents of the simple Artin generators in any $\S$-word representing $\bbeta$.

The \dfn{positive braid monoid} $\B_W^+$ is the monoid generated by ${\bf S}$; in other words, it is the set of elements of $\B_W$ that can be expressed as products of the Artin lifts of simple reflections without using inverses. An element of $\B_W^+$ is called a \dfn{positive braid}.  

There is a natural quotient map $\B_W\twoheadrightarrow W$ that sends each generator $\s\in{\bf S}$ to the corresponding simple reflection $s\in S$. In general, a \dfn{lift} of an element $w\in W$ is a preimage of $w$ under this quotient map. 
The pure braid group $\P_W$ is the kernel of this quotient map; in other words, $\P_W$ is the set of lifts of the identity element of $W$. 

The \dfn{Hecke algebra} $\H_W$ is the quotient of the group algebra $\mathbb{Z}[q^{\pm 1}][\B_W]$ by the quadratic relations
\[
(\s-q)(\s+1)=0 \quad \text{for all } \s\in {\bf S}.
\]
We denote the quotient map by $\H\colon \mathbb{Z}[q^{\pm 1}][\B_W]\to\H_W$.
One can view $\H_W$ as a deformation of the group algebra $\mathbb{Z}[W]$ since we recover this group algebra by setting $q=1$.  For each $w\in W$, let $\T_w=\H(\B(w))$. Then $\{\T_w\}_{w\in W}$ is the \dfn{standard basis} for $\H_W$ with multiplication determined by
\[
\T_s\T_w=
\begin{cases}
\T_{sw} & \text{if } \ell(sw)>\ell(w)\\
(q-1)\T_w+q\T_{sw} & \text{if } \ell(sw)<\ell(w)
\end{cases}
\] 
for all $w\in W$ and $s\in S$. 

Given $\mathrm{x}\in \H_W$ and $w\in W$, we write $[\T_w]\mathrm{x}$ for the coefficient of $\T_w$ in the standard-basis expansion of $\mathrm{x}$. The map $\H_W\to \mathbb Z[q^{\pm 1}]$ defined by $\mathrm{x}\mapsto[\T_e]\mathrm{x}$ is a trace map in the sense that it is linear and satisfies $[\T_e](\mathrm{xy})=[\T_e](\mathrm{yx})$ for all $\mathrm{x},\mathrm{y}\in \H_W$. 
We will repeatedly use the following coefficient identity, which is similar to~\cite[Example 7.4]{trinh2026partial} and \cite{kalman2009meridian}.

\begin{lemma}[{\cite[Example 7.4]{trinh2026partial},\cite{kalman2009meridian}}]\label{lem:coeff-delta}
For $\bbeta \in \B_W$, we have 
\[
[\T_e]\H({\Delta^2} {\bbeta}) = (-1)^{\mathrm{wr}(\bbeta)} q^{N+\mathrm{wr}(\bbeta)} [\T_e]\H({\bbeta}^{-1}). 
\] 
\end{lemma} 
\begin{proof}
The linear map sending each standard basis element $\T_w$ to its inverse $\T_w^{-1}$ is lower-triangular with respect to the Bruhat order, so $\{\T_w^{-1}\}_{w\in W}$ is a basis of $\H_W$. Let $\varpi$ be the anti-automorphism of $\H_W$ defined by $\varpi(\T_{s})=-q\T_s^{-1}=q-1-\T_s$ for all $s\in S$. For $\bbeta\in\B_W$, we have \[\varpi(\H(\bbeta))=(-q)^{\mathrm{wr}(\bbeta)}\H(\bbeta^{-1}).\] We will prove that 
\begin{equation}\label{eq:varpi}
[\T_e]\H(\Delta^2)\mathrm{x}=q^N[\T_e]\varpi(\mathrm{x})
\end{equation} for all $\mathrm{x}\in \H_W$. It suffices to prove this when $\mathrm{x}$ belongs to the basis $\{\T_w^{-1}\}_{w\in W}$. 

Fix $w\in W$, and let $\mathrm{x}=\T_w^{-1}$. Let $s_1\cdots s_k$ be a reduced $S$-word for $w_\circ w^{-1}$, and let $s_{k+1}\cdots s_N$ be a reduced $S$-word for $w$. We have \[\H(\Delta^2)\mathrm{x}=\T_{w_\circ}\T_{w_\circ}\T_w^{-1}=\T_{w_\circ}(\T_{s_1}\cdots\T_{s_N})\T_{s_N}^{-1}\cdots\T_{s_{k+1}}^{-1}=\T_{w_\circ}\T_{s_1}\cdots\T_{s_k}.\] If $w\neq e$, then $k<\ell(w_\circ)$, so $\T_e$ cannot appear in the standard-basis expansion of $\T_{w_\circ}\T_{s_1}\cdots\T_{s_k}$. In this case, we have $[\T_e]\H(\Delta^2)\mathrm{x}=0$, and $[\T_e]\varpi(\mathrm{x})=[\T_e](-q)^{-\ell(w)}\T_w=0$, so \eqref{eq:varpi} holds in this case. Now assume $w=e$. Then $s_1\cdots s_k$ is a reduced word for $w_\circ$, so the coefficient of $\T_e$ in $\T_{w_\circ}\T_{s_1}\cdots\T_{s_k}$ counts the walk in the weak order from $w_\circ$ to $e$ with steps obtained by multiplying by $s_1,\ldots,s_k$, with each step contributing a factor of $q$. Thus, $[\T_e]\H(\Delta^2)\mathrm{x}=q^N$; since $\mathrm{x}=\T_e$, the identity in \eqref{eq:varpi} holds in this case as well. 

Now take $\mathrm{x}=\H(\bbeta)$ in \eqref{eq:varpi} to find that 
\[[\T_e]\H(\Delta^2\bbeta)=q^N[\T_e]\varpi(\H(\beta))=q^N[\T_e](-q)^{\mathrm{wr}(\bbeta)}\H(\bbeta^{-1}).\] This simplifies to the desired identity. 
\end{proof}

\begin{lemma}\label{lem:coeff-identity}
For $\mathrm{x}\in\H_W$ and $w\in W$, we have 
\[
 [\T_e](\mathrm{x} \T_w) =[\T_e](\T_w \mathrm{x})= q^{\ell(w)} [\T_{w^{-1}}]\, \mathrm{x}.
\]
\end{lemma}

\begin{proof}
The first equality is just the trace property. 
We prove the second equality by induction on $\ell(w)$. The case $w=e$ is immediate, so assume $\ell(w)\geq 1$. Write $w=vs$, where
$\ell(w)=\ell(v)+1$. Then
\[
[\T_e](\T_w\mathrm{x})=[\T_e](\T_v\T_s\mathrm{x})=q^{\ell(v)}[\T_{v^{-1}}](\T_s\mathrm{x}).
\]
Since $\ell(sv^{-1})=\ell(v^{-1})+1$, the only term of $\mathrm{x}$ that can contribute to the coefficient of
$\T_{v^{-1}}$ in $\T_s\mathrm{x}$ is $\T_{sv^{-1}}$, and it contributes with coefficient $q$. Hence, $[\T_{v^{-1}}](\T_s\mathrm{x})=q[\T_{sv^{-1}}]\,\mathrm{x}=q[\T_{w^{-1}}]\,\mathrm{x}$.
This yields the desired identity. 
\end{proof}

\subsection{Type~A} 
We will illustrate several concepts using examples in type~$A$. We write $\mathfrak{S}_n$ for the symmetric group of permutations of the set $[n]:=\{1,\ldots,n\}$, which is the Coxeter group of type $A_{n-1}$. In this setting, the simple reflections are $s_1,\ldots,s_{n-1}$, where we let $s_i$ denote the transposition $(i\,\,i+1)$.  
When $W=\mathfrak{S}_n$, the braid group $\B_W$ and the pure braid group $\P_W$ are the usual $n$-strand braid group and pure braid group, which we denote by $\B_n$ and $\P_n$, respectively. We sometimes denote the transposition $(ij)$ by $t_{(ij)}$ so that we can write the $c$-dual lift as $\t_{(ij)}$ and the square of the $c$-dual lift as $\tt_{(ij)}$. 

\section{Dual Lifts and the Full Twist} \label{sec:pure}
The purpose of this section is to develop explicit factorizations of certain pure braids in terms of generators that are compatible with noncrossing combinatorics. More precisely, after fixing a standard Coxeter element $c$, we want to rewrite pure Artin braids as products of the pure braid generators  coming from the $c$-dual lifts of reflections. Such factorizations are the input needed later to turn Hecke coefficients into subword counts.

A natural starting point is the Reidemeister--Schreier generating set for $\P_W\subseteq \B_W$ obtained by using the Artin lifts $\B(w)$ for $w\in W$ as coset representatives. This procedure produces generators indexed by cover relations in weak order, namely braids of the form $(\B(u')\B(u)^{-1})^2$ with $u\lessdot u'$. This generating set is redundant. In \cref{sec:shards}, we show that the braid $\B(u')\B(u)^{-1}$ depends only on the shard label of the cover relation $u\lessdot u'$. 

We then identify the noncrossing part of this shard picture with the usual dual braid group. In \cref{sec:noncrossing_shards}, we recall the $c$-dual lifts of reflections and show how to express them in the standard Artin generators using the $c$-sorting word for $w_\circ$. The resulting pure braids $\tt=\t^2$ are precisely the pure shard generators attached to the $c$-noncrossing shards. We also extend this comparison from individual reflections to noncrossing partitions, leading in \cref{sec:Mikado} to a uniform proof that the dual lift of a noncrossing partition is a Mikado braid. 

The main factorization result of the section concerns the pure braids $\B(w)\B(w^{-1})$ for $w\in W$.
In \cref{sec:Pure_from_Coxeter}, we prove that every such braid belongs to the dual pure braid monoid, and we give an explicit factorization indexed by the inversions of $w$. The full twist comes from the special case where $w=w_\circ$; in \cref{cor:wo}, this yields the factorization of $\Delta^2$ as the product of the pure dual generators associated to the inversion sequence of the $c$-sorting word for $w_\circ$. These factorizations will be used throughout \cref{part:applications}, especially for the Fuss--Catalan, Fuss--Dogolon, parking, and knot-theoretic applications. In later sections, when more complicated pure braids arise, we will either guess similarly efficient factorizations or fall back on Reidemeister--Schreier rewriting.

\subsection{Shards and Braids}\label{sec:shards}
We are interested in braids of the form $\B(u')\B(u)^{-1}$, where $u\lessdot u'$ in the weak order. A primary motivation for studying these elements comes from the Reidemeister--Schreier method applied to $\P_W \subset \B_W$ with coset representatives given by the Artin lifts $\B(w)$ for $w \in W$. This method yields that $\P_W$ can be generated by $\{(\B(u')\B(u)^{-1})^2:u,u'\in W,\, u\lessdot u'\}$~\cite{salvetti1987topology}, although this generating set is redundant. Our first result explains exactly how redundant this generating set is, recovering a theorem we proved in~\cite[Theorem 1.1]{defant2022pop}. Namely, we show that the element $\B(u')\B(u)^{-1} \in \B_W$ depends only on the shard label $j_{u,u'}$.  The algebraic proof we give here is rather simpler than our earlier geometric proof. 
\begin{theorem}\label{thm:shards}
Let $u\lessdot u'$ and $v\lessdot v'$ be cover relations in the weak order on $W$. If $j_{u,u'}=j_{v,v'}$, then 
\[\B(u')\B(u)^{-1}=\B(v')\B(v)^{-1}.\]
\end{theorem}
\begin{proof}
Let $j=j_{u,u'}=j_{v,v'}$, and let $j_*$ be the unique element covered by $j$. We will prove that $\B(u')\B(u)^{-1}=\B(j)\B(j_*)^{-1}$. Since the same argument shows that $\B(v')\B(v)^{-1}=\B(j)\B(j_*)^{-1}$, this will complete the proof. 

Let us write $u'=us$ and $j=j_*s'$, where $s,s'\in S$. Let $t=usu^{-1}$ and $t'=j_*s'j_*^{-1}$, and let $z=j^{-1}u'$. It is known that $t=t'$. Indeed, this is because the (polyhedral) shard corresponding to $j$ is contained in the hyperplanes $H_t$ and $H_{t'}$, and each shard is contained in a unique hyperplane from $\mathcal H_W$. Hence, $u=tu'=tjz=j_*z$. 

By the definition of the shard label $j=j_{u,u'}$, we have $j\vee u=u'$, so $j\leq u'$. This implies that $\B(u')=\B(j)\B(z)$. In addition, 
\[\ell(u)=\ell(u')-1=\ell(j)+\ell(z)-1=\ell(j_*)+\ell(z),\] so $\B(u)=\B(j_*)\B(z)$. This shows that 
\begin{align*}
\B(u')\B(u)^{-1}&=(\B(j)\B(z))(\B(j_*)\B(z))^{-1} =\B(j)\B(j_*)^{-1},
\end{align*}
as desired. 
\end{proof} 

\subsection{Noncrossing Shards and Dual Braids}\label{sec:noncrossing_shards}
We are interested in particular lifts of reflections that depend on a fixed standard Coxeter element $c$ and feature prominently in dual Garside theory. These lifts come from the following theorem due to Bessis.  

\begin{theorem}[{\cite{bessis2003dual}}]\label{thm:Bessis_Hurwitz} 
Let $c$ be a standard Coxeter element of $W$. Let $(s_1,\ldots,s_r)$ be a reduced word for $c$. For each $t\in T$, there exists a unique lift $\t$ of $t$ in $\B_W$ that appears in a word in the Hurwitz orbit of the ${\bf S}$-word $(\s_1,\ldots,\s_r)$. 
\end{theorem}

We refer to $\t$ as the \dfn{$c$-dual lift} of $t$. Note that $\t$ depends on $c$; when we wish to stress this dependence, we will denote $\t$ by $\B_c(t)$. We are also interested in the pure braid $\tt=\t^2$. When we wish to stress the dependence on $c$, we denote $\tt$ by $\P_c(t)$. Thus,
$\P_c(t)\coloneq\B_c(t)^2$.
Bessis's \dfn{dual presentation} of the braid group is  
\[\B_{W}=\langle \R:\,[\bc]\rangle,\] 
where $\R=\R_c=\{\t: t\in T\}$ and $[\bc]$ is the relation that sets equal all elements of the form $\t-1\cdots\t_r$ such that $(t_1,\ldots,t_r)$ is a reduced $\RR$-word for $c$. 

\begin{remark}
According to D.~Bessis~\cite[Section 3]{bessis2003dual}, one should view the $c$-dual presentation as a presentation for a different group called the \emph{dual braid group}. Then the braid group and the $c$-dual braid group are fundamental groups of  $V^{\mathrm{reg}}/W$ using different base points. From this perspective, the two groups are technically different, but they are isomorphic to each other since $V^{\mathrm{reg}}/W$ is connected. For our purposes, we find it simpler to view them as the same group. Note that we can identify the groups coming from the Artin presentation and the $c$-dual presentation by first identifying the lifts of the simple reflections.  
\end{remark} 

The \dfn{dual braid monoid} is the monoid $\B_{W,c}^+$ generated by the set $\R={\bf T}_c=\{\t:t\in T\}$. The \dfn{dual pure braid monoid} is the monoid $\P_{W,c}^+$ generated by $\mathbb T=\mathbb{T}_c=\{\tt:t\in T\}$. 

The $c$-dual lifts can also be obtained naturally from the $c$-sorting word for $w_\circ$; this is well known to experts, but we have not seen an explicit statement and proof in the literature (the closest may be~\cite[Proposition~3.13]{digne2017dual}, which proves a similar result for the $h$th power of a reduced word for $c$). 

\begin{proposition}\label{prop:dual_artin}
Let ${\sf w}_\circ(c)=(s_1,\ldots,s_N)$, and write $t_i=s_1\cdots s_{i-1}s_is_{i-1}\cdots s_1$ so that the inversion sequence of ${\sf w}_\circ(c)$ is ${\inv({\sf w}_\circ(c))=(t_1,\ldots,t_N)}$. For each $i\in[N]$, we have \[\t_i=\B_c(t_i)=\s_1\cdots\s_{i-1}\s_i\s_{i-1}^{-1}\cdots\s_1^{-1}.\]    
\end{proposition}
    
\begin{proof}
We proceed by induction on $i$, noting that the result is obvious when $i=1$. Now assume $i\geq 2$. Consider the standard Coxeter element $\widetilde c=s_1cs_1$. The word ${\sf w}_\circ(\widetilde c)$ is commutation equivalent to $(s_2,\ldots,s_N,s')$, where $s'=w_\circ s_1 w_\circ$. Moreover, the inversion sequence $\inv({\sf w}_\circ(\widetilde c))$ is commutation equivalent to $(u_1,\ldots,u_{N-1},s_1)$, where $u_j=s_1t_{j+1}s_1$. We know by induction that $\B_{\widetilde c}(u_{i-1})=\s_2\cdots \s_{i-1}\s_i\s_{i-1}^{-1}\cdots\s_2^{-1}$. By definition, this means that $\s_2\cdots \s_{i-1}\s_i\s_{i-1}^{-1}\cdots\s_2^{-1}$ appears in a word in the Hurwitz orbit of $(\s_2,\cdots,\s_r,\s_1)$. Each Hurwitz move commutes with the operation of simultaneously conjugating all entries in a tuple by $\s_1$. Therefore, $\s_1(\s_2\cdots \s_{i-1}\s_i\s_{i-1}^{-1}\cdots\s_2^{-1})\s_1^{-1}$ appears in a word in the Hurwitz orbit of $(\s_1\s_2\s_1^{-1},\cdots,\s_1\s_r\s_1^{-1},\s_1)$. Finally, 
\[(\Hur_1^{-1}\circ\Hur_2^{-1}\circ\cdots\circ\Hur_{r-1}^{-1})(\s_1\s_2\s_1^{-1},\cdots,\s_1\s_r\s_1^{-1},\s_1)=(\s_1,\ldots,\s_r),\] so $\s_1\s_2\cdots \s_{i-1}\s_i\s_{i-1}^{-1}\cdots\s_2^{-1}\s_1^{-1}$ appears in a word in the Hurwitz orbit of $(\s_1,\ldots,\s_r)$. 
\end{proof} 

Given a shard (i.e., join-irreducible element of the weak order) $j$, let 
\[\mathbb{\Sigma}_j=(\B(j)\B(j_*)^{-1})^2,\]where $j_*$ is the unique element covered by $j$ in the weak order. 
\cref{thm:shards,prop:dual_artin} imply that if we write \[\inv({\sf w}_\circ(c))=(t_1,\ldots,t_N)\quad\text{and}\quad\Sha({\sf w}_\circ(c))=(j_1,\ldots,j_N),\] then \[\mathbb \Sigma_{j_i}=\tt_i=\t_i^2.\]

For the next lemma, we need the \dfn{pop-stack operator} $\Pop\colon W\to W$ defined by \[\Pop(x)=x\,w_\circ(\Des(x)),\] where we recall that $w_\circ(\Des(x))$ is the long element of the parabolic subgroup generated by the descents of $x$. Equivalently, $\Pop(x)$ is the meet (in the weak order) of $x$ with the elements covered by $x$ \cite{defantPopCoxeter}. 

\begin{lemma}\label{lem:semidistrim}
Let $w\in W$ be $c$-sortable. Let $t$ be a cover reflection of $w$, and let $v=\Pop(w)$. There exists $s\in\Des(w)$ such that $j_{v,vs}$ is the $c$-noncrossing shard corresponding to the reflection $t$. Moreover, 
\[
    \t=\B_c(t)=\B(v)\,\s\,\B(v)^{-1}.
\]
\end{lemma} 

\begin{proof} 
Because $w$ is $c$-sortable, $j_{tw,w}$ is the $c$-noncrossing shard corresponding to $t$. It follows from \cite[Proposition~9.5]{defant2023semidistrim} that there is a simple reflection $s\in S$ such that $v\lessdot vs\leq w$ and $j_{v,vs}=j_{tw,w}$. Because $v\,w_\circ(\Des(w))=w$ and $\ell(w)=\ell(v)+\ell(w_\circ(\Des(w)))$, the simple reflection $s$ must be in $\Des(w)$. Finally, let $\inv({\sf w}_\circ(c))=(t_1,\ldots,t_N)$, and let $i$ be the index such that $t_i=t$. Let $x_{i-1}=s_1\cdots s_{i-1}$ and $x_i=s_1\cdots s_i=x_{i-1}s_i$. Since $j_{x_{i-1},x_i}=j_{v,vs}$, it follows from \cref{thm:shards,prop:dual_artin} that \[\B_c(t)=\B(x_i)\B(x_{i-1})^{-1} 
=\B(vs)\B(v)^{-1} = \B(v)\,\s\,\B(v)^{-1}. \qedhere\] 
\end{proof} 

The $c$-dual lifts of reflections extend to an embedding $\B_c$ of $\NC(W,c)$ into the (dual) braid group $\B_{W}$. For $\pi\in\NC(W,c)$, we define \[\B_c(\pi)=\B_c(t_1)\cdots\B_c(t_k),\] where $(t_1,\ldots,t_k)$ is a reduced $\RR$-word for $\pi$ (it follows from the relations in Bessis's dual presentation that the choice of reduced $T$-word does not matter).

\subsection{Mikado Braids}\label{sec:Mikado} 
A \dfn{Mikado braid} is a braid of the form $\bbeta_1\bbeta_2^{-1}$, where $\bbeta_1,\bbeta_2\in\B_W^+$ are positive braids~\cite[Theorem 5.8]{digne2017dual}. T.~Gobet proved that for every $\pi\in\NC(W,c)$, the braid $\B_c(\pi)$ is a Mikado braid.  His approach relied on some case-by-case computations in \cite[Lemma 4.6]{gobet2020dual} (see also~\cite{digne2017dual,baumeister2017simple,licata2017braid}). Our next theorem explains Gobet's result in a completely type-uniform manner. Recall Reading's bijection $\mathrm{nc}_c$ from \cref{thm:Reading_bijection}. 

\begin{theorem}\label{thm:mikado}
Let $c$ be a standard Coxeter element of $W$. Let $\pi\in\NC(W,c)$, and consider the $c$-sortable element $w_\pi=\mathrm{nc}_c^{-1}(\pi)\in W$. We have 
\[\B_c(\pi)=\B(w_\pi)\B(\pi^{-1}w_\pi)^{-1}.\] 
\end{theorem} 

\begin{proof}
For notational convenience, let $v=\Pop(w_\pi)$. Then $w_\pi=v\,w_{\circ}(\Des(w_\pi))$, and $v$ is the minimum-length representative
of the coset $v\,W_{\Des(w_\pi)}$.

Let $(r_1,\ldots,r_k)$ be the word obtained from $\inv({\sf w}_\pi(c))$ by keeping only the cover reflections of $w_\pi$. Then $\pi=r_1\cdots r_k$ by the definition of Reading's bijection. Let $u_i=v^{-1}r_i v$.
Since $r_1,\ldots,r_k$ are the cover reflections of $w_\pi$, it follows from \cref{lem:semidistrim} that $u_1,\ldots,u_k$ are precisely the descents of $w_\pi$. In particular,
$c':=u_1u_2\cdots u_k$
is a Coxeter element of the parabolic subgroup $W_{\Des(w_\pi)}$. Since $w_{\circ}(\Des(w_\pi))$ is the maximum element of $W_{\Des(w_\pi)}$ in weak order, we can choose $y\in W_{\Des(w_\pi)}$ such
that
\[
    w_{\circ}(\Des(w_\pi))=c'y\quad\text{and}\quad 
    \ell(w_{\circ}(\Des(w_\pi)))=\ell(c')+\ell(y).
\]
Then
\[
    \pi=r_1\cdots r_k
        =v\,u_1\cdots u_k\,v^{-1}
        =vc'v^{-1}.
\]
Hence,
\[
    \pi^{-1}w_\pi
      =(vc'v^{-1})^{-1}v\,w_{\circ}(\Des(w_\pi))
      =v(c')^{-1}w_{\circ}(\Des(w_\pi))
      =vy.
\]
This also shows that $\pi^{-1}w_\pi\leq w_\pi$ in weak order.  Now compute that
\begin{align*}
    \B(w_\pi)\B(\pi^{-1}w_\pi)^{-1}
      &= \B(v\,w_{\circ}(\Des(w_\pi)))\,\B(vy)^{-1}  \\
      &= \B(v)\B(w_{\circ}(\Des(w_\pi)))\,\B(y)^{-1}\B(v)^{-1} \\
      &= \B(v)\B(c')\B(y)\,\B(y)^{-1}\B(v)^{-1} \\
      &= \B(v)\B(c')\B(v)^{-1} \\
      &= \B(v)\mathbf{u}_1\cdots \mathbf{u}_k\B(v)^{-1} \\
      &= \prod_{i=1}^{k} (\B(v)\mathbf{u}_i\B(v)^{-1}).
\end{align*}
By \cref{lem:semidistrim}, $\B(v)\mathbf{u}_i\B(v)^{-1}=\B_c(r_i)$ for every $i$. Therefore,
\[
    \B(w_\pi)\B(\pi^{-1}w_\pi)^{-1}
      =\B_c(r_1)\cdots\B_c(r_k)=\B_c(\pi),
\]
as desired.
\end{proof}

\subsection{Pure Braids from Coxeter Group Elements}\label{sec:Pure_from_Coxeter} 
Our next goal in this section is to show that for every standard Coxeter element $c$, every braid of the form $\B(w)\B(w^{-1})$ with $w\in W$ belongs to the dual pure braid monoid $\P_{W,c}^+$. We first need some lemmas. 

Fix ${\gamma}=(r_1,\ldots,r_M)\in S^M$, and let ${\boldsymbol{\gamma}}=(\mathbf{r}_1,\ldots,\mathbf{r}_M)\in{\bf S}^M$. Fix $x\in W$, and set
\[
    x_i=xr_1r_2\cdots r_i
\]
for $0\leq i\leq M$. Note that for each $i$, the elements $x_{i-1}$ and $x_i$ form a cover relation in the weak order on $W$ (so either $x_{i-1}\lessdot x_i$ or $x_{i}\lessdot x_{i-1}$). 
Define
\[
    {\bf D}_x({\boldsymbol{\gamma}})
    =
    \prod_{\substack{1\leq i\leq M\\ x_i\lessdot x_{i-1}}}^{\longrightarrow}
    \B(x_{i})\,{\bf r}_i^2\,\B(x_i)^{-1}
    \quad\text{and}\quad {\bf U}_x({\boldsymbol{\gamma}})
    =
    \prod_{\substack{1\leq i\leq M\\ x_{i-1}\lessdot x_{i}}}^{\longleftarrow}
    \B(x_{i-1})\,{\bf r}_i^2\,\B(x_{i-1})^{-1},
\]
where $\prod\limits^{\longrightarrow}$ means the product taken in increasing order of $i$ and
$\prod\limits^{\longleftarrow}$ means the product taken in decreasing order of $i$. 

Given ${\boldsymbol\alpha}=\s_1\cdots\s_k\in\B_W^+$, where $\s_1,\ldots,\s_k\in{\bf S}$, we write $\mathrm{rev}({\boldsymbol\alpha})=\s_k\cdots\s_1$.  

\begin{lemma}\label{lem:DA}
Let ${\gamma}=(r_1,\ldots,r_M)\in S^M$ and ${\boldsymbol{\gamma}}=(\mathbf{r}_1,\ldots,\mathbf{r}_M)\in{\bf S}^M$. For $x\in W$, we have
\[
    {\bf D}_x(\boldsymbol{\gamma}){\bf U}_x(\boldsymbol{\gamma})
    =
    \B(x){\bf r}_1\cdots {\bf r}_M{\bf r}_M\cdots {\bf r}_1\B(x)^{-1}.
\]
\end{lemma}

\begin{proof}
Let $\bbeta_i=\mathbf{r}_1\cdots\mathbf{r}_i$. We first prove that
\[
    \B(x)\bbeta_i={\bf D}_x(\bbeta_i)\B(x_i)
\]
for every $i$. This is clear when $i=0$, so we may assume $i\geq 1$ and proceed by induction. Assume $\B(x)\bbeta_{i-1}={\bf D}_x(\bbeta_{i-1})\B(x_{i-1})$. 
If $x_{i-1}\lessdot x_i$, then ${\bf D}_x(\bbeta_i)={\bf D}_x(\bbeta_{i-1})$, so
\[
    \B(x)\bbeta_i
    =
    {\bf D}_x(\bbeta_{i-1})\B(x_{i-1})\mathbf{r}_i
    =
    {\bf D}_x(\bbeta_i)\B(x_i).
\]
On the other hand, if $x_i\lessdot x_{i-1}$, then ${\bf D}_x(\bbeta_i)={\bf D}_x(\bbeta_{i-1})\B(x_{i})\,{\bf r}_i^2\,\B(x_i)^{-1}$, so 
\[\B(x)\bbeta_i
    =
    {\bf D}_x(\bbeta_{i-1})\B(x_{i-1})\mathbf{r}_i  =
    {\bf D}_x(\bbeta_{i-1})\B(x_i)\mathbf{r}_i^2 =
    {\bf D}_x(\bbeta_i)\B(x_i).
\]
This proves the claim.

Now set
\[
    {\bf F}_i
    =
    \B(x){\bf r}_1\cdots{\bf r}_i{\bf r}_i\cdots{\bf r}_1\B(x)^{-1}.
\]
We prove by induction that
\[
    {\bf F}_i={\bf D}_x(\bbeta_i){\bf U}_x(\bbeta_i).
\]
The result is clear for $i=0$. For $i\geq 1$, we have
\[
    {\bf F}_i
    =
    \B(x)\bbeta_{i-1}\mathbf{r}_i^2
    \operatorname{rev}(\bbeta_{i-1})\B(x)^{-1}.
\]
Using the identity $\B(x)\bbeta_{i-1}
    =
    {\bf D}_x(\bbeta_{i-1})\B(x_{i-1})$,
we compute
\[
\begin{aligned}
    {\bf F}_i
    &=
    \B(x)\bbeta_{i-1}\mathbf{r}_i^2
    \operatorname{rev}(\bbeta_{i-1})\B(x)^{-1} \\
    &=
    {\bf D}_x(\bbeta_{i-1})\B(x_{i-1})\mathbf{r}_i^2
    \operatorname{rev}(\bbeta_{i-1})\B(x)^{-1} \\ 
    &=
    {\bf D}_x(\bbeta_{i-1})
    \left(\B(x_{i-1})\mathbf{r}_i^2\B(x_{i-1})^{-1}\right)
    \B(x_{i-1})\operatorname{rev}(\bbeta_{i-1})\B(x)^{-1}. 
\end{aligned}
\]
In addition, 
\[
\begin{aligned}
    {\bf F}_{i-1}
    &=
    \B(x)\bbeta_{i-1}\operatorname{rev}(\bbeta_{i-1})\B(x)^{-1} \\
    &=
    {\bf D}_x(\bbeta_{i-1})\B(x_{i-1})
    \operatorname{rev}(\bbeta_{i-1})\B(x)^{-1},
\end{aligned}
\]
so 
\[
    \B(x_{i-1})\operatorname{rev}(\bbeta_{i-1})\B(x)^{-1}
    =
    {\bf D}_x(\bbeta_{i-1})^{-1}{\bf F}_{i-1}.
\]
Substituting this into the previous expression yields
\[{\bf F}_i
    =
    {\bf D}_x(\bbeta_{i-1})
    \left(\B(x_{i-1})\mathbf{r}_i^2\B(x_{i-1})^{-1}\right)
    {\bf D}_x(\bbeta_{i-1})^{-1}
    {\bf F}_{i-1}
\] By induction, ${\bf F}_{i-1}={\bf D}_x(\bbeta_{i-1}){\bf U}_x(\bbeta_{i-1})$, so
\[
    {\bf F}_i
    =
    {\bf D}_x(\bbeta_{i-1})\left(\B(x_{i-1})\mathbf{r}_i^2\B(x_{i-1})^{-1}\right){\bf U}_x(\bbeta_{i-1}).
\]

If $x_i\lessdot x_{i-1}$, then \[{\bf D}_x(\bbeta_i)={\bf D}_x(\bbeta_{i-1})\left(\B(x_{i-1})\mathbf{r}_i^2\B(x_{i-1})^{-1}\right)\quad\text{and}\quad{\bf U}_x(\bbeta_i)={\bf U}_x(\bbeta_{i-1}).\] If
$x_{i-1}\lessdot x_i$, then \[{\bf D}_x(\bbeta_i)={\bf D}_x(\bbeta_{i-1})\quad\text{and}\quad{\bf U}_x(\bbeta_i)=\left(\B(x_{i-1})\mathbf{r}_i^2\B(x_{i-1})^{-1}\right){\bf U}_x(\bbeta_{i-1}).\] In either case, ${\bf F}_i={\bf D}_x(\bbeta_i){\bf U}_x(\bbeta_i)$, as desired. Taking $i=M$ proves the
lemma. 
\end{proof} 

\begin{lemma}\label{lem:rank-2} 
Let $s,s'\in S$. Let $M=m_{s,s'}$, and let
\[
    a_1a_2\cdots a_M=ss's\cdots\quad\text{and}
    \quad
    b_1b_2\cdots b_M=s'ss'\cdots
\]
be the two alternating reduced words for $w_\circ(\{s,s'\})$. Suppose that
$x\in W$ is minimal in its right coset $x\,W_{\{s,s'\}}$. The weak order has two saturated chains
\[
    x=y_0\lessdot y_1\lessdot\cdots\lessdot y_M
    =
    xw_\circ(\{s,s'\})
\]
and
\[
    x=y'_0\lessdot y'_1\lessdot\cdots\lessdot y'_M
    =
    xw_\circ(\{s,s'\}),
\]
where $y_i=xa_1a_2\cdots a_i$ and $y'_i=xb_1b_2\cdots b_i.$
Let $j_i=j_{y_{i-1},y_i}$ and $j_i'=j_{y_{i-1}',y_i'}$. Then for all $0\leq k\leq M$,
\[
    \mathbb{\Sigma}_{j_k}\cdots \mathbb{\Sigma}_{j_1}
    =
    \mathbb{\Sigma}_{j_M'}\cdots \mathbb{\Sigma}_{j_{M-k+1}'}.
\] 
\end{lemma} 
\begin{proof}
Let $\Delta_{\{s,s'\}}={\bf a}_1\cdots {\bf a}_M={\bf b}_1\cdots {\bf b}_M$ be the Artin lift of $w_\circ(\{s,s'\})$. For $0\leq k\leq M$, define
\[
    \boldsymbol{\alpha}_k={\bf a}_1\cdots {\bf a}_k,\quad
    \bbeta_k={\bf b}_1\cdots {\bf b}_{M-k},\quad
    \boldsymbol{\gamma}_k={\bf b}_{M-k+1}\cdots {\bf b}_M.
\]
Applying \cref{lem:DA} to the word $({\bf a}_1,\ldots,{\bf a}_k)$ 
gives
\[
    \mathbb{\Sigma}_{j_k}\cdots \mathbb{\Sigma}_{j_1}
    =
    \B(x){\bf a}_1\cdots{\bf a}_k{\bf a}_k\cdots{\bf a}_1\B(x)^{-1}.
\]
Similarly, applying \cref{lem:DA} to the word $({\bf b}_{M-k+1},\ldots,{\bf b}_M)$ with $x$ replaced by $xb_1\cdots b_{M-k}$ gives
\[
\begin{aligned}
    \mathbb{\Sigma}_{j_M'}\cdots \mathbb{\Sigma}_{j'_{M-k+1}}
    &=
    \B(x)\bbeta_k\boldsymbol{\gamma}_k\operatorname{rev}(\boldsymbol{\gamma}_k)\bbeta_k^{-1}\B(x)^{-1} \\
    &=
    \B(x)\Delta_{\{s,s'\}}\operatorname{rev}(\boldsymbol{\gamma}_k)\bbeta_k^{-1}\B(x)^{-1}.
\end{aligned}
\]
Thus, it remains only to check that
\[
    \boldsymbol{\alpha}_k\operatorname{rev}(\boldsymbol{\alpha}_k)
    =
    \Delta_{\{s,s'\}}\operatorname{rev}(\boldsymbol{\gamma}_k)\bbeta_k^{-1}.
\]
Let $\omega$ be the automorphism of the rank-2 braid group $\B_{W_{\{s,s'\}}}$ given by conjugation by $\Delta_{\{s,s'\}}$. 
Thus, $\omega$ fixes $\s$ and $\s'$ if $M$ is even, and it interchanges them
if $M$ is odd. In particular, $\omega$ is an involution. 

Observe that $\boldsymbol{\gamma}_k=\omega(\mathrm{rev}(\boldsymbol{\alpha}_k))$ and $\mathrm{rev}(\boldsymbol{\gamma}_k)=\omega(\boldsymbol{\alpha}_k)$. Also, since
$\Delta_{\{s,s'\}}=\bbeta_k\boldsymbol{\gamma}_k$,
we have
$\bbeta_k^{-1}=\boldsymbol{\gamma}_k\Delta_{\{s,s'\}}^{-1}$.
Therefore,
\[
\begin{aligned}
\Delta_{\{s,s'\}}\mathrm{rev}(\boldsymbol{\gamma}_k)\bbeta_k^{-1}
&=
\Delta_{\{s,s'\}}\mathrm{rev}(\boldsymbol{\gamma}_k)\boldsymbol{\gamma}_k
  \Delta_{\{s,s'\}}^{-1}  \\
&=
\Delta_{\{s,s'\}}\omega(\boldsymbol{\alpha}_k)
  \omega(\mathrm{rev}(\boldsymbol{\alpha}_k))
  \Delta_{\{s,s'\}}^{-1} \\
&=
\Delta_{\{s,s'\}}\omega(\boldsymbol{\alpha}_k\mathrm{rev}(\boldsymbol{\alpha}_k))
  \Delta_{\{s,s'\}}^{-1} \\
&=
\omega^2(\boldsymbol{\alpha}_k\mathrm{rev}(\boldsymbol{\alpha}_k)) \\
&=
\boldsymbol{\alpha}_k\mathrm{rev}(\boldsymbol{\alpha}_k),
\end{aligned}
\]
as desired. 
\end{proof}

Given a reduced word ${\sf w}_\circ$ for $w_\circ$, let $\Sha_w({\sf w}_\circ)=(j_1,\ldots,j_{\ell(w)})$ be the restriction of $\Sha({\sf w}_\circ)$ to the shards contained in the inversion hyperplanes of $w$, and let \[{\bf Z}_w({\sf w}_\circ)=\mathbb{\Sigma}_{j_{\ell(w)}}\cdots\mathbb{\Sigma}_{j_1}.\] 

\begin{theorem}\label{thm:BB=Z}
Fix $w\in W$. For every reduced word $\mathsf{w}_\circ$ for $w_\circ$, we have 
\[\B(w)\B(w^{-1})={\bf Z}_{w}(\mathsf{w}_\circ).\]
\end{theorem} 
\begin{proof}
We first claim that if $\mathsf{w}_\circ$ and $\mathsf{w}_\circ'$ are two reduced words for $w_\circ$, then ${\bf Z}_{w}(\mathsf{w}_\circ)={\bf Z}_{w}(\mathsf{w}_\circ')$. To prove this, it suffices by Matsumoto's theorem \cite[Theorem~3.3.1]{BjornerBrenti} to prove it when $\mathsf{w}_\circ$ and $\mathsf{w}_\circ'$ differ by a braid move. Suppose they differ by replacing the alternating
subword
\[
    a_1a_2\cdots a_M=ss's\cdots
\quad\text{with}\quad b_1b_2\cdots b_M=s'ss'\cdots.
\]
Let $x$ be the element represented by the common prefix before this braid
move. Since both full words are reduced, the two local subwords give the two
saturated chains in the rank-$2$ interval $[x,x\,w_\circ(\{s,s'\})]$. The maximal chains corresponding to the two reduced words agree outside this rank-2 interval. Therefore, the
corresponding shard labels outside the braid move agree. 

Consider the rank-2 subsystem \[\Psi^+=x\left(\Phi^+_{\{s,s'\}}\right)\] of positive roots whose hyperplanes are crossed inside the interval
$[x,xw_\circ(\{s,s'\})]$. The two local root sequences are the two opposite
reflection orders on $\Psi^+$. The selected hyperplanes are precisely
those indexed by the roots in $\Psi^+$ that are inversions of $w$.
Since the set of inversions of $w$ is biclosed in $\Phi^+$ (see~\cite{DyerWeakOrder}), the set of roots in $\Psi^+$ that are inversions of $w$ is biclosed in $\Psi^+$. In rank~2, the biclosed subsets are exactly the initial segments (equivalently, the final segments) in the reflection orders. Therefore, it follows from \cref{lem:rank-2} that the contributions of the shards crossed in the interval $[x,xw_\circ(\{s,s'\})]$ are the same in ${\bf Z}_w(\mathsf{w}_\circ)$ and ${\bf Z}_w(\mathsf{w}_\circ')$. Thus, ${\bf Z}_w(\mathsf{w}_\circ)={\bf Z}_w(\mathsf{w}_\circ')$. 

To complete the proof of the theorem, it suffices to find one reduced word $\mathsf{w}_\circ$ for $w_\circ$ such that $\B(w)\B(w^{-1})={\bf Z}_w(\mathsf{w}_\circ)$. Choose  $\mathsf{w}_\circ$ so that it contains a reduced word $(r_{1},\ldots, r_{\ell(w)})$ of $w$ as a prefix. Let $\boldsymbol{\gamma}=(\mathbf{r}_1,\ldots,\mathbf{r}_{\ell(w)})$. Then $\boldsymbol{\gamma}\,\mathrm{rev}(\boldsymbol{\gamma})=\B(w)\B(w^{-1})$, so it follows from \cref{lem:DA} that 
\[\B(w)\B(w^{-1})=\B(e)\B(w)\B(w^{-1})\B(e)^{-1}={\bf D}_e(\boldsymbol{\gamma}){\bf U}_e(\boldsymbol{\gamma}).\] But ${\bf D}_e(\boldsymbol{\gamma})=e$, and ${\bf U}_e(\gamma)={\bf Z}_w(\mathsf{w}_\circ)$. 
\end{proof}

We can take ${\sf w}_\circ$ to be the $c$-sorting word ${\sf w}_\circ(c)$ in \cref{thm:BB=Z} to obtain the following corollary.  
\begin{corollary}\label{lem:inversions} 
For $w\in W$, we have 
\[\B(w)\B(w^{-1})=\prod_{t \in \inv_c(w)}^{\longleftarrow} \tt.\]
\end{corollary}

Note that an element $w\in W$ is $c$-sortable if and only if it is $c$-aligned if and only if $\inv_c(w) = \inv({\sf w}(c))$.

\begin{example}
Take $W=\mathfrak{S}_5$, and let $c=s_1s_2s_3s_4$, where $s_i=(i,i+1)$. Let $w=s_4s_3s_1s_2s_1$ so that $\inv(w)=((45),(35),(12),(15),(25))$. Denote each transposition $(ij)$ by $t_{(ij)}$ so that $\P_c((ij))=\tt_{(ij)}$. Then \[\B(w)\B(w^{-1}) = \s_4\s_3\s_1\s_2\s_1\s_1\s_2\s_1\s_3\s_4 =\tt_{(45)}\tt_{(35)}\tt_{(25)}\tt_{(15)}\tt_{(12)}. \qedhere\]
\end{example} 

\begin{corollary}\label{cor:wo}
Let $\inv({\sf w}_\circ(c))=(t_1,\ldots,t_N)$. Then 
$\Delta^2=\tt_N\cdots \tt_1.$
\end{corollary}

\section{Counting Subwords}\label{sec:subwords}
\subsection{Technique}
Let us briefly present the main ideas behind our technique. As before, fix a standard Coxeter element $c$ of $W$.  

For a reflection $t \in \RR$, we project the pure braid $\tt=\P_c(t)$ to the Hecke algebra to obtain the element $\H(\tt)$, which we expand in the standard basis $\{\T_w\}_{w\in W}$. We will find (see~\Cref{lem:dual_square_expansion}) that 
\begin{equation}
\H(\tt) = \sum_{w \in W} (q-1)^{\ell_T(w)}a^{c}_{t,w}(q) \T_w,\label{eq:image}\end{equation}
where each coefficient $a_{t,w}^c(q)$ is in $\mathbb Z[q^{\pm 1}]$ and $a_{t,w}^c(1)=0$ for all $w\in\NC(W,c^{-1})\setminus\{e,t\}$.  We will also see that $a_{t,e}^c(1)=a_{t,t}^c(1)=1$. 

For a sequence $\vect = (t_1,\ldots,t_M) \in \RR^M$ and a $c^{-1}$-noncrossing partition $\pi\in\NC(W,c^{-1})$, we consider the coefficient $[\T_\pi]\H(\tt_1\cdots\tt_M)$. 
As every basis element $\T_w$ in~\eqref{eq:image} has a coefficient divisible by $(q-1)^{\ell_T(w)}$, this desired coefficient of $\T_\pi$ is divisible by $(q-1)^{\ell_T(\pi)}$.  Moreover, since every reduced $T$-word for $\pi$ comes from a chain in $\NC(W,c^{-1})$, the only contributions to $(q-1)^{-\ell_T(\pi)}[\T_\pi]\H(\tt_1\cdots\tt_M)$ at $q=1$ can be those subwords of $\vect$ that are reduced $T$-words for $\pi$.  In summary, we will find that
\begin{equation}\label{eq:sub2}
\Subb_\RR(\vect,\pi)=(q-1)^{-\ell_{\RR}(\pi)}[\T_\pi]\H(\tt_1\cdots\tt_M)\Big|_{q=1}.
\end{equation} We are especially interested in this formula when $\pi=c^{-1}$. The purpose of this section is to establish~\eqref{eq:image} in \cref{lem:dual_square_expansion}, prove a
signed refinement of~\eqref{eq:sub2} in \cref{thm:main}, and then deduce~\eqref{eq:sub2} in
\cref{cor:main}.

\subsection{Divisibility}\label{sec:divisible}
Given $\protect\vv{\beta}=(s_1,\ldots,s_k)\in S^k$ and $v,w\in W$, recall that $\mathcal D^v(\protect\vv{\beta},w)$ denotes the set of $v$-distinguished $w$-subwords of $\protect\vv{\beta}$. Let \[\kappa^v(\protect\vv{\beta},w)=\min\{\mathrm{skip}(\su):\su\in\Dist^v(\vword,w)\},\] where $\mathrm{skip}(\su)$ is the number of skips of the subword $\su$. (We use the convention that $\min\emptyset=\infty$.)  

\begin{lemma}\label{lem:first}
Let $\protect\vv{\beta}=(s_1,\ldots,s_k)\in S^k$ and $v\in W$. For each $y\in W$, the coefficient 
\[[\T_{vy}](\T_v\T_{s_1}^{-1}\cdots \T_{s_k}^{-1})\] is divisible by $(q-1)^{\kappa^v(\vword,y)}$. 
\end{lemma}
\begin{proof}
Assume first that $k=0$ so that $\vword$ is the empty word $\epsilon$. The result is trivial when $y\neq e$ since then $[\T_{vy}]\T_v=0$. On the other hand, $\Dist^v(\vword,e)=\{\epsilon\}$, so $\kappa^v(\vword,e)=0$. Thus, the result follows again when $y=e$. 

Now assume $k\geq 1$, and proceed by induction on $k$. Let $\vword'=(s_1,\ldots,s_{k-1})$. We can write \begin{equation}\label{eq:Tv}
\T_v\T_{s_1}^{-1}\cdots \T_{s_{k-1}}^{-1}=\sum_{x\in W}f_x(q)\T_{vx},
\end{equation} where each $f_x(q)\in\mathbb Z[q^{\pm 1}]$ is divisible by $(q-1)^{\kappa^v(\vword',x)}$. Fix $y\in W$. If $x\in W$ is such that $s_k$ is a descent of $vx$, then $\T_{vx}\T_{s_k}^{-1}=\T_{vxs_k}$. If $x\in W$ is such that $s_k$ is not a descent of $vx$, then the Hecke algebra quadratic relation $\T_{s_k}^{-1}=q^{-1}\T_{s_k}-q^{-1}(q-1)$ implies that $\T_{vx}\T_{s_k}^{-1}=q^{-1}\T_{vxs_k}-q^{-1}(q-1)\T_{vx}$. Thus, the only terms from the sum in \eqref{eq:Tv} that can contribute a term involving $\T_{vy}$ after we multiply by $\T_{s_k}^{-1}$ on the right are those with $x=y$ and those with $x=ys_k$. In other words, 
\begin{equation}\label{eq:Tv2}
[\T_{vy}](\T_v\T_{s_1}^{-1}\cdots \T_{s_k}^{-1})=[\T_{vy}](f_{y}(q)\T_{vy}\T_{s_k}^{-1}+f_{ys_k}(q)\T_{vys_k}\T_{s_k}^{-1}).
\end{equation} 
We now consider two cases depending on whether or not $s_k$ is a descent of $vy$. 

Suppose first that $s_k$ is a descent of $vy$. Then we have \[\T_{vy}\T_{s_k}^{-1}=\T_{vys_k}\quad\text{and}\quad \T_{vys_k}\T_{s_k}^{-1}=q^{-1}\T_{vy}-q^{-1}(q-1)\T_{vys_k},\] so it follows from \eqref{eq:Tv2} that 
\[[\T_{vy}](\T_v\T_{s_1}^{-1}\cdots \T_{s_k}^{-1})=q^{-1}f_{ys_k}(q).\] This coefficient is divisible by $(q-1)^{\kappa^v(\vword',ys_k)}$. Every $v$-distinguished $y$-subword of $\vword$ must end with a take, so $\kappa^v(\vword,y)=\kappa^v(\vword',ys_k)$. This completes the proof in this case.  

Suppose next that $s_k$ is not a descent of $vy$. Then we have \[\T_{vy}\T_{s_k}^{-1}=q^{-1}\T_{vys_k}-q^{-1}(q-1)\T_{vy}\quad\text{and}\quad \T_{vys_k}\T_{s_k}^{-1}=\T_{vy},\] so it follows from \eqref{eq:Tv2} that 
\[[\T_{vy}](\T_v\T_{s_1}^{-1}\cdots \T_{s_k}^{-1})=-q^{-1}(q-1)f_{y}(q)+f_{ys_k}(q).\] This coefficient is divisible by $(q-1)^{\min\{\kappa^v(\vword',y)\}+1,\,\kappa^v(\vword',ys_k)\}}$. If $(\os_1,\ldots,\os_k)\in\Dist^v(\vword,y)$, then either we have $\os_k=e$ and $(\os_1,\ldots,\os_{k-1})\in\Dist^v(\vword',y)$ or we have $\os_k=s_k$ and $(\os_1,\ldots,\os_{k-1})\in\Dist^v(\vword',ys_k)$. Therefore, \[\kappa^v(\vword,y)=\min\{\kappa^v(\vword',y)\}+1,\,\kappa^v(\vword',ys_k)\}.\] This completes the proof in this case.  
\end{proof}

\begin{lemma}\label{lem:reflection_length_divisibility} 
Let $v_1,\ldots,v_k\in W$ and $\epsilon_1,\ldots,\epsilon_k\in\{\pm 1\}$. For each $w\in W$, the coefficient \[[\T_w](\T_{v_1}^{\epsilon_1}\cdots \T_{v_k}^{\epsilon_k})\] is divisible by $(q-1)^{\ell_{\RR}(wv_k^{-\epsilon_k}\cdots v_1^{-\epsilon_1})}$. 
\end{lemma} 

\begin{proof}
Since each basis element $\T_{v_i}$ can be expressed as a product of generators in the set $\{\T_s:s\in S\}$, it suffices to prove the lemma in the special case where $v_1,\ldots,v_k$ are all simple reflections. Thus, assume $v_i=s_i\in S$. 

The result is trivial if $k=0$, so we may assume $k\geq 1$ and proceed by induction on $k$. Let $b_{ws_{k}}=[\T_{ws_k}](\T_{s_1}^{\epsilon_1}\cdots \T_{s_{k-1}}^{\epsilon_{k-1}})$ and $b_{w}=[\T_{w}](\T_{s_1}^{\epsilon_1}\cdots \T_{s_{k-1}}^{\epsilon_{k-1}})$. We know by induction that $b_{ws_k}$ is divisible by $(q-1)^{\ell_\RR(ws_ks_{k-1}\cdots s_{1})}$. Moreover, since $ws_{k-1}\cdots s_{1}=(ws_kw^{-1})ws_k\cdots s_{1}$ and since $ws_kw^{-1}$ is a reflection, we have $\ell_\RR(ws_{k-1}\cdots s_{1})=\ell_\RR(ws_{k}\cdots s_{1})\pm 1$. We know by induction that $b_w$ is divisible by $(q-1)^{\ell_\RR(ws_{k-1}\cdots s_{1})}$, which is divisible by $(q-1)^{\ell_\RR(ws_k\cdots s_{1})-1}$. 

If $\epsilon_k=1$, then we compute that 
\[[\T_w](\T_{s_1}^{\epsilon_1}\cdots \T_{s_{k-1}}^{\epsilon_{k-1}}\T_{s_k})=\begin{cases} 
        (q-1)b_w+b_{ws_k} & \text{if } \ell(ws_k)<\ell(w) \\ 
        qb_{ws_k} & \text{if } \ell(ws_k)>\ell(w), 
    \end{cases}\] 
 and the preceding discussion ensures that this is divisible by $(q-1)^{\ell_\RR(ws_k\cdots s_1)}$. On the other hand, if $\epsilon_k=-1$, then since $\T_{s_k}^{-1}=q^{-1}\T_{s_k}-q^{-1}(q-1)$, we compute that  
\[[\T_w](\T_{s_1}^{\epsilon_1}\cdots \T_{s_{k-1}}^{\epsilon_{k-1}}\T_{s_k}^{-1})=\begin{cases} 
        q^{-1}b_{ws_k} & \text{if } \ell(ws_k)<\ell(w) \\ 
        b_{ws_k}-q^{-1}(q-1)b_w & \text{if } \ell(ws_k)>\ell(w). 
    \end{cases}\] 
 Once again, the preceding discussion ensures that this is divisible by $(q-1)^{\ell_\RR(ws_k\cdots s_1)}$.    
\end{proof} 

\subsection{Noncrossing Partition Coefficients}

A corollary of \cref{lem:reflection_length_divisibility} is that $[\T_w]\H({\brho})$ is divisible by $(q-1)^{\ell_T(w)}$ whenever $\brho$ is a pure braid. Hence, for $t\in T$, we can write \[\H(\tt)=\H(\P_c(t))=\sum_{w\in W}(q-1)^{\ell_T(w)}a_{t,w}^c(q)\T_w,\] where each coefficient $a_{t,w}^c(q)$ is in $\mathbb Z[q^{\pm 1}]$. 

\begin{lemma}\label{lem:dual_square_expansion}
Fix $t\in \RR$. If ${w\in\NC(W,c^{-1})\setminus\{e,t\}}$, then $a_{t,w}^c(q)$ is divisible by $q-1$. Moreover, \[a_{t,e}^c(1)=a_{t,t}^c(1)=1.\] 
\end{lemma}

\begin{proof}
Let $j_t$ be the unique $c$-sortable join-irreducible element of $W$ whose cover reflection is $t$. Fix a reduced word $(s_1,\ldots,s_k)$ for $j_t$. Let $\vword=(s_{k-1},\ldots,s_{1})$. For each $1\leq j\leq k$, let $r_j=s_1\cdots s_{j-1}s_js_{j-1}\cdots s_1$. By \cref{prop:dual_artin}, we have \[\H(\t)=\T_{s_1}\cdots \T_{s_{k-1}}\T_{s_k}\T_{s_{k-1}}^{-1}\cdots \T_{s_1}^{-1}=\T_{j_t}\T_{s_{k-1}}^{-1}\cdots \T_{s_1}^{-1},\] so 
\begin{align}
\nonumber \H(\tt)&=\H(\t)^2 \\ \nonumber &=(q-1)\H(\t)+q \\ 
\label{eq:q-1} &=(q-1)\T_{j_t}\T_{s_{k-1}}^{-1}\cdots \T_{s_1}^{-1}+q. 
\end{align}
Choose $w\in W$; it follows from \cref{lem:reflection_length_divisibility} that $[\T_w]\H(\tt)$ is divisible by $(q-1)^{\ell_{\RR}(w)}$. Now assume that $w\in\NC(W,c^{-1})\setminus\{e,t\}$ and that $[\T_w]\H(\tt)\neq 0$. Write $w=j_ty$. Since $[\T_{w}]\H(\tt)\neq 0$, it follows from \cref{lem:first} that $\kappa^{j_t}(\vword,y)<\infty$, so there exists a $j_t$-distinguished subword \[{\sf y}=(\os_{k-1},\ldots,\os_1)\in\Dist^{j_t}(\vword,y)\] such that $\mathrm{skip}({\sf y})=\kappa^{j_t}(\vword,y)$. Let $p_1<\cdots<p_m$ 
be the skips of ${\sf y}$, where $m=\mathrm{skip}({\sf y})$ (so $\os_{k-p_i}=e$ for all $1\leq i\leq m$). Since $y=\os_{k-1}\cdots\os_{1}$, a straightforward computation shows that 
\begin{equation}\label{eq:wproduct}
w=j_ty=t\prod_{i=1}^mr_{k-p_i}.
\end{equation} 
It follows from \cref{lem:first} and \eqref{eq:q-1} that $[\T_w]\H(\t)$ is divisible by $(q-1)^{\kappa^{j_t}(\vword,y)}$, so $[\T_w]\H(\tt)$ is divisible by $(q-1)^{\kappa^{j_t}(\vword,y)+1}$. Note that $m\geq 1$ since $w\neq t$. It follows from \cref{lem:join-irreducible} that $tr_{k-p_1}$ is not $c^{-1}$-noncrossing. This implies that the reflection factorization in \eqref{eq:wproduct} is not a reduced reflection factorization, so $\ell_{\RR}(w)<\kappa^{j_t}(\vword,y)+1$. Hence, $a_{t,w}^c(q)$ is divisible by $q-1$. 

The fact that $a_{t,e}^c(1)=1$ follows from \eqref{eq:q-1}. When we specialize $q=1$, the element $\H(\t)$ specializes to the element $t$ in the group algebra $\mathbb Z[W]$. Hence, 
\[a_{t,t}^c(1)=(q-1)^{-1}[\T_t]\H(\tt)\Big|_{q=1}=[\T_t]\H(\t)\Big|_{q=1}=1. \qedhere \]
\end{proof} 

\begin{lemma}\label{lem:dual=standard} 
Fix $\vect=(t_1,\ldots,t_M)\in T^M$ and $(\epsilon_1,\ldots,\epsilon_M)\in\{\pm 1\}^M$. For each $w\in W$, the coefficient 
\[[\T_w]\H(\tt_1^{\epsilon_1}\cdots\tt_M^{\epsilon_M})\] is divisible by $(q-1)^{\ell_{\RR}(w)}$. Moreover, for each $\pi\in\NC(W,c^{-1})$, we have 
\[(q-1)^{-\ell_{\RR}(\pi)}[\T_\pi]\H(\tt_1^{\epsilon_1}\cdots \tt_M^{\epsilon_M})\Big|_{q=1} =(q-1)^{-\ell_{\RR}(\pi)}[\T_\pi]((1+\epsilon_1(q-1)\T_{t_1})\cdots (1+\epsilon_M(q-1)\T_{t_M}))\Big|_{q=1}.\] 
\end{lemma} 

\begin{proof}
The fact that $[\T_w]\H(\tt_1^{\epsilon_1}\cdots\tt_M^{\epsilon_M})$ is divisible by $(q-1)^{\ell_{\RR}(w)}$ is a direct consequence of \cref{lem:reflection_length_divisibility}. 

For each $1\leq i\leq M$, we can expand $\H(\tt_i)$ as in \cref{lem:dual_square_expansion} as $\sum_{w\in W}(q-1)^{\ell_\RR(w)}a_{t_i,w}^c(q)\T_w$. Let us write $a_{t_i,w}^{c,(1)}(q)=a_{t_i,w}^c(q)$ and also write \[\H(\tt_i)^{-1}=\sum_{w\in W}(q-1)^{\ell_\RR(w)}a_{t_i,w}^{c,(-1)}(q)\T_w.\]
Since $\t_i$ is a braid lift of a reflection, it is
conjugate to the Artin lift of a simple reflection.  Hence, we have the
quadratic Hecke relation $\H(\tt_i)=\H(\t_i)^2=(q-1)\H(\t_i)+q$.
It follows that
\[
\H(\t_i)^{-1}=q^{-1}\H(\t_i)-q^{-1}(q-1).
\]
Thus, 
\[
\H(\tt_i)^{-1}=\H(\t_i)^{-2}
=q^{-2}(\H(\t_i)-(q-1))^2.
\]
Using the identity $\H(\t_i)^2=(q-1)\H(\t_i)+q$, this becomes
\[
\H(\tt_i)^{-1}
=q^{-2}\left(-(q-1)\H(\t_i)+q+(q-1)^2\right)
.
\]
Finally, since $\H(\tt_i)=(q-1)\H(\t_i)+q$, we can rewrite this as
\[
\H(\tt_i)^{-1}
=q^{-2}\left(-\H(\tt_i)+q^2+1\right).
\]
This implies that 
\begin{equation}\label{eq:neg} 
a_{t_i,w}^{c,(-1)}(q)=-q^{-2}a_{t_i,w}^{c,(1)}(q)\quad\text{for all }w\in W\setminus\{e\}.
\end{equation} In particular, $a_{t_i,t_i}^{c,(\epsilon_i)}(1)=\epsilon_i$ by \cref{lem:dual_square_expansion}. Moreover, \[a_{t_i,e}^{c,(-1)}(q)=q^{-2}\left(-a_{t_i,e}^{c,(1)}(q)+q^2+1\right),\] so it follows from \cref{lem:dual_square_expansion} that $a_{t_i,e}^{c,(\epsilon_i)}(1)=1$. Note that 
\[\H(\tt_1^{\epsilon_1}\cdots \tt_M^{\epsilon_M})=\sum_{(w_1,\ldots,w_M)\in W^M}(q-1)^{\ell_T(w_1)+\cdots+\ell_T(w_M)}a_{t_1,w_1}^{c,(\epsilon_1)}(q)\cdots a_{t_M,w_M}^{c,(\epsilon_M)}(q)\T_{w_1}\cdots \T_{w_M}.\]

Now suppose $\pi\in\NC(W,c^{-1})$. For each $(w_1,\ldots,w_M)\in W^M$, we can invoke \cref{lem:reflection_length_divisibility} to see that \[[\T_\pi]\left((q-1)^{\ell_\RR(w_1)+\cdots+\ell_\RR(w_M)}a_{t_1,w_1}^{c,(\epsilon_1)}(q)\cdots a_{t_M,w_M}^{c,(\epsilon_M)}(q)\T_{w_1}\cdots \T_{w_M}\right)\] is divisible by \[(q-1)^{\ell_\RR(w_1)+\cdots+\ell_\RR(w_M)+\ell_\RR(\pi w_M^{-1}\cdots w_1^{-1})}.\] We have the inequality $\ell_\RR(\pi w_M^{-1}\cdots w_1^{-1})\geq \ell_\RR(\pi)-(\ell_{\RR}(w_1)+\cdots+\ell_\RR(w_M))$. In order for this inequality to be an equality, we must have the relation $w_1,\ldots,w_M\leq_{T} \pi$ in the absolute order, implying that $w_1,\ldots,w_M\in\NC(W,c^{-1})$. Thus, in order for \[(q-1)^{-\ell_\RR(\pi)}[\T_\pi]\left((q-1)^{\ell_\RR(w_1)+\cdots+\ell_\RR(w_M)}a_{t_1,w_1}^{c,(\epsilon_1)}(q)\cdots a_{t_M,w_M}^{c,(\epsilon_M)}(q)\T_{w_1}\cdots \T_{w_M}\right)\Big|_{q=1}\] to be nonzero, we must have $w_1,\ldots,w_M\in\NC(W,c^{-1})$, and the product $a_{t_1,w_1}^{c,(\epsilon_1)}(q)\cdots a_{t_M,w_M}^{c,(\epsilon_M)}(q)$ must not be divisible by $q-1$. According to \cref{lem:dual_square_expansion} and \eqref{eq:neg}, this would require $w_i\in\{e,t_i\}$ for all $1\leq i\leq M$. Hence, $(q-1)^{-\ell_{\RR}(\pi)}[\T_\pi]\H(\tt_1^{\epsilon_1}\cdots\tt_M^{\epsilon_M})\big|_{q=1}$ is equal to 
\[(q-1)^{-\ell_{\RR}(\pi)}[\T_\pi]\left(\left(a_{t_1,e}^{c,(\epsilon_1)}(q)+(q-1)a_{t_1,t_1}^{c,(\epsilon_1)}(q)\T_{t_1}\right)\cdots \left(a_{t_M,e}^{c,(\epsilon_M)}(q)+(q-1)a_{t_M,t_M}^{c,(\epsilon_M)}(q)\T_{t_M}\right)\right)\Big|_{q=1}.\] The desired identity now follows since $a_{t_i,e}^{c,(\epsilon_i)}(1)=1$ and $a_{t_i,t_i}^{c,(\epsilon_i)}(1)=\epsilon_i$ for all $1\leq i\leq M$.  
\end{proof}

In the following lemma, we must use multiple different standard Coxeter elements, so we use the notation $\B_c(t)$ and $\P_c(t)$ instead of $\t$ and $\tt$.

\begin{lemma}\label{lem:conjugation}
Fix $\vect=(t_1,\ldots,t_M)\in T^M$ and $(\epsilon_1,\ldots,\epsilon_M)\in\{\pm 1\}^M$. Fix a standard Coxeter element $c$ of $W$, and choose $s\in S$ such that $\ell(sc)<\ell(c)$. Let $\pi\in\NC(W,c^{-1})$. Then 
\begin{align*}
&(q-1)^{-\ell_{\RR}(\pi)}[\T_\pi](\H(\P_c(t_1))^{\epsilon_1}\cdots \H(\P_c(t_M))^{\epsilon_M})\Big|_{q=1} \\ =&(q-1)^{-\ell_{\RR}(s\pi s)}[\T_{s\pi s}](\H(\P_{scs}(st_1s))^{\epsilon_1}\cdots \H(\P_{scs}(st_Ms))^{\epsilon_M})\Big|_{q=1}. \end{align*}
\end{lemma} 
\begin{proof}
Choose a reduced word $(s_1,s_2,\ldots,s_r)$ for $c$, where $s_1=s$. Then $(s_2,\ldots,s_r,s_1)$ is a reduced word for $scs$. By definition, $\B_c(t_i)$ is the unique lift of $t_i$ that appears in a word in the Hurwitz orbit of $(\s_1,\s_2,\ldots,\s_r)$. Note that the word \[(\Hur_{r-1}\circ\cdots\circ\Hur_{1})(\s_1,\s_2,\ldots,\s_r)=(\s_1\s_2\s_1^{-1},\ldots,\s_1\s_r\s_1^{-1},\s_1)\] is obtained from $(\s_2,\ldots,\s_r,\s_1)$ by conjugating each entry by $\s_1$. Now, $\B_c(t_i)$ is the unique lift of $t_i$ that appears in a word in the Hurwitz orbit of $(\s_1\s_2\s_1^{-1},\ldots,\s_1\s_r\s_1^{-1},\s_1)$. Moreover, $\B_{scs}(st_is)$ is the unique lift of $st_is$ that appears in a word in the Hurwitz orbit of $(\s_2,\ldots,\s_r,\s_1)$. Every Hurwitz move in the braid group commutes with the operation of simultaneously conjugating all entries in a tuple by $\s$. Therefore, $\B_{scs}(st_is)=\s^{-1}\B_c(t_i)\s$. It follows that $\H(\B_{scs}(st_is))=\T_s^{-1}\H(\B_c(t_i))\T_s$ for all $1\leq i\leq M$. Invoking \cref{lem:coeff-identity} and the fact that traces are invariant under conjugation, we find that 
\begin{align}
\nonumber [\T_{s\pi s}]\H\!\left(\P_{scs}(st_1s)^{\epsilon_1}\cdots \P_{scs}(st_Ms)^{\epsilon_M}\right)&=[\T_{s\pi s}]\left(\T_s^{-1}\H(\P_{c}(t_1)^{\epsilon_1}\cdots \P_c(t_M)^{\epsilon_M})\T_s\right) \\ 
\nonumber &=q^{-\ell(s\pi s)}[\T_{e}]\left(\T_{s\pi^{-1}s}\T_s^{-1}\H(\P_{c}(t_1)^{\epsilon_1}\cdots {\P_c(t_M)}^{\epsilon_M}])\T_s\right) \\ 
\label{eq:u}&=q^{-\ell(s\pi s)}[\T_{e}]\left(\T_s\T_{s\pi^{-1}s}\T_s^{-1}\H({\P_{c}(t_1)}^{\epsilon_1}\cdots {\P_c(t_M)}^{\epsilon_M})\right). 
\end{align}
In addition,  
\begin{equation}\label{eq:t}
[\T_\pi]\H({\P_{c}(t_1)}^{\epsilon_1}\cdots {\P_c(t_M)}^{\epsilon_M})=q^{-\ell(\pi)}[\T_e](\T_{\pi^{-1}}\H({\P_{c}(t_1)}^{\epsilon_1}\cdots {\P_c(t_M)}^{\epsilon_M})).
\end{equation}
It is straightforward to check that we can write \begin{equation}\label{eq:conjugation_expand}
\T_s\T_{s\pi^{-1}s}\T_s^{-1}=\sum_{z\in Z}b_z(q)\T_z,\end{equation} where $Z=\{\pi^{-1},s\pi^{-1},\pi^{-1}s,s\pi^{-1}s\}$ and each $b_z(q)$ is an element of $\mathbb Z[q^{\pm 1}]$. (It is possible that $\pi$ and $s$ commute, in which case $Z=\{\pi^{-1},s\pi^{-1}\}$.) By specializing $q=1$ in \eqref{eq:conjugation_expand}, we find that $b_{\pi^{-1}}(1)=1$ and that $b_z(q)$ is divisible by $q-1$ for each $z\in Z\setminus\{\pi^{-1}\}$. We claim that for each $z\in Z\setminus\{\pi^{-1}\}$, the coefficient $b_z(q)$ is divisible by $(q-1)^{\ell_\RR(\pi)-\ell_{\RR}(z)+1}$. If we can prove this claim, then it will follow from \cref{lem:reflection_length_divisibility} that 
\[(q-1)^{-\ell_\RR(\pi)}[\T_e]\left(b_z(q)\T_z\H({\P_c(t_1)}^{\epsilon_1}\cdots {\P_c(t_M)}^{\epsilon_M})\right)\Big|_{q=1}=0\] for all $z\in Z\setminus\{\pi^{-1}\}$. Since $\ell_\RR(s\pi s)=\ell_\RR(\pi)$, it will then follow from \eqref{eq:u}, \eqref{eq:t}, and \eqref{eq:conjugation_expand} that 
\begin{align*}
&\hphantom{=}\hspace{0.2cm}(q-1)^{-\ell_\RR(s\pi s)}[\T_{s\pi s}](\H({\P_{scs}(st_1s)}^{\epsilon_1}\cdots {\P_{scs}(st_Ms)}^{\epsilon_M}))\Big|_{q=1} \\ &=q^{-\ell(s\pi s)}(q-1)^{-\ell_\RR(\pi)}[\T_{e}](\T_s\T_{s\pi^{-1}s}\T_s^{-1}\H({\P_{c}(t_1)}^{\epsilon_1}\cdots {\P_c(t_M)}^{\epsilon_M}))\Big|_{q=1} \\ &= q^{-\ell(s\pi s)}(q-1)^{-\ell_\RR(\pi)}[\T_{e}](b_{\pi^{-1}}
(q)\T_{\pi^{-1}}\H({\P_{c}(t_1)}^{\epsilon_1}\cdots {\P_c(t_M)}^{\epsilon_M}))\Big|_{q=1} \\ &= q^{-\ell(s\pi s)}(q-1)^{-\ell_\RR(\pi)}[\T_{e}](\T_{\pi^{-1}}\H({\P_{c}(t_1)}^{\epsilon_1}\cdots {\P_c(t_M)}^{\epsilon_M}))\Big|_{q=1} \\ &=q^{\ell(\pi)-\ell(s\pi s)}(q-1)^{-\ell_{\RR}(\pi)}[\T_\pi]\H({\P_{c}(t_1)}^{\epsilon_1}\cdots {\P_c(t_M)}^{\epsilon_M})\Big|_{q=1} \\  &=(q-1)^{-\ell_{\RR}(\pi)}[\T_\pi]\H({\P_{c}(t_1)}^{\epsilon_1}\cdots {\P_c(t_M)}^{\epsilon_M})\Big|_{q=1}, 
\end{align*}
as desired. 

We now prove the claim that for each $z\in Z\setminus\{\pi^{-1}\}$, the coefficient $b_z(q)$ is divisible by $(q-1)^{\ell_\RR(\pi)-\ell_\RR(z)+1}$.
We have already seen that every coefficient $b_z(q)$ with $z\in Z\setminus\{\pi^{-1}\}$ is divisible by $q-1$, so we may assume $\ell_\RR(z)\leq\ell_\RR(\pi)-1$. Since $\ell_\RR(s\pi^{-1}s)=\ell_\RR(\pi^{-1})=\ell_\RR(\pi)$, this assumption forces $z$ to be either $s\pi^{-1}$ or $\pi^{-1}s$. 

We have $\ell_\RR(s\pi^{-1})=\ell_\RR(\pi^{-1}s)=\ell_\RR(\pi^{-1})-1=\ell_\RR(\pi)-1$. Hence, $\ell_\RR(\pi s)=\ell_\RR(\pi)-1$. Let $y=c^{-1}\pi^{-1}$ so that $c^{-1}=y\pi$ and $\ell_\RR(c^{-1})=\ell_{\RR}(y)+\ell_\RR(\pi)$ (because $\pi\in\NC(W,c^{-1})$). Let $J=S\setminus\{s\}$, and let $W_J$ be the standard parabolic subgroup of $W$ generated by $J$. We have assumed that $s$ is a left descent of $c$, so $c^{-1}s$ is a standard Coxeter element of $W_J$. We have \[c^{-1}s=y(\pi s)\quad\text{and}\quad\ell_\RR(c^{-1}s)=\ell_\RR(c^{-1})-1=\ell_{\RR}(y)+\ell_\RR(\pi)-1=\ell_{\RR}(y)+\ell_\RR(\pi s),\] so $\pi s\in\NC(W_J,c^{-1}s)$. In particular, $\pi s$ is in $W_J$, so its inverse $s\pi^{-1}$ is in $W_J$ as well. This implies that $\ell(\pi^{-1})=\ell(s\pi^{-1}s)=\ell(s\pi^{-1})+1$. As a result, $\T_s\T_{s\pi^{-1}s}\T_s^{-1}=\T_{\pi^{-1}}$, so $b_z(q)$ is $0$, which is certainly divisible by $(q-1)^{\ell_\RR(\pi)-\ell_\RR(z)+1}$. 
\end{proof}

\subsection{Signed Subwords}
We now prove our main theorem. Recall the definition of the signed subword count $\Subb_\RR^{\protect\vv\epsilon}(\protect\vv{t},\pi)$ from \eqref{eq:sub}. 

\begin{theorem}\label{thm:main}
Fix a sequence $\vect=(t_1,\ldots,t_M)\in T^M$ of reflections and the corresponding sequence ${\vectt=(\tt_1,\ldots,\tt_M)\in\mathbb T^M}$ of squares of $c$-dual lifts. Fix a sign vector $\vece=(\epsilon_1,\ldots,\epsilon_M)\in\{\pm 1\}^M$.
For each $\pi\in\NC(W,c^{-1})$, we have 
\[(q-1)^{-\ell_{\RR}(\pi)}[\T_\pi]\H(\tt_1^{\epsilon_1}\cdots \tt_M^{\epsilon_M})\Big|_{q=1}=\Subb_\RR^{\protect\vv\epsilon}(\protect\vv{t},\pi).\]
\end{theorem}

\begin{proof}
If $M=1$, then the result follows from \cref{lem:dual_square_expansion}. Thus, we may assume $M\geq 2$ and proceed by induction on $M$. 

Consider simple reflections $s_1,\ldots,s_k$, and for $0\leq j\leq k$, let $v_j=s_1\cdots s_j$. In \cite[Lemma 1.7]{reading2007clusters} (see also  \cite[Theorem 1.2]{bernstein1973coxeter}), Reading
proved that we can choose $s_1,\ldots,s_k$ so that the following hold: 
\begin{itemize}
\item For all $0\leq j\leq k-1$, the Coxeter element $c_{j}:=v_j^{-1}c v_j$ has $s_{j+1}$ as a left descent. 
\item The reflection $v_k^{-1}t_M v_k$ is simple. 
\end{itemize}
These conditions ensure that for all $0 \leq j\leq k$, the Coxeter element $c_j$ is standard and the element $v_j^{-1}\pi v_j$ belongs to $\NC(W,c_j^{-1})$. Because reflection length is invariant under conjugation, we have $\ell_{\RR}(v_j^{-1}\pi v_j)=\ell_{\RR}(\pi)$ for all $j$. Moreover, if we let $\vect^{(j)}=(v_j^{-1}t_1v_j,\ldots,v_j^{-1}t_M v_j)$, then \[\Subb_\RR^{\protect\vv\epsilon}(\vect^{(j)},v_j^{-1}\pi v_j)=\Subb_\RR^{\protect\vv\epsilon}(\protect\vv{t},\pi)\] since conjugation by $v_j$ provides a bijection from $\Sub_\RR(\vect^{(j)},v_j^{-1}\pi v_j)$ to $\Sub_\RR(\protect\vv{t},\pi)$. Combined with \cref{lem:conjugation}, this discussion shows that it suffices to prove the desired result when $t_M$ is a simple reflection. Thus, we assume in what follows that $t_M=s\in S$.   

Let $\vect\,'=(t_1,\ldots,t_{M-1})$ and $\protect\vv\epsilon\,'=(\epsilon_1,\ldots,\epsilon_{M-1})$. According to \cref{lem:dual=standard}, we have \begin{align*}
&\hphantom{=}\hspace{0.15cm}(q-1)^{-\ell_{\RR}(\pi)}[\T_\pi]\H(\tt_1^{\epsilon_1}\cdots \tt_M^{\epsilon_M})\Big|_{q=1} \\ &=(q-1)^{-\ell_{\RR}(\pi)}[\T_\pi]\left((1+\epsilon_1(q-1)\T_{t_1})\cdots (1+\epsilon_M(q-1)\T_{t_M})\right)\Big|_{q=1} \\  
&=(q-1)^{-\ell_{\RR}(\pi)}[\T_\pi]\left((1+\epsilon_1(q-1)\T_{t_1})\cdots(1+\epsilon_{M-1}(q-1)\T_{t_{M-1}})\right)\Big|_{q=1} \\ &\hspace{0.4cm}+\epsilon_M(q-1)^{-\ell_{\RR}(\pi)+1}[\T_\pi](1+\epsilon_1(q-1)\T_{t_1})\cdots(1+\epsilon_{M-1}(q-1)\T_{t_{M-1}})\T_s\Big|_{q=1}.
\end{align*}
We know by \cref{lem:dual=standard} and induction on $M$ that 
\[(q-1)^{-\ell_{\RR}(\pi)}[\T_\pi]\left((1+\epsilon_1(q-1)\T_{t_1})\cdots(1+\epsilon_{M-1}(q-1)\T_{t_{M-1}})\right)\Big|_{q=1}=\Subb_\RR^{\protect\vv\epsilon\,'}(\vect\,',\pi),\] and this is the signed count of subwords of $\vect$ that are reduced $\RR$-words for $\pi$. Therefore, we must show that \begin{equation}\label{eq:main_theorem_1}
\epsilon_M(q-1)^{-\ell_{\RR}(\pi)+1}[\T_\pi]\left((1+\epsilon_1(q-1)\T_{t_1})\cdots(1+\epsilon_{M-1}(q-1)\T_{t_{M-1}})\T_s\right)\Big|_{q=1} 
\end{equation} is the signed count of subwords of $\vect$ that are reduced $\RR$-words for $\pi$ and that end with takes. Let $b_\pi(q)$ and $b_{\pi s}(q)$ be the coefficients of $\T_\pi$ and $\T_{\pi s}$, respectively, in \[(1+\epsilon_1(q-1)\T_{t_1})\cdots(1+\epsilon_{M-1}(q-1)\T_{t_{M-1}}).\] 
It follows from \cref{lem:dual=standard} that $b_\pi(q)$ is divisible by $(q-1)^{\ell_{\RR}(\pi)}$ and that $b_{\pi s}(q)$ is divisible by $(q-1)^{\ell_{\RR}(\pi s)}$. 
Because $s$ is simple, the defining relations of the Hecke algebra imply that 
\[[\T_\pi]\left((1+\epsilon_1(q-1)\T_{t_1})\cdots (1+\epsilon_{M-1}(q-1)\T_{t_{M-1}})\T_s\right)=b_\pi(q)[\T_\pi](\T_\pi \T_s)+b_{\pi s}(q)[\T_{\pi}](\T_{\pi s}\T_s).\] Hence, the expression in \eqref{eq:main_theorem_1} is equal to \[\epsilon_M(q-1)^{-\ell_{\RR}(\pi)+1}b_\pi(q)[\T_\pi](\T_\pi \T_s)\Big|_{q=1}+\epsilon_M(q-1)^{-\ell_{\RR}(\pi)+1}b_{\pi s}(q)[\T_\pi](\T_{\pi s}\T_s)\Big|_{q=1}.\] 
The first term in this sum is $0$ because $b_\pi(q)$ is divisible by $(q-1)^{\ell_{\RR}(\pi)}$. In addition, the coefficient $[\T_\pi](\T_{\pi s}\T_s)$ is either $1$ or $q$, depending on whether $s$ is a descent of $\pi$ or not. Consequently, the expression in \eqref{eq:main_theorem_1}
 is equal to 
\begin{equation}\label{eq:main_theorem_2}\epsilon_M(q-1)^{-\ell_{\RR}(\pi)+1}b_{\pi s}(q)\Big|_{q=1}.
 \end{equation} 
If $\ell_{\RR}(\pi s)\geq\ell_{\RR}(\pi)$, then this is $0$ because $b_{\pi s}(q)$ is divisible by $(q-1)^{\ell_{\RR}(\pi s)}$. In this case, the proof is complete because there are no reduced $\RR$-words for $\pi$ that end with the letter $t_M=s$. On the other hand, if $\ell_{\RR}(\pi s)=\ell_{\RR}(\pi)-1$, then it follows from \cref{lem:dual=standard} and induction on $M$ that the expression in \eqref{eq:main_theorem_2} is $\epsilon_M\,\Subb_\RR^{\protect\vv\epsilon\,'}(\vect\,',\pi s)$, which is also the signed count of subwords of $\vect$ that are reduced $\RR$-words for $\pi$ and that end with takes. 
\end{proof} 

Setting $\protect\vv\epsilon=(1,\ldots,1)$ in \cref{thm:main} yields the following corollary. 

\begin{corollary}\label{cor:main}
Fix a sequence $\vect=(t_1,\ldots,t_M)\in T^M$ of reflections and the corresponding sequence ${\vectt=(\tt_1,\ldots,\tt_M)\in\mathbb T^M}$ of squares of $c$-dual lifts. For each $\pi\in\NC(W,c^{-1})$, we have 
\[(q-1)^{-\ell_{\RR}(\pi)}[\T_\pi]\H(\tt_1\cdots \tt_M)\Big|_{q=1}=\Subb_\RR(\protect\vv{t},\pi).\]
\end{corollary}

\begin{remark}
Consider the set $\bigcup_{c}\NC(W,c^{-1})$, where the union is over all standard Coxeter elements $c$ of $W$. If $\pi\in\bigcup_{c}\NC(W,c^{-1})$, then an immediate consequence of \cref{lem:dual=standard,thm:main} is that 
\[\Subb_\RR^{\protect\vv\epsilon}(\protect\vv{t},\pi)=(q-1)^{-\ell_\RR(\pi)}[\T_\pi]\left((1+\epsilon_1(q-1)\T_{t_1})\cdots(1+\epsilon_M(q-1)\T_{t_M})\right)\Big|_{q=1}\] for all $\protect\vv{t}=(t_1,\ldots,t_M)\in \RR^M$ and $\protect\vv\epsilon=(\epsilon_1,\ldots,\epsilon_M)\in\{\pm 1\}^M$. This is quite strange, especially because this conclusion does not necessarily hold when $\pi\not\in\bigcup_{c}\NC(W,c^{-1})$. As an example, suppose $W=\mathfrak{S}_4$, and let $s_i\in \mathfrak{S}_4$ be the transposition $(i\,\,i+1)$. Let $\nu=s_1s_2s_3s_1s_2=(1324)$. Since $\nu$ is a non-standard Coxeter element, it is not in $\bigcup_{c}\NC(W,c^{-1})$. For $t_1=(14)$ and $t_2=(13)$, we have 
\[[\T_\nu]\left((1+(q-1)\T_{t_1})(1+(q-1)\T_{t_2})\right)=q(q-1)^3,\] 
so
\[(q-1)^{-\ell_\RR(\nu)}[\T_\nu]\left((1+(q-1)\T_{t_1})(1+(q-1)\T_{t_2})\right)\Big|_{q=1}=1,\] even though there are no subwords of $(t_1,t_2)$ with product $\nu$.  
\end{remark} 

\newpage

\part{Applications}\label{part:applications}

In this part of the paper, we study applications of our technique to algebraic combinatorics and knot invariants.  Recall that the \defn{rational Catalan numbers} and the \defn{rational parking numbers} for a finite Coxeter group $W$ are defined for a positive integer $p$ coprime to the Coxeter number $h$ as \begin{align}\Cat_p(W;q)&:=\prod_{i=1}^r \frac{[p+(p e_i\!\!\! \mod h)]_q}{[d_i]_q} &\text{ and }&&\Cat_p(W)&:=\Cat_p(W;1), \\ 
\Park_p(W;q)&:=[p]_q^r &\text{ and }& &\Park_p(W)&:=p^r.\end{align}
When $W$ is crystallographic, $\{pe_i \!\!\mod h\}_{i=1}^r= \{e_i\}_{i=1}^r$.   

\begin{itemize}
    \item In~\Cref{sec:noncrossing_lattice}, we use the full twist $\Delta^2$ to show that the noncrossing partition lattices are EL-shellable when edges are naturally labeled by reflections for any total ordering $\vect$ of $T$ such that $\Delta^2=\prod_{t \in \vect} \tt$.  We then use $\Delta^{-2}$ to compute the homotopy type of the order complex of $\NC(W,c^{-1})$ as the positive Catalan number $\Cat^+(W):=\Cat_{h-1}(W)$. 
    \item In~\Cref{sec:fuss}, we use the braids $\B(c)^{mh\pm1}$ to produce combinatorial models for the Fuss--Catalan ($p=mh+1$) and Fuss--Dogolon ($p=mh-1$) numbers.  The reason these models are so simple is that $\B(c)^{mh\pm1} = (\Delta^2)^m \B(c)^{\pm 1}$ is very close to a power of the full twist.
    \item By way of contrast, for $p\neq \pm 1 \pmod h$, the braid $\B(c)^p$ is not very close to a power of $\Delta^2$ and we need further tricks.  In~\Cref{sec:rational}, we use the fact that the standard Coxeter element $c$ and the (generally nonstandard) Coxeter element $c^p$ are conjugate in crystallographic type to produce rational models.
    \item In~\Cref{sec:parking}, we show how to extend our rational Catalan models to rational parking models.
    \item In \cref{sec:knot}, we give combinatorial interpretations of the $z=0$
specialization of the HOMFLYPT polynomial. 
\end{itemize}

\section{Noncrossing Partition Lattices}\label{sec:noncrossing_lattice}

In this section, we use our techniques applied to the full twist $\Delta^2$ written as a product of squares of dual lifts of reflections to recover and generalize the classic theorem of Athanasiadis--Brady--Watt that the noncrossing partition lattice $\NC(W,c^{-1})$ is EL-shellable~\cite{ABW}.  We further use $\Delta^{-2}$ to compute the homotopy type of the order complex of its proper part.

\subsection{EL-shelling Orders}
Suppose $P$ is a graded poset and $\lambda\colon \mathcal E(P)\to\Lambda$ is an edge-labeling, where $\mathcal E(P)$ is the set of edges in the Hasse diagram of $P$ and $\Lambda$ is a set. Let $\preceq$ be a total order on $\Lambda$. We say $\lambda$ is an \dfn{EL-labeling} with respect to $\preceq$ if every non-singleton interval has a unique $\preceq$-increasing maximal chain and this chain is
$\preceq$-lexicographically first among all maximal chains in the interval. 

Recall that the noncrossing partition lattice $\NC(W,c^{-1})$ has the natural reflection labeling that sends each cover relation $x\lessdot_{T}y$ to the reflection $x^{-1}y\in T$. Say a total order $\preceq$ on $T$ is an \dfn{EL-shelling order} if the natural reflection labeling is an EL-labeling with respect to $\preceq$. We can represent a total order $\preceq$ on $T$ as a word $(t_1,\ldots,t_N)$ in which each reflection appears exactly once (so $t_i\preceq t_j$ if and only if $i\leq j$). 

As usual, we fix the standard Coxeter element $c$ and, for $t\in T$, let $\t=\B_c(t)$ and $\tt=\P_c(t)$.

\begin{lemma}\label{lem:1word}
Let $\protect\vv{t}=(t_1,\ldots,t_N)$ be a $T$-word such that $\tt_1\cdots\tt_N=\Delta^2$. For every
$\pi\in \NC(W,c^{-1})$, we have 
\[
\Subb_T(\protect\vv{t},\pi)=1.
\]
\end{lemma}

\begin{proof}
We first prove the lemma in a special case. Let $d$ be a standard Coxeter element, let $\protect\vv{a}=(a_1,\ldots,a_N)$ be a $T$-word such that $\P_d(a_1)\cdots \P_d(a_N)=\Delta^2$, and let $\sigma\in \NC(W,d^{-1})$ be a standard Coxeter element of a standard parabolic subgroup $W_J$. Write $k=|J|=\ell_T(\sigma)$, and write $J=\{s_1,\ldots,s_k\}$ so that $\sigma=s_1\cdots s_k$. Then $(s_k,\ldots,s_1)$ is a reduced word for $\sigma^{-1}$.

By \cref{cor:main},
\[
\Subb_T(\protect\vv{a},\sigma)
=
(q-1)^{-k}[\T_\sigma]\H(\Delta^2)\Big|_{q=1}.
\]
Using \cref{lem:coeff-identity,lem:coeff-delta}, we compute
\[
[\T_\sigma]\H(\Delta^2)
=
q^{-k}[\T_e]\H(\Delta^2) \T_{\sigma^{-1}}
=
(-1)^k q^N [\T_e]\H(\B(\sigma^{-1})^{-1}).
\]
For each simple reflection $s$, the quadratic relation gives $\T_s^{-1}=q^{-1}\T_s-q^{-1}(q-1)$. In the product $\T_{s_1}^{-1}\cdots \T_{s_k}^{-1}$, the only contribution to the coefficient of $\T_e$ comes from choosing the scalar term $-q^{-1}(q-1)$ in every factor. Hence,
\[
[\T_e]\H(\B(\sigma^{-1})^{-1})
=
\bigl(-q^{-1}(q-1)\bigr)^k
=
(-1)^kq^{-k}(q-1)^k.
\]
It follows that $[\T_\sigma]\H(\Delta^2)=q^{N-k}(q-1)^k$, and therefore $\Subb_T(\protect\vv{a},\sigma)=1$. 

We now reduce the general case to the special case above. According to \cite[Lemma 1.7]{reading2007clusters} (see also  \cite[Theorem 1.2]{bernstein1973coxeter}), there exist simple reflections $u_1,\ldots,u_M$ such that, if we set $c_0=c$ and $\pi_0=\pi$, and then define $c_j=u_jc_{j-1}u_j$ and $\pi_j=u_j\pi_{j-1}u_j$ for $1\le j\le M$, then each $u_j$ is an initial simple reflection of $c_{j-1}$, each $c_j$ is standard, each $\pi_j$ lies in $\NC(W,c_j^{-1})$, and $\pi_M$ is a standard Coxeter element of a standard parabolic subgroup. 

Set $t_i^{(0)}=t_i$, and for $1\le j\le M$, let $t_i^{(j)}=u_jt_i^{(j-1)}u_j$. Let $\protect\vv{t}^{\,(j)}=(t_1^{(j)},\ldots,t_N^{(j)})$. We claim that $\P_{c_j}(t^{(j)}_1)\cdots \P_{c_j}(t^{(j)}_N)=\Delta^2$ for every $j$. Indeed, the compatibility of dual lifts under rotation, which we saw in the proof of \cref{lem:conjugation}, gives
\[
\P_{c_j}(u_jtu_j)=\B(u_j)^{-1}\P_{c_{j-1}}(t)\B(u_j)
\]
for every reflection $t\in T$. Thus,
\[
\P_{c_j}(t^{(j)}_1)\cdots \P_{c_j}(t^{(j)}_N)
=
\B(u_j)^{-1}
\bigl(\P_{c_{j-1}}(t^{(j-1)}_1)\cdots \P_{c_{j-1}}(t^{(j-1)}_N)\bigr)
\B(u_j).
\]
By induction, this is $\B(u_j)^{-1}\Delta^2\B(u_j)=\Delta^2$, since $\Delta^2$ is central. 

Finally, for each $j$, conjugation by $u_j$ gives a bijection
$\Sub_T(\protect\vv{t}^{\,(j-1)},\pi_{j-1})\to \Sub_T(\protect\vv{t}^{\,(j)},\pi_j)$. Conjugation preserves reflection length, so this bijection preserves reduced reflection factorizations. Therefore,
$\Subb_T(\protect\vv{t},\pi)=\Subb_T(\protect\vv{t}^{\,(M)},\pi_M)$. The right-hand side is $1$ by the special case already proved applied with $d=c_M$, $\protect\vv{a}=\protect\vv{t}^{\,(M)}$, and $\sigma=\pi_M$. 
\end{proof}

\begin{theorem}\label{thm:EL}
Let $\vect=(t_1,\ldots,t_N)$ be a total order on $T$. If $\tt_1\cdots\tt_N=\Delta^2$, then $(t_1,\ldots,t_N)$ is an EL-shelling order of $\NC(W,c^{-1})$.  
\end{theorem}

\begin{proof}
Let $\preceq$ be the total order on $T$ defined by $t_1\prec \cdots\prec t_N$. First consider a lower interval $[e,\pi]$ in $\NC(W,c^{-1})$. A maximal
chain $e=x_0\lessdot x_1\lessdot \cdots\lessdot x_k=\pi$ in this interval has label word
\[
(x_0^{-1}x_1,\ x_1^{-1}x_2,\ \ldots,\ x_{k-1}^{-1}x_k),
\]
which is a reduced $T$-word for $\pi$. This chain is $\preceq$-increasing if and only if its label word appears as a
subword of $\protect\vv{t}$. By \cref{lem:1word}, there is exactly one such
subword. This proves that every lower interval $[e,\pi]$ has a unique $\preceq$-increasing maximal
chain.

Now let $[u,v]$ be an arbitrary interval in $\NC(W,c^{-1})$. Left
multiplication by $u$ gives a poset isomorphism from $[e,u^{-1}v]$ to $[u,v]$ that preserves the natural reflection labeling, so $[u,v]$ also has a unique $\preceq$-increasing maximal chain.

It remains to check the lexicographic condition.  A simple argument given by Athanasiadis--Brady--Watt~\cite[Theorem~3.5(i)]{ABW} shows that for any total order on $T$, the lexicographically first maximal chain in every interval is increasing. Since we have already shown that the $\preceq$-increasing maximal chain is unique, it must be the lexicographically
first maximal chain. Thus, $(t_1,\ldots,t_N)$ is an EL-shelling order. 
\end{proof} 

A parabolic subgroup is \dfn{$c^{-1}$-noncrossing} if it is the smallest parabolic subgroup containing some $c^{-1}$-noncrossing partition. 
Say that a total order on $T$ is \defn{$c^{-1}$-cyclic} if, for every rank-2 $c^{-1}$-noncrossing parabolic subgroup $W'$, the reflections in $W'$ appear in the same cyclic order as they do in $\inv(\mathsf{w}_\circ(c^{-1}))$.  H.~Thomas has given an inductive proof---modulo base cases in exceptional types, including a full check of $E_8$---that a total order on $T$ is an EL-shelling order for $\NC(W,c^{-1})$ if and only if it is $c^{-1}$-cyclic, nearly resolving a conjecture of the second author.    On the other hand, it is clear that the reduced words for the full twist in the squares of the dual generators give reflection orders that are $c^{-1}$-cyclic, or equivalently, whose reverses are $c$-cyclic.
We believe this is good evidence that the converse of \cref{thm:EL} should hold.

\begin{conjecture}
Let $\protect\vv{t}=(t_1,\ldots,t_N)$ be a total order on $T$. Then $\tt_1\cdots\tt_N=\Delta^2$ if and only if $(t_1,\ldots,t_N)$ is an EL-shelling order of $\NC(W,c^{-1})$.
\end{conjecture}

\subsection{Homotopy Type}
A computation similar to the one above gives the homotopy type of the order complex of the proper part $\overline{\NC}(W,c^{-1})=\NC(W,c^{-1})\setminus\{e,c^{-1}\}$ (see also~\Cref{sec:fuss}, where a similar computation is interpreted in terms of full-support clusters). The \defn{positive Catalan numbers} for $W$ are defined as \[\Cat^+(W;q):=\prod_{i=1}^r \frac{[h-1+e_i]_q}{[d_i]_q} \quad\text{and}\quad\Cat^+(W):=\Cat^+(W;1).\]

\begin{theorem}[{\cite[Corollary 4.3]{ABW},\cite{brady2008non}}]\label{thm:homotopy_type}
The order complex of $\overline{\NC}(W,c^{-1})$ is homotopic to a wedge of $\Cat^+(W)$ many spheres of dimension $r-2$.
\end{theorem}
\begin{proof}
We consider the braid $\bbeta=\B(c)^{-h+1}=\Delta^{-2} \B(c)$.  From~\cite{galashin2024rational}, we have that \[[\T_e]\H(\B(c)^{-h+1})  = (-q)^{-r(h-1)}(q-1)^r \Cat^+(W;q).\]  
On the other hand, we have \[[\T_e]\H(\B(c)^{-h+1}) = [\T_e] \H(\Delta^{-2} \B(c)) = q^r [\T_{c^{-1}}] \H(\Delta^{-2}),\] so writing $\Delta^{-2}$ as a product of 
inverses of elements of $\mathbb{T}_c$,
we conclude by~\Cref{thm:main} that there are $\Cat^+(W)$ decreasing chains in any word for the full twist in these generators.  These decreasing chains index homotopy spheres of dimension $r-2$ and hence give the homotopy type of the order complex of $\overline{\NC}(W,c^{-1})$~\cite{bjorner1996shellable}.
\end{proof}

More generally, the full twist $\Delta^2\B(c^{-1})$ gives us a $q$-analogue of the zeta function for the noncrossing partition lattice in the following way.  Write $\zeta_{\NC(W,c^{-1})}$ for the incidence algebra zeta function of $\NC(W,c^{-1})$. Write $\zeta_{\NC(W,c^{-1})}^{(m)}$ for the $m$th power of $\zeta_{\NC(W,c^{-1})}$ in the incidence algebra; in particular, $\zeta_{\NC(W,c^{-1})}^{(-1)}=\mu_{\NC(W,c^{-1})}$ is the M\"obius function.  Then the following is proved in the same way as~\Cref{thm:EL,thm:homotopy_type} via R.~Stanley's reciprocity theorem for zeta polynomials~\cite{stanley1975combinatorial,bjorner1996shellable}. 

\begin{theorem}\label{thm:zeta}
For $m \in \mathbb{Z}$ and $w\in W$,
\[(q-1)^{\ell_T(w)-r}[\T_w] \H(\Delta^{2m}\B(c^{-1})) \Big|_{q=1} = \begin{cases}\zeta^{(m)}_{\NC(W,c^{-1})}(w,c^{-1}) &\text{ if } w \in \NC(W,c^{-1}) \\ 0 &\text{otherwise}.\end{cases}\] 
\end{theorem}

\newcounter{listidx}
\newcounter{edgecounter}
\newcommand{\DrawLabeledKn}[3]{
    \begin{tikzpicture}
        \setcounter{listidx}{0}
        \foreach \item in {#3} {
            \stepcounter{listidx}
            \expandafter\xdef\csname edgeLabel\thelistidx\endcsname{\item}
        }
        \def\R{#1}
        \def\n{#2}
        \def\startAngle{0} 
        \draw[gray!20, thin] (0,0) circle (\R);
        \foreach \i in {1,...,\n} {
            \pgfmathsetmacro{\angle}{\startAngle + (\i-1)*360/\n}
            \node[circle, fill=black, inner sep=0pt, outer sep=0pt] (V\i) at (\angle:\R) {};
            \node at (\angle:\R+0.4) {\i};
        }
        \setcounter{edgecounter}{0}
        \foreach \i in {1,...,\n} {
            \pgfmathsetmacro{\next}{\i+1}
            \ifnum \i < \n
                \foreach \j in {\next,...,\n} {
                    \stepcounter{edgecounter}
                    \draw[gray!70, thick] (V\i) -- (V\j) 
                        node[midway, 
                             sloped, 
                             fill=white, 
                             inner sep=1.5pt, 
                             rounded corners=2pt,
                             font=\scriptsize\color{blue!40!black},
                             pos=0.5] 
                        {\csname edgeLabel\theedgecounter\endcsname}; 
                }
            \fi
        }
        \foreach \i in {1,...,\n} {
            \pgfmathsetmacro{\angle}{\startAngle + (\i-1)*360/\n}
            \node[circle, fill=black, inner sep=2pt, outer sep=0pt] at (V\i) {};
        }
    \end{tikzpicture}
}
\newcommand{\DrawKn}[3]{
    \begin{tikzpicture}
        \setcounter{listidx}{0}
        \foreach \item in {#3} {
            \stepcounter{listidx}
            \expandafter\xdef\csname edgeLabel\thelistidx\endcsname{\item}
        }
        \def\R{#1}
        \def\n{#2}
        \def\startAngle{0} 
        \draw[gray!20, thin] (0,0) circle (\R);
        \foreach \i in {1,...,\n} {
            \pgfmathsetmacro{\angle}{\startAngle + (\i-1)*360/\n}
            \node[circle, fill=black, inner sep=0pt, outer sep=0pt] (V\i) at (\angle:\R) {};
        }
        \setcounter{edgecounter}{0}
        \foreach \i in {1,...,\n} {
            \pgfmathsetmacro{\next}{\i+1}
            \ifnum \i < \n
                \foreach \j in {\next,...,\n} {
                    \stepcounter{edgecounter}
                    \draw[gray!70, thick] (V\i) -- (V\j) 
                        node[midway, 
                             sloped, 
                             fill=white, 
                             inner sep=1.5pt, 
                             rounded corners=2pt,
                             font=\scriptsize\color{blue!40!black},
                             pos=0.5] 
                        {\csname edgeLabel\theedgecounter\endcsname}; 
                }
            \fi
        }
        \foreach \i in {1,...,\n} {
            \pgfmathsetmacro{\angle}{\startAngle + (\i-1)*360/\n}
            \node[circle, fill=black, inner sep=2pt, outer sep=0pt] (V\i) at (\angle:\R) {};
        }
    \end{tikzpicture}
}

\subsection{Combinatorial Interpretation in Type A}\label{sec:full_twist_combo}
Drawing on a problem the second author wrote for the Enumerative Combinatorics workshop held in January 2026 at Oberwolfach, we can make the preceding discussion very concrete in type~$A$ when $c$ is the linear Coxeter element $(12\cdots n)$.  In this case, the base case of Thomas's argument is easy (since the join in the noncrossing partition lattice of any two reflections in type $A$ has rank at most 3, the base case requires only exhaustively checking $A_4$), so we know that $c^{-1}$-cyclic total orderings of $T$ coincide with EL-shelling orders of $\NC(W,c^{-1})$. 

Draw the complete graph $K_n$ with vertices $1,2,\ldots,n$ around a circle.  Let $X_n$ be the set of bijective labelings of the edges of $K_n$ by $1, 2, \dots, \binom{n}{2}$ such that every triangle increases counterclockwise. 
  Call the elements of $X_n$ \defn{cyclic labelings}.
  
Then each cyclic labeling has a unique spanning tree with edges that increase counterclockwise around each vertex, corresponding to the EL-shelling condition.  And each cyclic labeling has $\Cat^+(A_{n-1})=\mathrm{Cat}_{n-1}$ spanning trees with edges that increase clockwise, corresponding to the homotopy type of the order complex of $\overline{\NC}(W,c^{-1})$.

\begin{example}
  Here is an interestingly symmetric cyclic labeling for $K_5$:
\[\raisebox{-0.5\height}{\DrawLabeledKn{2}{5}{3,2,8,1,5,4,10,7,6,9}} \in X_5,\]
and here is the corresponding ordered list of all $10$ edges:
\[\mathbf{(15)}(13)(12)\mathbf{(24)}(23)(35)\mathbf{(34)}(14)(45)\mathbf{(25)}.\]
The bolded reflections correspond to the unique counterclockwise increasing spanning tree, and hence the unique factorization for $c^{-1}=(15432)$ in the reflection order.
\end{example}

We record here that 
\[|X_3| = 3, \quad |X_4| = 48,\quad  |X_5| = 5150,\quad \text{and}\quad |X_6| = 4,218,360.\]

\section{Fuss Models}\label{sec:fuss} 
In this section, we apply \cref{thm:main} to recover combinatorial models for Fuss--Catalan and Fuss--Dogolon numbers in terms of subwords.  

\subsection{Why the Fuss?}

We recall the definitions of Fuss--Catalan and Dogolon clusters and noncrossing partitions as subwords from~\cite{stump2025cataland}.
\begin{definition}[{\cite[Definition 5.3.1 and Section 7]{stump2025cataland}}]\label{def:associahedron}
    An \dfn{$m$-cluster} is a  subword of $\inv({\sf w}_\circ(c))^{m} \inv(c)$ that is a reduced $T$-word for $c^{-1}$. A \dfn{positive $m$-cluster} is a subword of $\inv({\sf w}_\circ(c))^{m}$ that is a reduced $T$-word for $c^{-1}$. 
\end{definition}

\begin{definition}[{\cite[Definition 4.2.4 and Section 7]{stump2025cataland}}]
\label{def:noncrossing}
    An \dfn{$m$-noncrossing partition} is a subword of
$\inv({\sf w}_\circ(c^{-1}))^{m+1}$ that is a reduced $T$-word for
$c^{-1}$. A \dfn{positive $m$-noncrossing partition} is a subword of
\[
\inv({\sf w}_\circ(c^{-1}))^m\inv_{c^{-1}}(cw_\circ)
\]
that is a reduced $T$-word for $c^{-1}$.
\end{definition}

\subsection{Fuss--Catalan Models}\label{sec:fuss-cat}
When $\bbeta=\B(c)^{mh+1}$, the positive model gives $m$-noncrossing partitions, while the negative model gives $m$-clusters~\cite{stump2025cataland}.  We actually get new subword models, since our technique only depends on 
$\mathbb T$-words for the full twist.

Our starting point is the following result, which is a consequence of \cite[Corollaries~5.4~\&~6.13]{galashin2024rational}. It produces the $q$-rational Catalan numbers as traces in the Hecke algebra.  

\begin{theorem}[{\cite{galashin2024rational}}]\label{thm:rational_Catalan_trace}
Let $p$ be a positive integer coprime to the Coxeter number $h$ of $W$. 
We have
\[
[\T_e]\T_{c}^{p+h}
=
q^N(q-1)^r\Cat_p(W;q)
\]
and
\[
[\T_e]\T_{c}^{-p}
=
(q^{-1}-1)^r\Cat_p(W;q^{-1})
=
(-1)^r q^{-rp}(q-1)^r\Cat_p(W;q).
\]
\end{theorem}

\begin{proposition}\label{prop:Catmh+1}
Let $m$ be a positive integer, and consider a $T$-word $\protect\vv{t}=(t_1,\ldots,t_{(m+1)N})$ such that $\tt_1\cdots\tt_{(m+1)N}=\Delta^{2(m+1)}$. We have 
\[\Subb_\RR(\vect,c^{-1})=\Cat_{mh+1}(W).\]
\end{proposition}
\begin{proof}
Applying \cref{thm:rational_Catalan_trace} with $p=mh+1$ yields 
\[\Cat_{mh+1}(W)=(q-1)^{-r}[\T_e]\T_c^{(m+1)h+1}\Big|_{q=1}.\] 
Using \cref{lem:coeff-identity}, we can write 
\[[\T_e]\T_c^{(m+1)h+1}=q^r[\T_{c^{-1}}]\T_c^{(m+1)h}=q^r[\T_{c^{-1}}]\H(\tt_1\cdots\tt_{(m+1)N}).\] Applying \cref{cor:main} with $\pi=c^{-1}$ yields the desired result. 
\end{proof} 
Since $\Delta^2=\prod_{t \in \inv(\mathsf{w}_\circ(c^{-1}))} \tt$, \cref{prop:Catmh+1} allows us to recover the model of $m$-noncrossing partitions from~\Cref{def:noncrossing} when $\vect=\inv(\mathsf{w}_\circ(c^{-1}))^{m+1}$.   

\begin{proposition}\label{prop:Clusters_m}
Let $m$ be a positive integer, and consider a $T$-word $\protect\vv{t}=(t_1,\ldots,t_{mN})$ such that $\tt_{mN}\cdots\tt_{1}=\Delta^{2m}$.  We have 
\[\Subb_\RR(\vect\,\inv(c),c^{-1})=\Cat_{mh+1}(W).\]
\end{proposition} 
\begin{proof} 
Applying \cref{thm:rational_Catalan_trace} with $p=mh+1$ yields that  
\[\Cat_{mh+1}(W)=(-1)^r(q-1)^{-r}[\T_e]\T_{c^{-1}}^{-(mh+1)}\Big|_{q=1}.\] Using \cref{lem:coeff-identity}, we can write \[[\T_e]\T_{c^{-1}}^{-(mh+1)}=[\T_e](\H(\Delta^{-2m})\T_{c^{-1}}^{-1}\T_{c}^{-1}\T_{c})=q^r[\T_{c^{-1}}]\left(\H(\Delta^{-2m})\T_{c^{-1}}^{-1}\T_{c}^{-1}\right).\] 
Letting $\inv(c)=(u_1,\ldots,u_r)$, we know by \cref{lem:inversions} that 
$\B(c)\B(c^{-1})=\uu_r\cdots\uu_1$, so 
\[\T_{c^{-1}}^{-1}\T_{c}^{-1}=\H(\B(c)\B(c^{-1}))^{-1}=\H(\uu_r\cdots\uu_1)^{-1}.\] Moreover, $\H(\Delta^{-2m})=\H(\tt_1^{-1}\cdots\tt_{mN}^{-1})$. Putting this all together shows that 
\[[\T_e]\T_{c^{-1}}^{-(mh+1)}=q^r[\T_{c^{-1}}]\H(\tt_1^{-1}\cdots\tt_{mN}^{-1}\uu_1^{-1}\cdots\uu_r^{-1}).\] Now let $\protect\vv\epsilon=(\epsilon_1,\ldots,\epsilon_{mN+r})$, where $\epsilon_i=-1$ for all $i$. Applying \cref{thm:main} with $\pi=c^{-1}$, with the word $\vect\,\inv(c)$, and with the sign vector $\protect\vv\epsilon$ yields that 
\[
\Cat_{mh+1}(W)=(-1)^r\Subb_{\RR}^{\protect\vv\epsilon}(\vect\,\inv(c),c^{-1}). 
\]
Because every entry of $\protect\vv\epsilon$ is $-1$ and every reduced $\RR$-word for $c^{-1}$ uses exactly $r$ reflections, we have 
\[(-1)^r\Subb_{\RR}^{\protect\vv\epsilon}(\vect\,\inv(c),c^{-1})=\Subb_{\RR}(\vect\,\inv(c),c^{-1}),\] as desired.  
\end{proof}  

Since $\Delta^2=\prod_{t \in \inv(\mathsf{w}_\circ(c^{-1}))} \tt$, \cref{prop:Clusters_m} allows us to recover the model of $m$-clusters from~\Cref{def:associahedron} when $\vect=\inv(\mathsf{w}_\circ(c))^{m}$.

\subsection{Fuss--Dogolon Models}\label{sec:dog}
When $\bbeta=\B(c)^{mh-1}$, the positive model gives positive $m$-noncrossing partitions, while the negative model gives positive $m$-clusters~\cite{stump2025cataland}.  

\begin{proposition}
For every positive integer $m$, we have
\[
\left|
\operatorname{Sub}_T\!\left(
\operatorname{inv}({\sf w}_\circ(c^{-1}))^m
\operatorname{inv}_{c^{-1}}(cw_\circ),
c^{-1}
\right)
\right|
=
\operatorname{Cat}_{mh-1}(W).
\] 
\end{proposition}

\begin{proof}
Applying \cref{thm:rational_Catalan_trace} with the standard Coxeter element $c^{-1}$ and
$p=mh-1$ gives
\[
\operatorname{Cat}_{mh-1}(W)
=
\left.
(q-1)^{-r}
[\T_e]\T_{c^{-1}}^{(m+1)h-1}
\right|_{q=1}.
\]
Since $\ell(cw_\circ)=\ell(w_\circ c^{-1})=N-r$, we have
\[
\B(c^{-1})\B(cw_\circ)=
\B(w_\circ c^{-1})\B(c)=\Delta.
\]
Using the fact that $\Delta^2$ is central, we obtain
\begin{align}
\nonumber \mathbf{B}(cw_\circ)\mathbf{B}(w_\circ c^{-1})
\mathbf{B}(c)\mathbf{B}(c^{-1})
&=
\mathbf{B}(cw_\circ)\Delta\mathbf{B}(c^{-1})\\
\nonumber &=
\mathbf{B}(c^{-1})^{-1}\Delta^2\mathbf{B}(c^{-1})\\
\label{eq:prop_6.6}&=
\Delta^2.
\end{align}
It follows that $\Delta^2
\bigl(\mathbf{B}(c)\mathbf{B}(c^{-1})\bigr)^{-1}
=
\mathbf{B}(cw_\circ)\mathbf{B}(w_\circ c^{-1})$, so 
\[
\begin{aligned}
\B(c^{-1})^{(m+1)h-1}
&=
\Delta^{2(m+1)}\B(c^{-1})^{-1}\\
&=
\Delta^{2m}
\left(
\Delta^2
\bigl(\B(c)\B(c^{-1})\bigr)^{-1}
\right)
\B(c)\\
&=
\Delta^{2m}
\B(cw_\circ)\B(w_\circ c^{-1})
\B(c).
\end{aligned}
\]
\cref{lem:coeff-identity} now tells us that
\[
[\T_e]\T_{c^{-1}}^{(m+1)h-1}
=
q^r
[\T_{c^{-1}}]
\H\!\left(
\Delta^{2m}
\B(cw_\circ)\B(w_\circ c^{-1})
\right).
\]

By \cref{cor:wo}, there is a $\mathbb T$-word for $\Delta^2$ that is
commutation equivalent to
$\inv({\sf w}_\circ(c^{-1}))$.
\cref{lem:inversions} implies that one can replace the reflections in the $T$-word $
\mathrm{rev}(
\inv_c(cw_\circ)
)$ with the corresponding squares of $c$-dual lifts in order to obtain a $\mathbb{T}$-word for $\B(cw_\circ)\B(w_\circ c^{-1})$; moreover, $
\mathrm{rev}(
\inv_c(cw_\circ)
)$ is commutation equivalent to $\inv_{c^{-1}}(cw_\circ)$.
Hence, replacing the reflections in  $\inv({\sf w}_\circ(c^{-1}))^m
\inv_{c^{-1}}(cw_\circ)$ with the squares of their $c$-dual lifts yields a $\mathbb T$-word for 
\[
\Delta^{2m}
\B(cw_\circ)\B(w_\circ c^{-1}).
\]
Applying \cref{cor:main} with $\pi=c^{-1}$ gives
\[
\begin{aligned}
\Cat_{mh-1}(W)
&=
\left.
(q-1)^{-r}
[\T_{c^{-1}}]
\mathrm{H}\!\left(
\Delta^{2m}
\B(cw_\circ)\B(w_\circ c^{-1})
\right)
\right|_{q=1}\\
&=
\Subb_T\!\left(
\inv({\sf w}_\circ(c^{-1}))^m
\inv_{c^{-1}}(cw_\circ),
c^{-1}
\right),
\end{aligned}
\]
as desired.
\end{proof}

\begin{proposition}\label{prop:positive}
For each positive integer $m$, we have  
\[\Subb_\RR(\inv({\sf w}_\circ(c))^m,c^{-1})=\Cat_{mh-1}(W).\]
\end{proposition} 
\begin{proof}
\cref{thm:rational_Catalan_trace} tells us that 
\[\Cat_{mh-1}(W)=(-1)^r(q-1)^{-r}[\T_e]\T_{c}^{-(mh-1)}\Big|_{q=1}.\] Using \cref{lem:coeff-identity}, we can write \[[\T_e]\T_{c}^{-(mh-1)}=[\T_e](\H(\Delta^{-2m})\T_{c})=q^r[\T_{c^{-1}}]\H(\Delta^{-2m}).\] 
Letting $\inv({\sf w}_\circ(c))=(t_1,\ldots,t_N)$, we have $\Delta^2=\tt_N\cdots \tt_1$, so $\Delta^{-2m}=\left(\tt_1^{-1}\cdots \tt_N^{-1}\right)^m$. 
Now let $\protect\vv{t}=\inv({\sf w}_\circ(c))^m$, and let $\protect\vv\epsilon=(\epsilon_1,\ldots,\epsilon_{mN})$, where $\epsilon_i=-1$ for all $i$. Applying \cref{thm:main} with $\pi=c^{-1}$ and with these choices of $\protect\vv{t}$ and $\protect\vv\epsilon$ yields that 
\begin{align*}
\Cat_{mh-1}(W)&=(-1)^r(q-1)^{-r}q^r[\T_{c^{-1}}]\H\!\left(\tt_1^{-1}\cdots \tt_N^{-1}\right)^m\Big|_{q=1} \\ 
&=(-1)^r\Subb_{\RR}^{\protect\vv\epsilon}(\inv({\sf w}_\circ(c))^m,c^{-1}). 
\end{align*}
Because every entry of $\protect\vv\epsilon$ is $-1$ and every reduced $\RR$-word for $c^{-1}$ uses exactly $r$ reflections, we have 
\[(-1)^r\Subb_{\RR}^{\protect\vv\epsilon}(\inv({\sf w}_\circ(c))^m,c^{-1})=\Subb_{\RR}(\inv({\sf w}_\circ(c))^m,c^{-1}),\] as desired.  
\end{proof}

\section{Rational Models}\label{sec:rational}
We now construct reflection subword models for rational $W$-Catalan numbers. 

Given a standard Coxeter element $c$, we know by \cref{thm:rational_Catalan_trace} that the trace $[\T_e] \T_c^{p+h}$ is the rational $q$-Catalan number, up to multiplication by powers of $q$ and $q-1$. Thus, one might hope to use \cref{cor:main} to find, for each $p>0$ coprime to $h$ and each choice of $c$, a $\RR$-word $\vect$ such that $\Subb_\RR(\vect,c^{-1})=\Cat_p(W)$ (as we did when $p\equiv \pm 1\pmod h$ in \cref{sec:fuss}). Note that it suffices to solve this problem for a particular standard Coxeter element $c$ since we can then easily solve it for all standard Coxeter elements by conjugating. Indeed, it is well known that all standard Coxeter elements are conjugate to each other~\cite{humphreys1992reflection}. 

The trouble is that $c^p$ is not generally a \emph{standard} Coxeter element.  In contrast to~\Cref{prop:dual_artin}---where we can simply take the Hurwitz orbit of the standard generators---we do not have a systematic way to write the lifts of the reflections in the dual braid group for a nonstandard Coxeter element in terms of the standard Artin generators (since a nonstandard Coxeter element does not have a reduced reflection expression using only simple reflections).  Thus, even though we still have from~\Cref{lem:coeff-identity} that \[[\T_e] \T_c^{p+h} = q^{\ell(c^p)} [\T_{c^{-p}}] \H(\Delta^2\B(c)^p \B(c^p)^{-1}),\] we cannot directly use \Cref{cor:main} to compute the coefficient of $\T_{c^{-p}}$.

All, however, is not lost.  When $W$ is a Weyl group and $p$ is coprime to $h$, Springer proved that the elements $c^p$ and $c$ are conjugate.  

\begin{theorem}[{\cite[Theorem 4.2, Proposition 4.7]{springer1974regular}}] Let $c$ be a standard Coxeter element of a Weyl group $W$.  For $p$ coprime to the Coxeter number $h$, there exists $g \in W$ such that \begin{equation}g c^p g^{-1} = c.\label{eq:springer}\end{equation}  Moreover, the set of such elements $g$ is a left coset of the cyclic subgroup $\{c^i:0\leq i<h\}$. 
\label{thm:springer}
\end{theorem}

In light of the preceding theorem, we define a \defn{Coxeter $(c,p)$-conjugator} to be an element $g\in W$ satisfying $gc^pg^{-1}=c$. If $g$ is a Coxeter $(c,p)$-conjugator, then the set of all Coxeter $(c,p)$-conjugators is \[\Conj_W(c,p):=\{gc^i : 0 \leq i< h\}.\] For $g\in \Conj_W(c,p)$, we can use \cref{lem:coeff-identity} and the fact that traces are invariant under conjugation to compute that \[[\T_e] \T_c^p = [\T_e] (\T_g \T_c^p \T_g^{-1}) = q^r[\T_{c^{-1}}] (\T_g \T_c^p \T_g^{-1}\T_c^{-1}).\] This form is more desirable since it involves a coefficient of $\T_{c^{-1}}$ and $c^{-1}$ is a standard Coxeter element. 

\begin{example} Take $W=\mathfrak{S}_5$. For $c=s_1s_2s_3s_4=(12345)$ and $p=2$, the element $g=s_2s_3s_1$ satisfies $g c^2 g^{-1} = c$.  We have \[\Conj_{\mathfrak{S}_5}\big((12345),2\big) = \{s_2s_3s_1, s_3s_4s_2,s_1s_2s_3s_4s_2s_3s_1,s_4s_1s_2s_3s_2,s_2s_3s_4s_3s_1s_2s_1\}.\qedhere\]
\end{example}

Suppose $g\in\Conj_W(c,p)$. While $g c^p g^{-1} c^{-1}$ is the identity in the Weyl group $W$, we only know that $\T_g \T_c^p \T_g^{-1}\T_c^{-1}$ is the image of some pure braid in the Hecke algebra---and we still need to express this pure braid as a product of \emph{positive} squares of the dual braid generators in order to obtain a combinatorial reflection subword model for rational Catalan objects via \cref{cor:main}.

One difficulty comes from the fact that Coxeter $(c,p)$-conjugators depend on $c$, and it is not clear how to find the ``correct'' lift of an element of $\Conj_W(c,p)$ to $\B_W$.  For example, even when $p\equiv -1 \pmod h$ so that both $c$ and $c^p=c^{-1}$ are standard Coxeter elements, there is a very reasonable lift of $g$ to $\B_W$, which comes from conjugating by a sequence of initial reflections, as in~\cite{reading2007clusters}.  This is easy and type-uniform for bipartite Coxeter elements, but it is not as straightforward for other Coxeter elements. 

For $g\in \Conj_W(c,p)$, choose a braid lift $\mathbf g\in \B_W$ of $g$, and define
\[\XX
    :=
    \mathbf g \B(c)^p \mathbf g^{-1} \B(c)^{-1}.
\] 
Note that $\XX$ is a pure braid and $\XX\B(c)$ is an $h$th root of $\Delta^{2p}$. 
In order to apply \cref{cor:main}, we need to factor $\XX$ as 
\[
    \XX
    =
    \P_c(t_1)\cdots \P_c(t_M)
\]
for some $\vect=(t_1,\ldots,t_M)\in\RR^M$, where $M=r(p-1)/2$. If such a factorization exists, we call the lift $\mathbf g$ a \dfn{positive dual $(c,p)$-conjugator}. Note that if one exhibits a factorization of the form 
\[
    \XX
    =
    \P_c(t_1)^{\epsilon_1}\cdots
    \P_c(t_M)^{\epsilon_M}
\]
with $\epsilon_i\in\{\pm 1\}$ for all $i$, then \cref{thm:main} still gives a signed reflection subword model. The positivity
condition is precisely what distinguishes an honest counting model
from a signed one. 

\begin{example}
Let $W=\mathfrak{S}_5$ so that $\B_W=\B_5$ is the $5$-strand braid group. Let $c=s_1s_2s_3s_4$ and $p=2$. Suppose we choose $g=s_2s_3s_1 \in \mathfrak{S}_5$ and $\mathbf{g}=\s_2\s_3\s_1 \in B_5$. Then
\[\mathbf{g}\B(c)^p \mathbf{g}^{-1}=(\s_2\s_3\s_1) (\s_1\s_2\s_3\s_4)^2 (\s_1^{-1}\s_3^{-1}\s_2^{-1}) = (\s_2\s_1\s_1\s_2) \cdot \s_1\s_2\s_3\s_4 = \XX \cdot \B(c),\]
where $\XX=\s_2\s_1\s_1\s_2$.  By~\Cref{lem:inversions}, $\XX=\tt_{(23)}\tt_{(13)}$, where we are taking the $c$-lifts of the reflections to the dual braid group. 

On the other hand, if we instead choose $g=s_2s_3s_4s_3s_1s_2s_1$ and $\mathbf{g}=\s_2\s_3\s_4\s_3\s_1\s_2\s_1$, then 
\begin{align*}\mathbf{g} \B(c)^p \mathbf{g}^{-1} &= \left(\s_3^{-1} \s_1^{-1} \s_2^{-1} \s_2^{-1} \s_1^{-1} \s_3  \s_4 \s_2 \s_3 \s_2 \s_2 \s_3 \s_2 \s_4\right) \cdot \s_1 \s_2 \s_3 \s_4 \\
&=  \left(\tt_{(12)}^{-1} \tt_{(14)}^{-1} \tt_{(45)} \tt_{(23)} \tt_{(35)} \tt_{(25)} \right) \cdot \s_1 \s_2 \s_3 \s_4.
\end{align*}
In this case, the element $\XX=\tt_{(12)}^{-1} \tt_{(14)}^{-1} \tt_{(45)} \tt_{(23)} \tt_{(35)} \tt_{(25)}$ uses both positive and negative powers of the squares of the $c$-dual lifts of reflections. Hence, this choice only gives a signed subword model. 
\end{example}

\begin{theorem}\label{thm:rational-subword-models}
Let $W$ be a Weyl group of rank $r$, let $c$ be a standard Coxeter
element, and let $p>0$ be coprime to the Coxeter number $h$. Choose
$g_+,g_-\in W$ such that $g_+c^pg_+^{-1}=g_-c^{-p}g_-^{-1}=c$, and choose
a braid lift $\mathbf g_+$ of $g_+$ and a braid lift $\mathbf g_-$ of $g_-$. Define
\[
\mathbbm c_{p,\mathbf g_+}
=
\mathbf g_+\B(c)^p\mathbf g_+^{-1}\B(c)^{-1}
\quad\text{and}\quad
\mathbbm c_{-p,\mathbf g_-}
=
\mathbf g_-\B(c)^{-p}\mathbf g_-^{-1}\B(c)^{-1}.
\]
Suppose we have factorizations 
\[
\mathbbm c_{p,\mathbf g_+}
=
\prod_{i=1}^{M_+}\P_c(t_i)^{\epsilon_i},
\qquad
\mathbbm c_{-p,\mathbf g_-}
=
\prod_{j=1}^{M_-}\P_c(v_j)^{\delta_j},
\qquad
\Delta^2
=
\prod_{k=1}^{N}\P_c(u_k),
\]
where\[
\vect=(t_1,\ldots,t_{M_+})\in T^{M_+},\qquad
\protect\vv{v}=(v_1,\ldots,v_{M_-})\in T^{M_-},\qquad
\protect\vv{u}=(u_1,\ldots,u_N)\in T^N,
\]
\[\protect\vv{\epsilon}=(\epsilon_1,\ldots,\epsilon_{M_+})\in\{\pm 1\}^{M_+},\qquad\protect\vv{\delta}=(\delta_1,\ldots,\delta_{M_-})\in\{\pm 1\}^{M_-}.\]
Let $(1^N,\protect\vv{\epsilon})\in\{\pm 1\}^{M_++N}$ be obtained by concatenating the vector $(1,\ldots,1)\in\{\pm 1\}^N$ with $\protect\vv{\epsilon}$. 
Then
\[
\Subb_T^{(1^N,\protect\vv{\epsilon})}
(\vv{u\vphantom{t}}\vect,c^{-1})
=
\Cat_p(W)
=
(-1)^r
\Subb_T^{\protect\vv\delta}
(\protect\vv{v},c^{-1}). 
\] 
\end{theorem} 

\begin{proof}
By \cref{thm:rational_Catalan_trace} and the identity $\B(c)^h=\Delta^2$, we have
\[
[\T_e]\H(\Delta^2\B(c)^p)
=
q^N(q-1)^r\Cat_p(W;q).
\]
Using the centrality of $\Delta^2$, the invariance of the trace under
conjugation, and the definition of $\mathbbm c_{p,\mathbf g_+}$, we
obtain
\begin{align*}
[\T_e]\H(\Delta^2\B(c)^p)
&=
[\T_e]\H\bigl(\Delta^2\mathbf g_+\B(c)^p\mathbf g_+^{-1}\bigr)\\
&=
[\T_e]\H\bigl(\Delta^2\mathbbm c_{p,\mathbf g_+}\B(c)\bigr)\\
&=
q^r[\T_{c^{-1}}]\H\bigl(\Delta^2\mathbbm c_{p,\mathbf g_+}\bigr),
\end{align*}
where the last equality follows from \cref{lem:coeff-identity}. Dividing by
$(q-1)^r$, specializing at $q=1$, and applying \cref{thm:main} gives that 
\[
\Cat_p(W)
=
\Subb_T^{(1^N,\protect\vv{\epsilon})}
(\vv{u\vphantom{t}}\vect,c^{-1}).
\]

\cref{thm:rational_Catalan_trace} also tells us that 
\[
[\T_e]\H(\B(c)^{-p})
=
(-1)^r q^{-rp}(q-1)^r\Cat_p(W;q).
\]
Using the invariance of the trace under conjugation and the definition
of $\mathbbm c_{-p,\mathbf g_-}$, we obtain
\begin{align*}
[\T_e]\H(\B(c)^{-p})
&=
[\T_e]\H\bigl(\mathbf g_-\B(c)^{-p}\mathbf g_-^{-1}\bigr)\\
&=
[\T_e]\H\bigl(\mathbbm c_{-p,\mathbf g_-}\B(c)\bigr)\\
&=
q^r[\T_{c^{-1}}]\H\bigl(\mathbbm c_{-p,\mathbf g_-}\bigr).
\end{align*}
Dividing by $(q-1)^r$, specializing at $q=1$, and applying
\cref{thm:main} now yields the identity 
\[
(-1)^r\Cat_p(W)
=
\Subb_T^{\protect\vv\delta}(\protect\vv{v},c^{-1}),
\]
which is equivalent to the second equality.
\end{proof}

Let us say a positive dual $(c,p)$-conjugator ${\bf g}$ is \dfn{realizable} if there exists $x\in W$ such that 
\[\XX=\mathbf g \B(c)^p \mathbf g^{-1} \B(c)^{-1}=\B(x)\B(x^{-1}).\] 
Computing an appropriate factorization of $\XX$ can be tricky in general, but it is easy when ${\bf g}$ is realizable because we can apply \cref{lem:inversions}.  
One might hope to use this idea to describe subword models for rational Catalan numbers in a type-uniform manner. Sadly, realizable positive dual $(c,p)$-conjugators do not always exist. For example, we exhaustively checked that they do not exist in type $E_7$ for $p=5$ (for any $c$).  Thus, we must content ourselves with constructing positive dual $(c,p)$-conjugators in a type-by-type manner. Let us stress again that it suffices to do so for a particular standard Coxeter element $c$. 

The next proposition allows us to restrict our attention to the case where $1\leq p< h$. 

\begin{proposition}\label{prop:p-range}
Let $c$ be a standard Coxeter element of a Weyl group $W$. Suppose $p>0$ is an integer coprime to the Coxeter number $h$. Write $p=dh+\overline p$, where $d$ and $\overline p$ are integers satisfying $d\geq 0$ and $1\leq \overline p< h$. If ${\bf g}$ is a positive dual $(c,\overline p)$-conjugator, then it is also a positive dual $(c,p)$-conjugator. More precisely, 
\[{\bf g}\B(c)^p{\bf g}^{-1}\B(c)^{-1}=\Delta^{2d}{\bf g}\B(c)^{\overline p}{\bf g}^{-1}\B(c)^{-1}.\]
\end{proposition} 
\begin{proof}
Since $\B(c)^h=\Delta^2$ and $\Delta^2$ is central, we have
\[
\begin{aligned}
    {\bf g}\B(c)^p{\bf g}^{-1}\B(c)^{-1}
    &=
    {\bf g}\Delta^{2d}\B(c)^{\overline p} {\bf g}^{-1}\B(c)^{-1} \\
    &=
    \Delta^{2d}
    {\bf g}\B(c)^{\overline p} {\bf g}^{-1}\B(c)^{-1}.
\end{aligned}
\] 
If ${\bf g}$ is a positive dual $(c,\overline p)$-conjugator, then, by definition, ${\bf g}\B(c)^{\overline p}{\bf g}^{-1}\B(c)^{-1}$ is in the monoid generated by $\mathbb T_c$. Then ${\bf g}\B(c)^{p}{\bf g}^{-1}\B(c)^{-1}$ also belongs to this monoid since, by \cref{cor:wo}, $\Delta^2$ belongs to this monoid. 
\end{proof}

In the following subsections, we will freely assume that $1\leq p\leq h-1$.

We will also make use of the following lemma, which comes from Bessis--Brou\'e--Michel Springer theory for braid groups~\cite[Theorem~12.4]{bessis2015finite}. 

\begin{lemma}[{\cite[Proposition 5.24]{broue2010introduction}}]\label{lem:Bessis_Springer}
Let $c$ be a standard Coxeter element of a Weyl group $W$. Suppose $p>0$ is an integer coprime to the Coxeter number $h$. If $\bbeta\in \B_W$ satisfies $\bbeta^h=\Delta^{2p}$, then $\bbeta$ is conjugate to $\B(c)^p$ in the braid group $\B_W$. 
\end{lemma}

\subsection{Folding}\label{subsec:folding}
The standard technique of \emph{folding} will help us produce positive dual conjugators in one type from positive dual conjugators in another type. 

Let $(W,S)$ be a finite irreducible Coxeter system. Suppose $\theta\colon{S}\to{S}$ is an automorphism of the Coxeter graph such that each orbit of $\theta$ consists of pairwise-commuting simple reflections. Let $\mathrm{Orb}_\theta$ be the set of orbits of $\theta$. Given $\mathcal O\in\mathrm{Orb}_\theta$, let $s_\mathcal O=\prod_{s\in \mathcal O}s$. Let $S^{\fold}$ be the set of elements of the form $s_{\mathcal O}$ such that $\mathcal O$ is a $\theta$-orbit. Let $W^{\fold}$ be the subgroup of $W$ generated by $S^{\fold}$. Then $(W^{\fold},S^{\fold})$ is a Coxeter system. We say $(W^{\fold},S^{\fold})$ is obtained from $({W},{S})$ by \dfn{folding}.   

There is an induced injection \[\iota:\B_{W^{\fold}}\hookrightarrow \B_{W}
\]
defined on the Artin generators by
\[
    \iota(\mathbf{s}_{\mathcal O})
    =
    \prod_{{s}\in\mathcal O}\mathbf{s}.
\]
The factors in this product commute, so the order is irrelevant.

Let $c^{\fold}=s_{\mathcal O_1}s_{\mathcal O_2}\cdots s_{\mathcal O_{r'}}$ be a standard Coxeter element of $W^{\fold}$. Let us write $c$ for the same element $c^{\fold}$ viewed as a standard Coxeter element of $W$. Then
\[
    \iota(\B(c^{\fold}))=\B(c).
\] The Coxeter groups $W$ and $W^{\fold}$ have the same Coxeter number. 

Now consider a reflection $t^{\fold}\in\RR_{W^{\fold}}$ of $W^{\fold}$. Then there exist $w\in W^{\fold}$ and $\mathcal O\in\mathrm{Orb}_\theta$ such that $t^{\fold}=ws_{\mathcal O}w^{-1}$. This yields the set 
\[\mathrm{Orb}_\theta(t^{\fold}):=\{wsw^{-1}:s\in\mathcal O\},\] which consists of pairwise-commuting reflections in $W$. 

We claim that the embedding $\iota$ sends the $c^{\fold}$-dual lift of $t^{\fold}$ to the product of $c$-dual lifts of the reflections in $\mathrm{Orb}_\theta(t^{\fold})$. That is, 
\begin{equation}\label{eq:iota}
\iota(\B_{c^{\fold}}(t^{\fold}))=\prod_{t\in\mathrm{Orb}_\theta(t^{\fold})}\B_{c}(t).
\end{equation} 
In addition, 
\begin{equation}\label{eq:iota2}
\iota(\P_{c^{\fold}}(t^{\fold}))=\prod_{t\in\mathrm{Orb}_\theta(t^{\fold})}\P_{c}(t).
\end{equation} 
To prove this, we use the original definition of dual lifts coming from \cref{thm:Bessis_Hurwitz}. Starting with the reduced word $(s_{\mathcal O_1},s_{\mathcal O_2},\ldots,s_{\mathcal O_{r'}})$ for $c^{\fold}$, we obtain a reduced word ${\mathsf{c}}$ for $c$ by replacing each letter $s_{\mathcal O_i}$ with a word consisting of the elements of $\mathcal O_i$ (each used exactly once). Thus, we view ${\mathsf{c}}$ as a concatenation of $r'$ blocks of simple reflections from $S$. We can then produce the corresponding braid words by replacing simple reflections by their Artin lifts. By \cref{thm:Bessis_Hurwitz}, we can perform a sequence of Hurwitz moves to transform the braid word $(\s_{\mathcal O_1},\s_{\mathcal O_2},\ldots,\s_{\mathcal O_{r'}})$ into a word containing $\B_{c^{\fold}}(t^{\fold})$. This sequence of Hurwitz moves corresponds to a sequence of Hurwitz moves in $\B_{W}$ starting with the word ${\mathsf c}$. Indeed, each of the original Hurwitz moves corresponds to a sequence of Hurwitz moves that moves a block of reflections past another block. This latter sequence of Hurwitz moves transforms ${\mathsf c}$ into a word containing the lifts in the set $\{\B_{c}(t):t\in\Orb_\theta(t^{\fold})\}$ as a consecutive block. This implies \eqref{eq:iota}. At each step, the lifts of reflections in a particular block all get conjugated by the same element. Thus, the elements $\B_{c}(t)$ for $t\in\mathrm{Orb}_\theta(t^{\fold})$ are obtained by conjugating the Artin lifts of the simple reflections in $\mathcal O$ by the same element. It follows that the elements $\B_{c}(t)$ for $t\in\mathrm{Orb}_\theta(t^{\fold})$ all commute with each other, so \eqref{eq:iota2} follows from squaring both sides of \eqref{eq:iota}.

\begin{proposition}\label{prop:folding}
Suppose $W^{\fold}$ is a folding of $W$. Let $c^{\fold}$ be a standard Coxeter element of $W^{\fold}$, and let $c$ be the Coxeter element $c^{\fold}$ viewed as an element of $W$. Let $h$ be the Coxeter number of $W$ and $W^{\fold}$. Let $p$ be a positive integer such that $\gcd(p,h)=1$. Suppose ${\bf g}^{\fold}\in\B_{W^{\fold}}$ is a positive dual $(c^{\fold},p)$-conjugator and that we can write 
\[
    \mathbbm{c}^{\fold}_{p,{\bf g}^{\fold}}
    :=
    {\bf g}^{\fold} \B(c^{\fold})^p({\bf g}^{\fold})^{-1}\B(c^{\fold})^{-1}=\P_{c^{\fold}}(t_1^{\fold})\cdots\P_{c^{\fold}}(t_M^{\fold})
\]
for some $(t_1^{\fold},\ldots,t_M^{\fold})\in T_{W^{\fold}}^M$. Then the element ${\bf g}=\iota({\bf g}^{\fold})$ is a positive dual $(c,p)$-conjugator. More precisely, \[
    \mathbbm{c}_{p,{\bf g}}
    :=
    {\bf g} \B(c)^p{\bf g}^{-1}\B(c)^{-1}=\iota(\mathbbm{c}^{\fold}_{p,{\bf g}^{\fold}})=\prod_{k=1}^M\prod_{t\in\mathrm{Orb}_\theta(t_k^{\fold})}\P_c(t). 
\]
\end{proposition}

\begin{proof}
We use the fact that $\iota$ is a homomorphism. The identity $
    {\bf g} \B(c)^p{\bf g}^{-1}\B(c)^{-1}=\iota(\mathbbm{c}^{\fold}_{p,{\bf g}^{\fold}})$ follows from the fact that $\iota(\B(c^{\fold}))=\B(c)$. The identity \[\iota(\mathbbm{c}^{\fold}_{p,{\bf g}^{\fold}})=\prod_{k=1}^M\prod_{t\in\mathrm{Orb}_\theta(t_k^{\fold})}\P_c(t)\] follows from \eqref{eq:iota2}.  
\end{proof}

\subsection{Type A}\label{subsec:Type_A}

Let $W=\mathfrak{S}_n$. Let $S=\{s_1,\ldots,s_{n-1}\}$, where $s_i=(i\,\,i+1)$. Let
\[c=s_1s_2\cdots s_{n-1}=(12\cdots n)\] be the linear Coxeter element. The Coxeter number is $h=n$. 

Fix an integer $p$ with $\gcd(p,n)=1$ and $1\leq p< n$. Let $p'$ be the unique integer satisfying $1\leq p'<n$ and
$pp'\equiv 1 \pmod n$. Let $g\in \mathfrak{S}_n$ be the permutation that acts on $[n]$ by multiplication by $p'$ (modulo $n$). That is, $g(i)$ is the unique element of $[n]$ congruent to $ip'$ modulo $n$. A straightforward computation shows that $g c^p g^{-1}=c.$

For
$1\leq i< p$, set
$
    k_i=\left\lfloor \frac{in}{p}\right\rfloor.$ Since $p<n$, we have $1\leq k_1<k_2<\cdots<k_{p-1}< n$. Define the $S$-words
    \[\mathsf{v}_k=(s_k,s_{k-1},\ldots,s_1)\quad\text{and}\quad \mathsf{x}_p=\mathsf{v}_{k_1}\mathsf{v}_{k_2}\cdots\mathsf{v}_{k_{p-1}},\] and let 
    \[v_k=s_ks_{k-1}\cdots s_1 \quad \text{and}
    \quad x_p=v_{k_1}v_{k_2}\cdots v_{k_{p-1}}
\] be the permutations represented by them. 

\begin{lemma}\label{lem:reduced_x_rho}
The word $\mathsf{x}_p=\mathsf{v}_{k_1}\mathsf{v}_{k_2}\cdots\mathsf{v}_{k_{p-1}}$
is a reduced word for $x_p$. Moreover, 
\[
    \mathrm{Inv}(x_p)
    =
    \{(a\,k_j+1): 1\leq j< p,\ 1\leq a\leq k_j\}.
\]
\end{lemma}

\begin{proof}
For $1\leq i\leq p-1$, let $y_i=v_{k_1}v_{k_2}\cdots v_{k_i}$. We claim that the one-line notation of
$y_i$ begins with the decreasing sequence $k_i+1, k_{i-1}+1,\ldots,k_1+1$ and then lists the remaining elements of $[n]$ in increasing order. This is clear when $i=0$ since $y_0=e$. Now assume $i\geq 1$, and proceed by induction on $i$. The one-line notation of $y_{i-1}$ consists of the decreasing sequence $k_{i-1}+1, k_{i-2}+1,\ldots,k_1+1$ followed by the remaining elements of $[n]$ in increasing order. Because $k_1<\cdots<k_{i-1}$, the number $k_i+1$ is in position $k_i+1$. Since $y_i=y_{i-1}v_{k_i}=y_{i-1}s_{k_i}s_{k_i-1}\cdots s_1$, we obtain the one-line notation of $y_i$ from that of $y_{i-1}$ by sliding the number $k_i+1$ all the way to the left. This proves the claim. 

Now, the inversions of $x_p=y_{p-1}$ are exactly the transpositions
$(a\,k_j+1)$ with $1\leq a\leq k_j$. It follows that $\ell(x_p)=k_1+\cdots+k_{p-1}$,
so the word $\mathsf{x}_p$ is reduced. 
\end{proof} 

\begin{example}
Let $n=11$ and $p=8$.  The $\frac{r(p-1)}{2}=\frac{10\cdot 7}{2}=35$ inversions of $x_p$ appear as the shaded boxes in~\Cref{fig:a10}---we have actually displayed each inversion twice, the second time reversed, so the $p-1$ inversions in each orbit of the reflections under conjugation by $c$ are more easily visible (as the shaded boxes in each row). 
\end{example}

\begin{figure}[htbp]
\begin{center}
\scalebox{0.7}{
\begin{tikzpicture}[x={(0.7cm,-0.7cm)}, y={(0.7cm,0.7cm)}]
\def\getcolor#1{
  \def\boxcolor{white}
  \ifnum#1=2 \def\boxcolor{blue!40}\fi
  \ifnum#1=3 \def\boxcolor{blue!40}\fi
  \ifnum#1=5 \def\boxcolor{blue!40}\fi
  \ifnum#1=6 \def\boxcolor{blue!40}\fi
  \ifnum#1=7 \def\boxcolor{blue!40}\fi
  \ifnum#1=9 \def\boxcolor{blue!40}\fi
  \ifnum#1=10 \def\boxcolor{blue!40}\fi
}

\foreach \j in {2,...,11} {
  \pgfmathtruncatemacro{\jminus}{\j-1}
  \foreach \i in {1,...,\jminus} {
    \getcolor{\j}
    \draw[gray!50, fill=\boxcolor, rounded corners=2pt] 
      (\i-0.45, \j-0.45) -- (\i+0.45, \j-0.45) -- 
      (\i+0.45, \j+0.45) -- (\i-0.45, \j+0.45) -- cycle;
    \node at (\i, \j) {\scriptsize $(\i\,\,\j)$};
  }
}

\foreach \j in {2,...,11} {
  \pgfmathtruncatemacro{\jminus}{\j-1}
  \foreach \i in {1,...,\jminus} {
    \getcolor{\j}
    \pgfmathtruncatemacro{\v}{11+\i}
    \draw[gray!50, fill=\boxcolor, rounded corners=2pt] 
      (\j-0.45, \v-0.45) -- (\j+0.45, \v-0.45) -- 
      (\j+0.45, \v+0.45) -- (\j-0.45, \v+0.45) -- cycle;
    \node at (\j, \v) {\scriptsize $(\j\,\,\i)$};
  }
}
\end{tikzpicture}}
\end{center}
\caption{The $\frac{r(p-1)}{2}=\frac{10\cdot 7}{2}=35$ inversions of $x_p$ (for $n=11$ and $p=8$), listed twice so that rows are the orbits of the reflections under conjugation by $c$.  It is interesting to observe that the diagonal lengths encode the run lengths of the associated rational Dyck path.}
\label{fig:a10}
\end{figure}
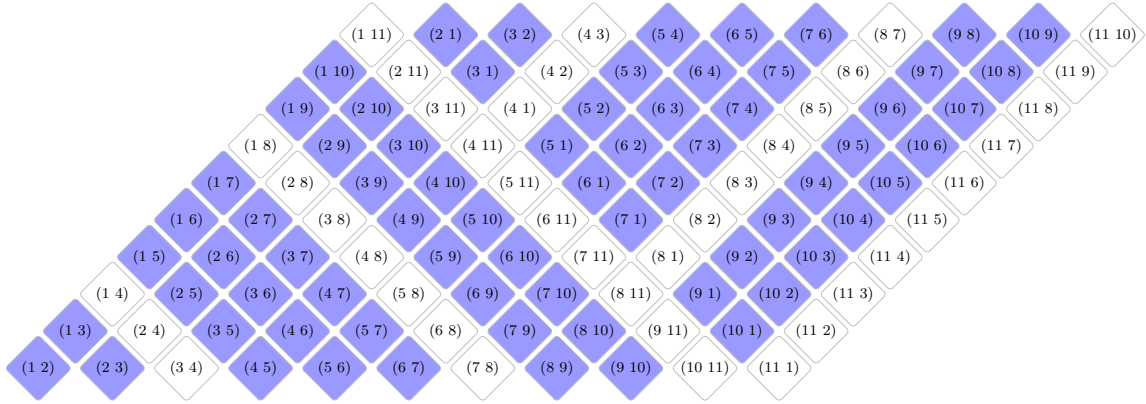

Now let $\mathbf{g}=\B(g)$. Let us write $t_{(ij)}$ for the transposition $(ij)$. For $1\leq i<j\leq n$, the $c$-dual lift of $t_{(ij)}$ is 
\[
    \t_{(ij)}=\mathbf{s}_i\mathbf{s}_{i+1}\cdots \mathbf{s}_{j-2}
    \mathbf{s}_{j-1}
    \mathbf{s}_{j-2}^{-1}\cdots \mathbf{s}_{i+1}^{-1}\mathbf{s}_i^{-1}.
\]
We use the standard dual braid relations, which state that 
\[
\t_{(ij)}\t_{(k\ell)}=\t_{(k\ell)}\t_{(ij)}
\]
whenever the chords $(i,j)$ and $(k,\ell)$ are disjoint and noncrossing and that 
\[
    \t_{(ij)}\t_{(jk)}
    =
    \t_{(jk)}\t_{(ik)}
    =
    \t_{(ik)}\t_{(ij)}
\] for $i<j<k$. 
We will also use the cyclic conjugation relation
\[
    \B(c) \t_{(ij)}\B(c)^{-1}=\t_{(i+1\,\, j+1)},
\]
where the indices are interpreted modulo $n$. Let 
\[\XX=\mathbf{g} \B(c)^p \mathbf{g}^{-1}\B(c)^{-1}.\]

By \cref{lem:inversions}, we have
\[
    \B(v_k)\B(v_k^{-1})
    =
    \tt_{k,k+1}\tt_{k-1,k+1}\cdots\tt_{1,k+1}.
\]
We will prove that 
\[
    \mathbf{g} \B(c)^p\mathbf{g}^{-1}\B(c)^{-1}
    =
    \B(x_p)\B(x_p^{-1}).
\] 

\begin{lemma}\label{lem:B(u)B(c)}
For $u\in \mathfrak{S}_n$, we have $\B(u)\B(c)=\B(v_M)\B(v_M^{-1})\B(uc)$, where $M=u(1)-1$. 
\end{lemma}

\begin{proof}
Let $\sigma=v_M^{-1}u$. Since $v_M(1)=M+1=u(1)$, we have $\sigma(1)=1$. Hence,
\[
    \sigma\in \langle s_2,s_3,\ldots,s_{n-1}\rangle,
\]
and the
factorization $u=v_M\sigma$
is length-additive. Therefore, $\B(u)=\B(v_M)\B(\sigma)$. Now set $\sigma'=c^{-1}\sigma c$. Since conjugation by $c^{-1}$ sends $s_i$ to $s_{i-1}$ for
$2\leq i\leq n-1$, we have
\[
    \sigma'\in \langle s_1,s_2,\ldots,s_{n-2}\rangle.
\]
We also know that $\B(c)^{-1}\mathbf{s}_i\B(c)=\mathbf{s}_{i-1}$ for all $2\leq i\leq n-1$.
It follows that 
\[
    \B(\sigma)\B(c)=\B(c) \B(\sigma').
\]

Observing that $v_Mc=s_{M+1}s_{M+2}\cdots s_{n-1}$, we obtain the identity $\B(v_Mc)=\s_{M+1}\s_{M+2}\cdots \s_{n-1}$.
Now,
\[
    uc=v_M\sigma c=v_Mc(c^{-1}\sigma c)=v_Mc\,\sigma'.
\]
The permutation $v_Mc$ has no descents in $\{s_1,\ldots,s_{n-2}\}$. Because $\sigma'\in\langle s_1,\ldots,s_{n-2}\rangle$, we have
\[
    \B(uc)=\s_{M+1}\s_{M+2}\cdots\s_{n-1}\B(\sigma').
\]
Putting these identities together gives
\[
\begin{aligned}
    \B(u)\B(c)
    &=
    \B(v_M)\B(\sigma)\B(c) \\
    &=
    \B(v_M)\B(c)\B(\sigma') \\
    &=
    \B(v_M)\B(v_{M}^{-1})\s_{M+1}\s_{M+2}\cdots\s_{n-1}\B(\sigma') \\
    &=
    \B(v_M)\B(v_{M}^{-1})\B(uc),
\end{aligned}
\]
as desired. 
\end{proof} 

\begin{theorem}\label{lem:conjugator_rho}
Fix $1\leq p\leq n-1$ with $\gcd(p,n)=1$, and let $p'$ be the integer satisfying $1\leq p'<n$ and $pp'\equiv 1\pmod n$. Let $g\in \mathfrak{S}_n$ be the permutation defined so that $g(i)\equiv ip'\pmod n$ for all $i\in[n]$. Let ${\bf g}=\B(g)$. We have
\[
    {\bf g}\B(c)^p {\bf g}^{-1}\B(c)^{-1}
    =
    \B(x_p)\B(x_p^{-1}).
\] In particular, ${\bf g}$ is a realizable positive dual $(c,p)$-conjugator. 
\end{theorem}

\begin{proof}
For $0\leq j\leq p-1$, set $u_j=g c^{p-1-j}$.
We first claim that $\B(c){\bf g}=\B(u_0)\B(c)$.
Indeed, since $g c^p g^{-1}=c$, we have $cg=g c^p=u_0c$.
It remains to check that $\ell(cg)=\ell(c)+\ell(g)=\ell(u_0)+\ell(c)$. The permutation
$g$ fixes $n$, so left multiplication by $c$ preserves all inversions
among the first $n-1$ entries of the one-line notation of $g$ and creates
exactly $n-1$ new inversions involving the last position. Hence,
\[
    \ell(cg)=\ell(g)+n-1=\ell(g)+\ell(c).
\]
Similarly,
\[
    u_0(1)=g c^{p-1}(1)=g(p)=1,
\]
so multiplying $u_0$ on the right by $c$ moves this initial $1$ to the last position. Consequently,
\[
    \ell(u_0c)=\ell(u_0)+n-1=\ell(u_0)+\ell(c).
\]
This proves the claim.

Let $\Delta_M^2$ denote the full twist in the parabolic braid group generated by
$\mathbf{s}_1,\ldots,\mathbf{s}_{M-1}$. For $1\leq b\leq n-1$, we use the standard identity \[
\B(v_b)\B(v_b^{-1})=\Delta_b^{-2}\Delta_{b+1}^2.
\] Since $\Delta_b^2$ and $\Delta_{b+1}^2$ both commute with $\mathbf{s}_1,\ldots,\mathbf{s}_{b-1}$, it follows that $\B(v_b)\B(v_b^{-1})$ 
commutes with both $\B(v_a)$ and $\B(v_a^{-1})$ for all $0\leq a<b$. 

For $1\leq j\leq p-1$, set
$M_j=u_j(1)-1$.
Applying \cref{lem:B(u)B(c)} to $u_j$ gives
\[
    \B(v_{M_j})\B(v_{M_j}^{-1})\B(u_{j-1})=\B(u_j)\B(c).
\]
Repeatedly using this identity, together with the fact that $\B(c){\bf g}=\B(u_0)\B(c)$, we obtain
\[
\begin{aligned}
    \B(v_{M_1})\B(v_{M_1}^{-1})\B(c){\bf g}
    &=
    \B(u_1)\B(c)^2, \\
    \B(v_{M_2})\B(v_{M_2}^{-1})
    \B(v_{M_1})\B(v_{M_1}^{-1})\B(c){\bf g}
    &=
    \B(u_2)\B(c)^3, \\
    &\hspace{0.2cm}\vdots \\
    \B(v_{M_{p-1}})\B(v_{M_{p-1}}^{-1})\cdots
    \B(v_{M_1})\B(v_{M_1}^{-1})\B(c){\bf g}
    &=
    \B(u_{p-1})\B(c)^p.
\end{aligned}
\]
Since $u_{p-1}=g$, this gives
\[
\begin{aligned}
    {\bf g}\B(c)^p {\bf g}^{-1}\B(c)^{-1}
    =
    \B(v_{M_{p-1}})\B(v_{M_{p-1}}^{-1})\cdots
    \B(v_{M_1})\B(v_{M_1}^{-1}).
\end{aligned}
\]

We now identify the indices $M_j$. As $j$ runs from $1$ to $p-1$, the
numbers $u_j(1)$ are precisely
\[
    g(1),g(2),\ldots,g(p-1),
\]
in some order. We claim that
\begin{equation}\label{eq:sets} 
\{g(1),g(2),\ldots,g(p-1)\}
    =
    \{k_1+1,k_2+1,\ldots,k_{p-1}+1\}.
    \end{equation}
Indeed, for $1\leq i\leq p-1$, we have
\[
    k_i<\frac{in}{p}<k_i+1.
\]
Therefore, $0<(k_i+1)p-in<p$. Let $r_i=(k_i+1)p-in$.
Note that $1\leq r_i\leq p-1$ and that
${(k_i+1)p\equiv r_i \pmod n}$. 
Since $g$ is multiplication by $p'$ modulo $n$, this means $g(r_i)=k_i+1$.
This proves \eqref{eq:sets}. 

Using the commutation fact above, we may reorder the factors
$\B(v_{M_j})\B(v_{M_j}^{-1})$ by increasing value of $M_j$. Hence,
\[
\begin{aligned}
    {\bf g}\B(c)^p {\bf g}^{-1}\B(c)^{-1}
    &=
    \B(v_{k_1})\B(v_{k_1}^{-1})
    \B(v_{k_2})\B(v_{k_2}^{-1})
    \cdots
    \B(v_{k_{p-1}})\B(v_{k_{p-1}}^{-1}).
\end{aligned}
\]
Using the same commutation fact once more, we can move each
$\B(v_{k_i}^{-1})$ past all later factors
$\B(v_{k_j})\B(v_{k_j}^{-1})$ with $j>i$. Therefore,
\[
{\bf g}\B(c)^p {\bf g}^{-1}\B(c)^{-1}=
    \B(v_{k_1})\B(v_{k_2})\cdots \B(v_{k_{p-1}})
    \B(v_{k_{p-1}}^{-1})\cdots \B(v_{k_2}^{-1})\B(v_{k_1}^{-1}).
\]
By \cref{lem:reduced_x_rho}, the word $v_{k_1}v_{k_2}\cdots v_{k_{p-1}}$ is reduced and
represents $x_p$, so \[\B(v_{k_1})\B(v_{k_2})\cdots \B(v_{k_{p-1}})
    \B(v_{k_{p-1}}^{-1})\cdots \B(v_{k_2}^{-1})\B(v_{k_1}^{-1})=\B(x_p)\B(x_p^{-1}).\] 
This proves the desired identity. 
\end{proof}

\subsection{Type B}\label{subsec:Type_B}
We now describe an explicit construction of positive dual $(c,p)$-conjugators in type~$B_n$. Throughout this subsection, $W$ is the Coxeter group of type $B_n$
with simple reflections $s_1,s_2,\ldots,s_n$, where $(s_1s_2)^4=e$ and
$(s_is_{i+1})^3=e$ for all $2\leq i\leq n-1$. We choose the standard Coxeter
element $c=s_ns_{n-1}\cdots s_1$. The Coxeter number is $h=2n$. By
\cref{prop:p-range}, it suffices to treat integers $p$ satisfying
$1\leq p<2n$ and $\gcd(p,2n)=1$. In particular, $p$ is odd.

Let ${\sf c}=(s_n,s_{n-1},\ldots,s_1)$. The $c$-sorting word for $w_\circ$ is
${\sf w}_\circ(c)={\sf c}^n=(s_n,s_{n-1},\ldots,s_1)^n$. Let us write
${\sf c}^n=(u_1,u_2,\ldots,u_{n^2})$ and
$\inv({\sf c}^n)=(t_1,t_2,\ldots,t_{n^2})$. Thus,
$t_k=u_1u_2\cdots u_{k-1}u_ku_{k-1}\cdots u_2u_1$. We identify the
reflections $t_1,\ldots,t_{n^2}$ with the cells in a parallelogram-shaped
grid $\mathtt{Grid}_n$, as illustrated in~\cref{fig:type-b-grids}.

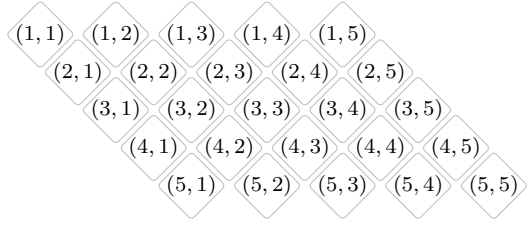
\begin{figure}[htbp]
\centering
\newcommand{\gridbox}[3][white]{%
  \pgfmathsetmacro{\cx}{#2 - 0.5*#3}%
  \pgfmathsetmacro{\cy}{0.5*#3}%
  \draw[gray!50, fill=#1, rounded corners=1.5pt]
    (\cx,\cy+0.45) -- (\cx+0.45,\cy) --
    (\cx,\cy-0.45) -- (\cx-0.45,\cy) -- cycle;
}
\newcommand{\gridcell}[3]{%
  \gridbox{#1}{#2}%
  \pgfmathsetmacro{\cx}{#1 - 0.5*#2}%
  \pgfmathsetmacro{\cy}{0.5*#2}%
  \node[grid label] at (\cx,\cy) {$#3$};
}

\tikzset{grid label/.style={font=\scriptsize}}
\begin{tikzpicture}[scale=1]
  \gridcell{0}{0}{(1,1)}
  \gridcell{1}{0}{(1,2)}
  \gridcell{2}{0}{(1,3)}
  \gridcell{3}{0}{(1,4)}
  \gridcell{4}{0}{(1,5)}
  \gridcell{0}{-1}{(2,1)}
  \gridcell{1}{-1}{(2,2)}
  \gridcell{2}{-1}{(2,3)}
  \gridcell{3}{-1}{(2,4)}
  \gridcell{4}{-1}{(2,5)}
  \gridcell{0}{-2}{(3,1)}
  \gridcell{1}{-2}{(3,2)}
  \gridcell{2}{-2}{(3,3)}
  \gridcell{3}{-2}{(3,4)}
  \gridcell{4}{-2}{(3,5)}
  \gridcell{0}{-3}{(4,1)}
  \gridcell{1}{-3}{(4,2)}
  \gridcell{2}{-3}{(4,3)}
  \gridcell{3}{-3}{(4,4)}
  \gridcell{4}{-3}{(4,5)}
  \gridcell{0}{-4}{(5,1)}
  \gridcell{1}{-4}{(5,2)}
  \gridcell{2}{-4}{(5,3)}
  \gridcell{3}{-4}{(5,4)}
  \gridcell{4}{-4}{(5,5)}
\end{tikzpicture}
\caption{The parallelogram-shaped $\mathtt{Grid}_5$, with the cell $(a,b)$ in row $a$ and down-right diagonal $b$.}
\label{fig:type-b-grids}
\end{figure}

We coordinatize this grid by horizontal rows and down-right diagonals (i.e.,
diagonals of slope $-1$). The rows are indexed $1,\ldots,n$ from top to
bottom, while the down-right diagonals are indexed $1,\ldots,n$ from left
to right. The cell $(a,b)\in\mathtt{Grid}_n$ lies in row $a$ and down-right
diagonal $b$, and it corresponds to the reflection $t_{(b-1)n+a}$. The
inversion sequence $\inv({\sf w}_\circ(c))=(t_1,\ldots,t_{n^2})$ induces a
total ordering on $\mathtt{Grid}_n$; namely, $(a,b)$ precedes $(a',b')$ if
$(b-1)n+a<(b'-1)n+a'$. For a set $I\subseteq\mathtt{Grid}_n$, let
\[
    \mathbb X_I=\prod_{(a,b)\in I}^{\longleftarrow}
    \P_c(t_{(b-1)n+a}),
\]
where the notation indicates that we take the product in the reverse of
this total order. By \cref{cor:wo}, we have
$\mathbb X_{\mathtt{Grid}_n}=\Delta^2$. We also have
$\B(c)\P_c(t_i)\B(c)^{-1}=\P_c(t_{i+n})$, where the indices are read
modulo $n^2$.

Let us first assume $1\leq p<n$. For $1\leq j\leq(p-1)/2$, set
$\eta_j=\left\lfloor 2nj/p\right\rfloor$, and define
\[
    I_j^{+}=
    \{(a,n-p+2j-a+1):1\leq a\leq \eta_j\}\quad\text{and}\quad
    I_j^{-}=
    \{(a,n-2j+1):\eta_j<a\leq n\}.
\]
Then set $I_j=I_j^+\sqcup I_j^-$ and
$J_p=I_1\sqcup I_2\sqcup\cdots\sqcup I_{(p-1)/2}$. Since $p<n$, we have $2j\leq\eta_j$, while
$n-p+2j-2nj/p=(n-p)(p-2j)/p>0$ gives
$\eta_j<n-p+2j$. Thus, all of these pairs are cells of
$\mathtt{Grid}_n$. The sets $I_j$ are pairwise disjoint. Indeed, if
$j<k$, then pieces of the same sign cannot meet because their second
coordinates differ. An intersection $I_j^+\cap I_k^-$ would force
$a=2(j+k)-p$ and $a\leq\eta_j<\eta_k<a$, while an intersection
$I_j^-\cap I_k^+$ would force
$a=2(j+k)-p\leq 2j-1\leq\eta_j<a$. Each $I_j$ contains exactly one
cell in each row. Hence, $J_p$ contains exactly
$(p-1)/2$ cells in each row. Geometrically, $I_j$ is the union of a
down-left segment of length $\eta_j$ and a down-right segment of length
$n-\eta_j$. We are interested in the pure braid
\[
    \mathbb X_{J_p}=
    \prod_{(a,b)\in J_p}^{\longleftarrow}
    \P_c(t_{(b-1)n+a}).
\]

\begin{example}
Fix $n=19$ and $p=11$ so that $W=B_{19}$ and $h=2n=38$. The values of
$\eta_j=\left\lfloor 2nj/p\right\rfloor$ for $j=1,2,3,4,5$ are
$3,6,10,13,17$. The set $J_p=J_{11}$ is shaded in \cref{fig:B19}.
\end{example}

\begin{figure}
\begin{center}\scalebox{0.8}{
\begin{tikzpicture}[scale=0.6]

\newcommand{\mybox}[3][white]{
  \pgfmathsetmacro{\cx}{#2 - 0.5*#3}
  \pgfmathsetmacro{\cy}{0.5*#3}
  \draw[gray!50, fill=#1, rounded corners=2pt]
    (\cx, \cy+0.45) -- (\cx+0.45, \cy) -- (\cx, \cy-0.45) -- (\cx-0.45, \cy) -- cycle;
}

\foreach \y in {0,...,-18} {
  \foreach \x in {0,...,18} {
    \mybox{\x}{\y}
  }
}

\foreach \i in {0,...,16} { \mybox[blue!40]{17-\i}{-\i} }
\foreach \i in {0,...,12} { \mybox[blue!40]{15-\i}{-\i} }
\foreach \i in {0,...,9}  { \mybox[blue!40]{13-\i}{-\i} }
\foreach \i in {0,...,5}  { \mybox[blue!40]{11-\i}{-\i} }
\foreach \i in {0,...,2}  { \mybox[blue!40]{9-\i}{-\i} }

\foreach \y in {-18,...,-3}  { \mybox[red!40]{17}{\y} }
\foreach \y in {-18,...,-6}  { \mybox[red!40]{15}{\y} }
\foreach \y in {-18,...,-10} { \mybox[red!40]{13}{\y} }
\foreach \y in {-18,...,-13} { \mybox[red!40]{11}{\y} }
\foreach \y in {-18,...,-17} { \mybox[red!40]{9}{\y} }

\end{tikzpicture}}
\end{center}
\caption{The set $J_{11}$ in $\mathtt{Grid}_{19}$. Cells from
$I_1^+\sqcup I_2^+\sqcup I_3^+\sqcup I_4^+\sqcup I_5^+$ are shaded blue,
while those in $I_1^-\sqcup I_2^-\sqcup I_3^-\sqcup I_4^-\sqcup I_5^-$
are shaded red.}
\label{fig:B19}
\end{figure}
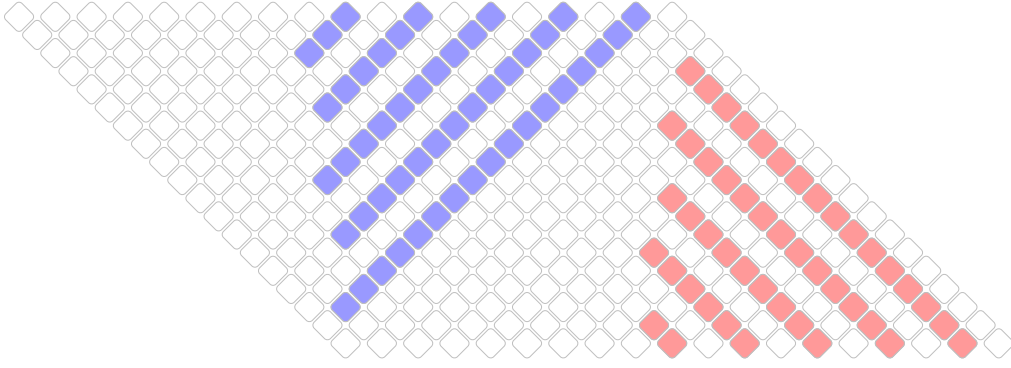

The sets $J_p$ have a useful interpretation as inversion sets when
$p<n$. Realize $W$ as the group of signed permutations in one-line
notation, where $s_1$ changes the sign of the first entry and $s_i$
interchanges the entries in positions $i-1$ and $i$ for $2\leq i\leq n$.
We identify each reflection with its positive root and use
\[
    \Phi^+=\{e_i:1\leq i\leq n\}
    \sqcup\{e_v-e_u,e_v+e_u:1\leq u<v\leq n\}.
\]

\begin{lemma}\label{lem:type-B-inversion-set}
Assume $1\leq p<n$. For $0\leq k\leq p-1$, set
$\lambda_k=\left\lfloor nk/p\right\rfloor+1$. Define a signed permutation
$x_{n,p}$ by
\[
    x_{n,p}(\lambda_k)=(-1)^k(p-k)
    \qquad(0\leq k\leq p-1),
\]
and fill the remaining positions from left to right with
$p+1,p+2,\ldots,n$. Then
\[
    \Inv(x_{n,p})=
    \{t_{(b-1)n+a}:(a,b)\in J_p\}.
\]
Consequently, $\mathbb X_{J_p}=\B(x_{n,p})\B(x_{n,p}^{-1})$.
\end{lemma}

\begin{proof}
The integers $\lambda_0,\lambda_1,\ldots,\lambda_{p-1}$ are strictly
increasing and belong to $[n]$, so the stated rule defines a signed
permutation. Fix $1\leq j\leq(p-1)/2$, and write
$\eta=\eta_j$ and $u=p-2j$. We have
$\lambda_{2j}=\eta+1$ and $\lambda_{p-2j}=n-\eta$; the second identity
uses the fact that $2nj/p$ is not an integer.

Since $c(e_1)=-e_n$ and $c(e_i)=e_{i-1}$ for $2\leq i\leq n$, a direct
calculation from the inversion sequence of ${\sf c}^n$ shows that the
roots corresponding to the cells in $I_j^+$ are
\[
    \{e_{u+a}-e_u:1\leq a\leq\eta\},
\]
while the roots corresponding to the cells in $I_j^-$ are
\[
    \{e_{2j}\}\sqcup
    \{e_{2j}+e_v:2j<v\leq n-\eta+2j-1\}.
\]
Indeed, the cell $(a,n-p+2j-a+1)$ contributes $e_{u+a}-e_u$. If
$\eta<a<n$, then the cell $(a,n-2j+1)$ contributes
$e_{2j}+e_{n-a+2j}$, while the cell $(n,n-2j+1)$ contributes $e_{2j}$.

The usual inversion criterion says that
$e_u$ is an inversion of $x$ if and only if $x^{-1}(u)<0$, while, for $u<v$,
the roots $e_v-e_u$ and $e_v+e_u$ are inversions of $x$ if and only if
$x^{-1}(v)<x^{-1}(u)$ and $x^{-1}(v)<-x^{-1}(u)$, respectively.

The positive value $u=p-2j$ occurs in position $\eta+1$ of $x_{n,p}$.
The selected entries before this position have absolute values
$u+1,u+2,\ldots,p$, and the filler entries before it have values
$p+1,p+2,\ldots,u+\eta$. Thus, the values larger than $u$ that occur
before this position are exactly $u+1,u+2,\ldots,u+\eta$. The inversion
criterion shows that the inversions whose smaller index is $u$ are precisely
$e_{u+1}-e_u,\ldots,e_{u+\eta}-e_u$; there are no sum-root inversions
involving $u$.

The value $2j$ occurs with negative sign in position $n-\eta$. The selected
entries before this position have absolute values $2j+1,2j+2,\ldots,p$,
and the filler entries before it have values
$p+1,p+2,\ldots,n-\eta+2j-1$. Hence, the values larger than $2j$ that
occur before this position are exactly
$2j+1,2j+2,\ldots,n-\eta+2j-1$. The inversion criterion shows that the
inversions whose smaller index is $2j$ are precisely the short root $e_{2j}$
and the sum roots $e_{2j}+e_v$ for
$2j<v\leq n-\eta+2j-1$; there are no difference-root inversions involving
$2j$.

As $j$ ranges from $1$ to $(p-1)/2$, the values $p-2j$ are the odd
integers in $[p-1]$, while the values $2j$ are the even integers in
$[p-1]$. The entries $p+1,p+2,\ldots,n$ are positive and occur in
increasing order. Thus, the roots listed above are all of the inversions of
$x_{n,p}$, and their union is exactly the set of roots indexed by $J_p$.
The final assertion follows from \cref{lem:inversions}.
\end{proof}

We now define $J_p$ in the remaining range $n<p<2n$. Let
$p^*=2n-p$, so $1\leq p^*<n$ and $\gcd(p^*,2n)=1$. The preceding
construction gives a subset $J_{p^*}$ of $\mathtt{Grid}_n$. The inequality
${\eta_j<n-p^*+2j}$ shows that $J_{p^*}$ is disjoint from the leftmost
down-right diagonal
\[\mathcal C=\{(a,1):1\leq a\leq n\}.\] Define
\[
    J_p=\mathtt{Grid}_n\setminus(J_{p^*}\sqcup\mathcal C).
\]
The set $J_p$ contains
$n-(p^*-1)/2-1=(p-1)/2$ cells in each row. Thus, for every admissible
$p$, the braid $\mathbb X_{J_p}$ is a product of exactly
$n(p-1)/2$ positive pure dual generators.

We next prove the power identity needed to apply
\cref{lem:Bessis_Springer}. Let $\widetilde W$ be the Coxeter group of type
$A_{2n-1}$ that unfolds $W$, and let
$\iota:\B_W\hookrightarrow\B_{\widetilde W}$ be the folding embedding from \cref{subsec:folding}. Let $\widetilde c$ be the standard Coxeter element of $\widetilde W$
obtained by unfolding $c$, so $\iota(\B(c))=\B(\widetilde c)$. Label the entries in the cycle of
$\widetilde c$ by $v_k$ for $k\in\mathbb Z/2n\mathbb Z$ so that
$\widetilde c(v_k)=v_{k+1}$, and let $t_{k,\ell}$ denote the transposition
that interchanges $v_k$ and $v_\ell$.

\begin{lemma}\label{lem:type-B-unfolding}
For $(a,b)\in\mathtt{Grid}_n$, with all subscripts read modulo $2n$, we
have
\[
\iota\bigl(\P_c(t_{(b-1)n+a})\bigr)=
\begin{cases}
\P_{\widetilde c}(t_{b-1,b-1+a})
\P_{\widetilde c}(t_{b-1+n,b-1+n+a}) & \text{if }a<n,\\
\P_{\widetilde c}(t_{b-1,b-1+n}) & \text{if }a=n.
\end{cases}
\]
The two factors in the first case commute.
\end{lemma}

\begin{proof}
For $b=1$, the cell $(a,1)$ corresponds to the root
$e_n-e_{n-a}$ when $a<n$ and to the short root $e_n$ when $a=n$. Under
the folding, these roots lift, respectively, to the two antipodal
transpositions $t_{0,a},t_{n,n+a}$ and to the transposition $t_{0,n}$.
Thus, the formula for $b=1$ follows from the pure-generator folding identity
in~\eqref{eq:iota2}. The two reflections in the first case are supported on
disjoint antipodal pairs, so their dual lifts commute. The formula for
arbitrary $b$ follows by conjugating by $\B(c)^{b-1}$.
\end{proof}

\begin{proposition}\label{prop:type-B-rational-cut}
Assume $1\leq p<n$. The braid
$\iota(\mathbb X_{J_p})\B(\widetilde c)$ is conjugate to
$\B(\widetilde c)^p$ in $\B_{\widetilde W}$.
\end{proposition}

\begin{proof}
View $\B_{\widetilde W}$ as the braid group on the $2n$ points
$v_0,v_1,\ldots,v_{2n-1}$ arranged cyclically. The braid
$\B(\widetilde c)^p$ is obtained by cutting the $(2n,p)$-torus braid along
a standard meridian. Cutting the same closed braid along a different
meridian changes the resulting braid by conjugation: moving the cut gives
a path between the two base configurations, and the two resulting loops
differ by conjugation by this path. We use the centrally symmetric
staircase meridian that follows the line segment from $(0,0)$ to $(2n,p)$.

At height $j$, where $1\leq j\leq(p-1)/2$, the line has horizontal
coordinate $2nj/p$, which lies strictly between $\eta_j$ and
$\eta_j+1$. Pairing height $j$ with the antipodal height $p-j$, the $n$
antipodal orbits of crossings are indexed by $a\in[n]$. Their endpoints
are as follows:
\[
\begin{array}{c|c|c}
 &1\leq a\leq\eta_j&\eta_j<a\leq n\\ \hline
\text{first endpoint}
&v_{n-p+2j-a}&v_{n-2j}\\
\text{second endpoint}
&v_{n-p+2j}&v_{n-2j+a}
\end{array}.
\]
In the first case, this is the antipodal orbit associated by
\cref{lem:type-B-unfolding} with the cell
\[(a,n-p+2j-a+1)\in I_j^+.\] In the second case, it is the orbit associated
with \[(a,n-2j+1)\in I_j^-.\] When $a=n$, the orbit consists of one
antipodal chord, as in the second case of \cref{lem:type-B-unfolding}.
The staircase encounters these crossings in decreasing order of
$(b-1)n+a$, and after the pure crossing factors it contributes the final
cyclic motion $\B(\widetilde c)$. Hence, this cut gives
$\iota(\mathbb X_{J_p})\B(\widetilde c)$, while the standard cut gives
$\B(\widetilde c)^p$.
\end{proof}

\begin{proposition}\label{prop:XJp}
For every integer $p$ satisfying $1\leq p<2n$ and $\gcd(p,2n)=1$, we have
\[
    \left(\mathbb X_{J_p}\B(c)\right)^{2n}=\Delta^{2p}.
\]
\end{proposition}

\begin{proof}
Assume first that $p<n$. By \cref{prop:type-B-rational-cut}, there exists
$\widetilde{\bf g}\in\B_{\widetilde W}$ such that
\[
    \iota(\mathbb X_{J_p}\B(c))
    =\widetilde{\bf g}\B(\widetilde c)^p\widetilde{\bf g}^{-1}.
\]
Since $\B(\widetilde c)^{2n}=\iota(\Delta^2)$ is central, it follows that
\[
    \iota\left(\left(\mathbb X_{J_p}\B(c)\right)^{2n}\right)
    =\B(\widetilde c)^{2np}
    =\iota(\Delta^{2p}).
\]
The map $\iota$ is injective, so the result follows in this case.

Now assume that $n<p<2n$, and let $p^*=2n-p$. Let $\dagger$ be the
automorphism of $\B_W$ that sends each Artin generator to its inverse.
For every $w\in W$, we have
$\B(w)^\dagger=\B(w^{-1})^{-1}$; in particular,
$\Delta^\dagger=\Delta^{-1}$. Set $x=x_{n,p^*}$. Since $\Delta$ is
central in type~$B_n$, \cref{lem:type-B-inversion-set} gives
\[
\begin{aligned}
    \Delta^2\left(\mathbb X_{J_{p^*}}\B(c)\right)^\dagger
    &=\Delta^2\B(x^{-1})^{-1}\B(x)^{-1}\B(c^{-1})^{-1}\\
    &=\B(w_\circ x)\B(w_\circ x^{-1})\B(c^{-1})^{-1}.
\end{aligned}
\]
Indeed, we have 
$\Delta\B(x^{-1})^{-1}=\B(w_\circ x)$ and
$\Delta\B(x)^{-1}=\B(w_\circ x^{-1})$.

The long element acts as $-1$ in type~$B_n$, so
$\Inv(w_\circ x)=T\setminus\Inv(x)$. Moreover,
$(w_\circ x)^{-1}=w_\circ x^{-1}$. Hence, \cref{lem:inversions} yields
\[
    \B(w_\circ x)\B(w_\circ x^{-1})
    =\mathbb X_{\mathtt{Grid}_n\setminus J_{p^*}}
    =\mathbb X_{J_p}\mathbb X_{\mathcal C}.
\]
The second equality uses the definition of $J_p$ and the fact that the
cells in $\mathcal C$ come last in the reverse grid order. Since
$\mathcal C$ indexes the inversion sequence of ${\sf c}$,
\cref{lem:inversions} also gives
$\mathbb X_{\mathcal C}=\B(c)\B(c^{-1})$. Therefore,
\[
    \Delta^2\left(\mathbb X_{J_{p^*}}\B(c)\right)^\dagger
    =\mathbb X_{J_p}\B(c).
\]
Using the result already proved for $p^*$ and the centrality of $\Delta^2$,
we conclude that
\[
\begin{aligned}
    \left(\mathbb X_{J_p}\B(c)\right)^{2n}
    &=\Delta^{4n}
      \left(\left(\mathbb X_{J_{p^*}}\B(c)\right)^{2n}\right)^\dagger\\
    &=\Delta^{4n}\left(\Delta^{2p^*}\right)^\dagger
      =\Delta^{4n-2p^*}
      =\Delta^{2p}
\end{aligned}
\] as desired. 
\end{proof}

\begin{theorem}\label{thm:conjugator_B}
Let $W$ be of type $B_n$, and let $c=s_ns_{n-1}\cdots s_1$. Let $p$ be
an integer such that $1\leq p\leq 2n-1$ and $\gcd(p,2n)=1$. There exists
a positive dual $(c,p)$-conjugator ${\bf g}$ such that
\[
    \XX
    :={\bf g}\B(c)^p{\bf g}^{-1}\B(c)^{-1}
    =\mathbb X_{J_p}
    =\prod_{(a,b)\in J_p}^{\longleftarrow}
    \P_c(t_{(b-1)n+a}).
\]
If $p<n$, then this conjugator is realizable in the sense that
$\mathbb X_{J_p}=\B(x_{n,p})\B(x_{n,p}^{-1})$.
\end{theorem}

\begin{proof}
Let $\bbeta=\mathbb X_{J_p}\B(c)$. By \cref{prop:XJp}, we have
$\bbeta^{2n}=\Delta^{2p}$. Since $h=2n$, \cref{lem:Bessis_Springer}
implies that there exists ${\bf g}\in\B_W$ such that
${\bf g}\B(c)^p{\bf g}^{-1}=\bbeta$. Hence,
${\bf g}\B(c)^p{\bf g}^{-1}\B(c)^{-1}=\mathbb X_{J_p}$. The expression
defining $\mathbb X_{J_p}$ is a product of $n(p-1)/2$ positive pure dual
generators, so ${\bf g}$ is a positive dual $(c,p)$-conjugator. The final
assertion follows from \cref{lem:type-B-inversion-set}.
\end{proof}

\subsection{Type D}\label{subsec:Type_D}
Let $W$ be the Coxeter group of type
$D_n$ with simple reflections $s_0,s_1,\ldots,s_{n-1}$, where $s_0$ and $s_1$ are the two fork vertices, both
adjacent to $s_2$ in the Coxeter graph. Thus, we have $(s_0s_2)^3=e$ and $(s_is_{i+1})^3=e$ for all $1\leq i\leq n-2$. The Coxeter number of $W$ is $h=2n-2$. 

Let $\theta$ be the automorphism of the Coxeter graph of $W$ that swaps $s_0$ and $s_1$ and fixes
all other simple reflections. The folded Coxeter group $W^{\fold}$ is of type $B_{n-1}$. It has simple reflections $s_1^{\fold},s_2^{\fold},\ldots,s_{n-1}^{\fold}$, 
where $s_1^{\fold}=s_0s_1$ and $s_i^{\fold}=s_i$ for $2\leq i\leq n-1$. This folding induces the embedding $\iota:\B_{W^{\fold}}\hookrightarrow \B_{W}$. 

Consider the standard Coxeter element
$c^{\fold}=s^{\fold}_{n-1}s^{\fold}_{n-2}\cdots s^{\fold}_1$ of $W^{\fold}$. We can view $c^{\fold}$ as the standard Coxeter element $c=s_{n-1}s_{n-2}\cdots s_1s_0$ of $W$. 

Suppose $p$ is an integer coprime to $2n-2$ satisfying $1\leq p<2n-2$. Let $J_p$ be the subset of $\mathtt{Grid}_{n-1}$ defined in \cref{subsec:Type_B}. In this subsection, let us write $t^{\fold}_{(a,b)}$ for the reflection in $W^{\fold}$ corresponding to the cell $(a,b)$ of $\mathtt{Grid}_{n-1}$. \cref{thm:conjugator_B} ensures that there is a positive dual $(c^{\fold},p)$-conjugator ${\bf g}^{\fold}$ such that
\[
    \mathbbm{c}^{\fold}_{p,{\bf g}^{\fold}}:=
    {\bf g}^{\fold}\B(c^{\fold})^p({\bf g}^{\fold})^{-1}\B(c^{\fold})^{-1}
    =
    \mathbb{X}_{J_p}=
    \prod_{(a,b)\in J_p}^{\longleftarrow} \P_{c^{\fold}}(t^{\fold}_{(b-1)(n-1)+a}).
\] The following corollary is immediate from \cref{prop:folding}. 

\begin{corollary}
Let $p$ be an integer with $\gcd(p,2n-2)=1$ and $1\leq p<2n-2$. Let ${\bf g}^{\fold}$ be a $(c^{\fold},p)$-conjugator such that
\[
    \mathbbm{c}^{\fold}_{p,{\bf g}^{\fold}}:=
    {\bf g}^{\fold}\B(c^{\fold})^p({\bf g}^{\fold})^{-1}\B(c^{\fold})^{-1}
    =
    \mathbb{X}_{J_p}=
    \prod_{(a,b)\in J_p}^{\longleftarrow} \P_{c^{\fold}}(t^{\fold}_{(b-1)(n-1)+a}).
\] 
The element ${\bf g}=\iota({\bf g}^{\fold})$ is a positive dual $(c,p)$-conjugator, and \[
    \mathbbm{c}_{p,{\bf g}}
    :=
    {\bf g} \B(c)^p{\bf g}^{-1}\B(c)^{-1}=\iota(\mathbbm{c}^{\fold}_{p,{\bf g}^{\fold}})=\prod_{(a,b)\in J_p}^{\longleftarrow}\prod_{t\in\mathrm{Orb}_\theta(t_{(b-1)(n-1)+a}^{\fold})}\P_c(t). 
\]
\end{corollary}

\subsection{Exceptional Crystallographic Types}\label{sec:exceptional} 
In this section, we give the results of our computations for exceptional crystallographic types. We consider $p$ such that $2\leq p\leq h-2$ and $\gcd(p,h)=1$ (we ignore the cases $p\equiv\pm 1\pmod h$ since those are handled in \cref{sec:fuss}). We handle types $F_4, E_6, E_7,$ and $E_8$ (we can ignore $G_2$ since its Coxeter number is $h=6$ and there is no $p$ coprime to $6$ satisfying $2\leq p\leq 6-2$).  As before, we performed explicit calculations in each type for a particular standard Coxeter element since one can then conjugate in order to obtain similar results for all other standard Coxeter elements. Our GAP3 with CHEVIE~\cite{GH96} code is available at~\cite{code}; as an example, we were able to produce a reflection subword model for rational Catalan objects in type $E_8$ with $p=7$ in about a second.

Write $\mathrm{inv}({\sf w}_\circ(c))=(t_1,\ldots,t_N)$.  For each type and each $p$ coprime to $h$ with $2\leq p\leq h-2$, we found an element $g \in W$ such that $gc^pg^{-1}=c$, we chose a braid lift $\mathbf{g}$ of $g$, and we implemented the Reidemeister--Schreier method to rewrite the pure braid $\XX=\mathbf{g}\bc^p\mathbf{g}^{-1}\bc^{-1}$ using the generators in $\mathbb T$. These factorizations of $\XX$ are listed in~\Cref{tab:exceptional}, formatted so that $k$ represents the generator $\tt_k$ and $\textcolor{barclr}{\overline{k}}$ stands for $\tt_k^{-1}$.  When possible, we tried to find lifts that gave positive expressions resulting in unsigned reflection subword models; we did not succeed in type~$E_7$ for $p=5$ or in type~$E_8$, although there are almost certainly simplifications that could be applied to these non-positive expressions.

\begin{figure}[htbp]
\centering
\includegraphics[height=62.422mm]{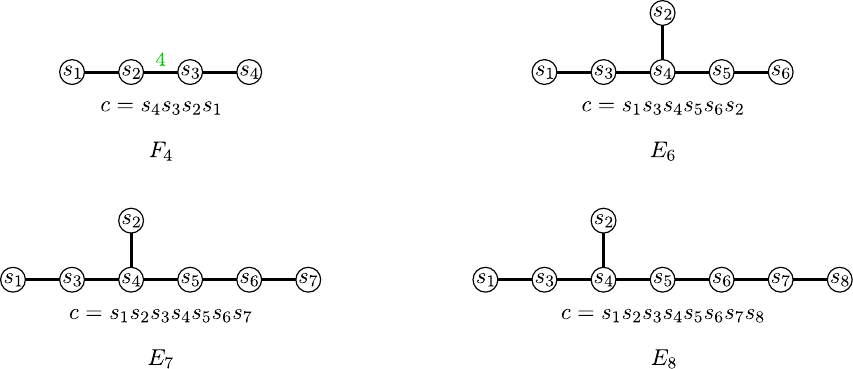}
\caption{Coxeter graphs and Coxeter elements in~\Cref{ex:f4} and~\Cref{tab:exceptional}. Edges without visible labels are assigned the label 3. }
\label{fig:exceptional-dynkin-conventions}
\end{figure}

\begin{example}\label{ex:f4}
We give an in-depth illustration of our calculations in type $F_4$ for $c=s_4s_3s_2s_1$.  In this case, $h=12$, so we take $p$ to be either $5$ or $7$. 
\begin{itemize}
\item For $p=5$, we have the realizable positive dual $(c,p)$-conjugator
\[\mathbf{g}=\s_3\s_2\s_1\s_3\s_2\s_4\s_3\s_2\s_1\s_4^{-1}\s_3^{-1}\s_4^{-1},\] and
\[\XX =\mathbf{g} \bc^5 \mathbf{g}^{-1} \bc^{-1}= \B(x)\cdot \B(x^{-1})
=\tt_{19} \tt_{17} \tt_{16} \tt_{14} \tt_8 \tt_7 \tt_6 \tt_5,\]
where $x=s_3s_2s_1s_3s_2s_4s_3s_2$. 
\item For $p=7$, we have the realizable positive dual $(c,p)$-conjugator
\[\mathbf{g}=\s_3\s_2\s_1\s_3\s_2\s_4\s_3\s_2\] and
\[\XX=\mathbf{g} \bc^7 \mathbf{g}^{-1} \bc^{-1}=\B(x)\cdot \B(x^{-1})
=\tt_{24} \tt_{22} \tt_{19} \tt_{17}\tt_{16} \tt_{15} \tt_{14} \tt_{13} \tt_{8} \tt_7 \tt_6 \tt_5,\]
where $x=s_1s_3s_2s_1s_3s_2s_3s_4s_3s_2s_1s_3$. \qedhere
\end{itemize}
\end{example}

\begin{table}[htbp]
\centering
\renewcommand{\arraystretch}{1.3}
\begin{tabular}{ll>{\RaggedRight\arraybackslash\footnotesize}p{10cm}}
\toprule
$(W,c)$ & $p$ & $\XX$\\
\midrule
$F_4,4321$ & 5 & \Xword{19\pbsep 17\pbsep 16\pbsep 14\pbsep 8\pbsep 7\pbsep 6\pbsep 5}\\
 & 7 & \Xword{24\pbsep 22\pbsep 19\pbsep 17\pbsep 16\pbsep 15\pbsep 14\pbsep 13\pbsep 8\pbsep 7\pbsep 6\pbsep 5}\\
\midrule
$E_6,134562$ & 5 & \Xword{35\pbsep 34\pbsep 29\pbsep 27\pbsep 26\pbsep 25\pbsep 17\pbsep 16\pbsep 15\pbsep 12\pbsep 13\pbsep 8}\\
 & 7 & \Xword{35\pbsep 34\pbsep 29\pbsep 28\pbsep 27\pbsep 26\pbsep 25\pbsep 24\pbsep 20\pbsep 17\pbsep 16\pbsep 15\pbsep 13\pbsep 12\pbsep 11\pbsep 9\pbsep 8\pbsep 7}\\
\midrule
$E_7,1234567$ & 5 & \Xword{\ov{63}\pbgap 18\pbsep 6\pbsep 5\pbsep 1\pbsep 2\pbsep 63\pbsep 62\pbsep 60\pbsep 59\pbsep 57\pbsep 47\pbsep 42\pbsep 41\pbsep 40\pbgap \ov{41}\pbsep \ov{47}\pbsep \ov{1}\pbgap 41\pbsep 35\pbsep 38\pbsep 36\pbsep 31\pbsep 7\pbsep 43\pbsep 30\pbgap \ov{38}\pbsep \ov{42}\pbsep \ov{43}\pbsep \ov{6}}\\
 & 7 & \Xword{41\pbsep 39\pbsep 37\pbsep 35\pbsep 33\pbsep 31\pbsep 29\pbsep 21\pbsep 20\pbsep 19\pbsep 18\pbsep 16\pbsep 15\pbsep 10\pbsep 2\pbsep 63\pbsep 62\pbsep 61\pbsep 60\pbsep 57\pbsep 52}\\
 & 11 & \Xword{43\pbsep 37\pbsep 32\pbsep 26\pbsep 20\pbsep 14\pbsep 30\pbsep 31\pbsep 29\pbsep 24\pbsep 18\pbsep 12\pbsep 6\pbsep 63\pbsep 22\pbsep 16\pbsep 11\pbsep 5\pbsep 62\pbsep 56\pbsep 10\pbsep 8\pbsep 2\pbsep 3\pbsep 60\pbsep 54\pbsep 48\pbsep 42\pbsep 58\pbsep 53\pbsep 41\pbsep 57\pbsep 52\pbsep 40\pbsep 28}\\
 & 13 & \Xword{1\pbsep 55\pbsep 60\pbsep 58\pbsep 54\pbsep 50\pbsep 63\pbsep 59\pbsep 49\pbsep 48\pbsep 47\pbsep 46\pbsep 45\pbsep 44\pbsep 42\pbsep 41\pbsep 39\pbsep 37\pbsep 35\pbsep 33\pbsep 27\pbsep 43\pbsep 38\pbsep 32\pbsep 30\pbsep 26\pbsep 24\pbsep 21\pbsep 20\pbsep 19\pbsep 18\pbsep 22\pbsep 16\pbsep 14\pbsep 36\pbsep 17\pbsep 13\pbsep 11\pbsep 5\pbsep 9\pbsep 15\pbsep 10}\\
\bottomrule
\end{tabular}
\caption{Pure braids giving reflection subword models for rational Catalan noncrossing partitions in exceptional types.  With
$\inv({\sf w}_\circ(c))=(t_1,\ldots,t_N)$, the integer
$k$ represents the pure generator $\tt_k$ associated with $t_k$, while $\textcolor{barclr}{\overline{k}}$ represents $\tt_k^{-1}$.  The longer
$E_8$ words (with negative generators) are supplied as machine-readable ancillary files, available at~\cite{code}.}
\label{tab:exceptional}
\end{table}

\subsection{Noncrystallographic Discussion}\label{sec:noncrystallographic}
In this section, we consider $W$ of noncrystallographic type ($I_2(m)$, $H_3$, or $H_4$), and $p$ coprime to the Coxeter number $h$ with $p \not\equiv \pm 1 \pmod h$.

Since the multiset of eigenvalues of $c$ is $\{e^{2\pi i e_i/h}\}_{i=1}^r$, the multiset of eigenvalues of $c^p$ is $\{e^{2\pi i p e_i/h}\}_{i=1}^r$.  In contrast with crystallographic type, these two multisets are not always equal, so $c$ and $c^p$ are not always conjugate.\footnote{As a simple example, take $W=I_2(5)$ so that $h=5$, $e_1=1$, and $e_2=4$. Take $p$ to be $2$ or $3$ so that a standard Coxeter element $c$ has eigenvalues $\zeta,\zeta^-1$ while $c^p$ has eigenvalues $\zeta^p,\zeta^{-p}$.}  This is the core of the problem.
\begin{itemize}
\item For a \emph{standard} Coxeter element $c$---by identifying the lifts of the simple reflections in both the dual braid group and the Artin braid group---we understand how to lift all reflections from $W$ and thus how to descend down to the Hecke algebra. 
\item When $c^p$ is not standard but remains \emph{conjugate} to a standard Coxeter element, we were able to simply lift the conjugation to an \emph{inner} automorphism of the Artin braid group in order to pass to the Hecke algebra. 
\item When $c^p$ is not conjugate to a standard Coxeter element, we have no way to relate the dual braid group to the usual braid group, so we cannot pass to the Hecke algebra.
\end{itemize}

\subsubsection{Failure of Galois twists}
One might try a suitable \defn{Galois twist} $\sigma$ that sends the reflection representation $V$ to the twisted representation $V^\sigma$, mapping $c$ to a matrix that \emph{is} conjugate to $c^p$.  This acts on the reflections as a reflection automorphism of the Coxeter group, and it induces an isomorphism of the noncrossing partition lattices $\NC(W,c)\simeq \NC(W,c)^\sigma= \NC(W,c^p)$~\cite{reiner2017non}.  All remains well at the level of the Coxeter group and the Hecke algebra (where we can also permute representations~\cite{opdam1998complex}).  We can even symbolically rewrite the presentation of the $c$-dual braid group using $c^p$ since it is presented using the data of the noncrossing partition lattice. 

However, we do not understand how (or if!) the Galois twist acts on the Artin braid group.  Certainly we cannot see how it acts when we view the Artin braid group as a fundamental group since $\mathbb{C}$ is too rigid to have a nontrivial twist by $\sigma$.

Because we do not know how to relate the $c^p$-dual braid group (for $p\not\equiv \pm 1 \pmod h$) directly to the standard Artin braid group, we cannot relate the $c^p$-dual braid group to the Hecke algebra. Thus, we cannot think of $c^p$ as a standard Coxeter element for some choice of simple reflections, so we cannot apply~\Cref{thm:main} to compute our trace.  That is, we do not know how to determine the image in the Hecke algebra of a $c^p$-dual lift of a reflection.

\subsubsection{Failure of folding}
Another idea is to work in a crystallographic group that \emph{folds} to our noncrystallographic group; for example, $A_{m-1}$ folds to $I_2(m)$, or $E_8$ folds to $H_4$.\footnote{This is a more general notion of folding than the one considered in \cref{subsec:folding}.}  Again, everything works perfectly at the level of the braid groups, dual braid groups (for bipartite $c$), and Coxeter groups.  But we again run into problems when we pass to the Hecke algebra quotient of the group algebra of the braid group---even for folding from $A_3$ to $B_2$, one sees that $(\T_1\T_3)^2 = (q-1)^2\T_{13} + q(q-1)(\T_1+\T_3)+q^2$ expands using undesirable terms.

\subsubsection{Further examples}
We checked computationally that, for every standard Coxeter element
$c$ of $H_3$, there is no triple
$(t_1,t_2,t_3)\in T^3$ such that
\[
[\T_e]\H(
\Delta^2\tt_1\tt_2\tt_3
\B(c))
=
[\T_e]\H(
\Delta^2\B(c)^3).
\]
Thus, we do not obtain a positive
reflection subword model for $\mathrm{Cat}_3(H_3)$. We also do
not currently know how to obtain such a model in type $H_4$. 

For the dihedral groups, we can find braids in $\B_{I_2(m)}$ whose
images in $I_2(m)$ are equal to $c$ and whose Hecke traces agree with
that of $\B(c)^p$. These braids cannot be conjugate to
$\B(c)^p$ because their images in $I_2(m)$ are $c$, rather
than $c^p$. We omit the details because the resulting constructions
do not appear to be particularly illuminating.

\section{Parking Models}\label{sec:parking}

\subsection{Parking Traces}
The following theorem is a consequence of results in the articles~\cite{galashin2024rational,trinh2026partial}.  As with Catalan objects, we will obtain two different parking models using~\Cref{lem:coeff-delta} and $\bc^{\pm p}$.
\begin{theorem}[{\cite{galashin2024rational,trinh2026partial}}]\label{thm:GLTW_parking}
For a standard Coxeter element $c$ of $W$ and for $p>0$ coprime to $h$,
\begin{align*}[\T_e]\H\!\left( \sum_{w \in W} q^{-\ell(w)} \B(w)\B(w^{-1}) \bc^p\right) &= (q-1)^r\Park_p(W;q),\\
[\T_e]\H\!\left( \sum_{w \in W} q^{-\ell(w)} \B(w)\B(w^{-1}) \bc^{-p}\right) &= (-q)^{-rp}(q-1)^r\Park_p(W;q). 
\end{align*}
\end{theorem}

\begin{definition}
Define the element $\PF$ of the Hecke algebra $\H_W$ by
\[\PF = \H\!\left(\sum_{w \in W} q^{-\ell(w)} \B(w) \B(w^{-1})\right).\]
\end{definition}

\begin{lemma}\label{lem:PF_central}
The element $\PF$ is in the center of $\H_W$. 
\end{lemma}
\begin{proof}
Recall that the map $\mathrm{x}\mapsto[\T_e]\mathrm{x}$ is a trace on $\H_W$. It follows from \cref{lem:coeff-identity} that $\{\T_w\}_{w \in W}$ and $\{q^{-\ell(w)}\T_{w^{-1}}\}_{w \in W}$ are dual bases with respect to this trace. 

Fix some element $\mathrm{x} \in \H_W$.  We first expand \[\mathrm{x}\, \T_w = \sum_{y \in W} q^{-\ell(y)} [\T_e](\mathrm{x}\, \T_w \T_{y^{-1}}) \T_y.\] Multiplying on the right by $q^{-\ell(w)}\T_{w^{-1}}$ and summing over $w$ yields 
\[
\mathrm{x}\,\PF = \sum_{w \in W} q^{-\ell(w)} \mathrm{x}\, \T_w \T_{w^{-1}} = \sum_{w \in W} \sum_{y \in W} q^{-\ell(y)-\ell(w)}[\T_e](\mathrm{x}\, \T_w \T_{y^{-1}}) \T_y \T_{w^{-1}}.
\]
By the trace property, we have $[\T_e](\mathrm{x}\, \T_w \T_{y^{-1}}) = [\T_e](\T_{y^{-1}} \,\mathrm{x}\, \T_w)$. Substituting this and reversing the summation order, we obtain that
\[
\mathrm{x}\,\PF = \sum_{y \in W} \T_y \left( \sum_{w \in W} q^{-\ell(y)-\ell(w)}[\T_e](\T_{y^{-1}}\, \mathrm{x}\, \T_w) \T_{w^{-1}} \right).
\]
The inner sum is exactly the expansion of $q^{-\ell(y)}\T_{y^{-1}}\,\mathrm{x}$ in the basis $q^{-\ell(w)}\T_{w^{-1}}$, so \[\sum_{w \in W} q^{-\ell(y)-\ell(w)} [\T_e](\T_{y^{-1}}\, \mathrm{x}\, \T_w) \T_{w^{-1}} = q^{-\ell(y)} \T_{y^{-1}}\, \mathrm{x}.\] Thus,
\[
\mathrm{x}\,\PF = \sum_{y \in W}q^{-\ell(y)} \T_y (\T_{y^{-1}}\,\mathrm{x}) = \left(\sum_{y \in W} q^{-\ell(y)} \T_y \T_{y^{-1}}\right)\,\mathrm{x} = \PF\,\mathrm{x}.
\]
This confirms that $\PF$ is central in $\H_W$. 
\end{proof}

\subsection{Fuss--Catalan Parking Models}

We now specialize \cref{thm:GLTW_parking} to the case when $p=mh+1$. As in \cref{sec:fuss},
the positive and negative trace evaluations produce two different
reflection subword models. The positive trace gives a model analogous
to $m$-noncrossing partitions, while the negative trace gives a model
analogous to $m$-clusters.

\begin{theorem}\label{thm:Fuss_parking_1}
Let $m$ be a positive integer, and let
$\vect=(t_1,\ldots,t_M)\in T^M$. If
$\tt_M\cdots\tt_1=\Delta^{2m}$,
then
\[
\sum_{w\in W}
\operatorname{sub}_T\!\left(
\operatorname{rev}(\vect)\,
\operatorname{rev}(\operatorname{inv}_c(w)),
c^{-1}
\right)
=
\operatorname{Park}_{mh+1}(W).
\]
\end{theorem}

\begin{proof}
By \cref{lem:coeff-identity,thm:GLTW_parking},
\[
\begin{aligned}
(q-1)^r\operatorname{Park}_{mh+1}(W;q)
&=
[\T_e]\left(\PF\,\T_c^{mh+1}\right)\\
&=
[\T_e]\left(
\PF\,\H(\Delta^{2m})\T_c
\right)\\
&=
q^r[\T_{c^{-1}}]\left(
\PF\,\H(\Delta^{2m})
\right)\\
&=
q^r\sum_{w\in W}q^{-\ell(w)}
[\T_{c^{-1}}]\H\left(
\Delta^{2m}
\B(w)\B(w^{-1})
\right),
\end{aligned}
\]
where we used the centrality of $\Delta^2$ in the last equality.

By hypothesis and \cref{lem:inversions}, we have
\[
\begin{aligned}
\Delta^{2m}\mathbf{B}(w)\mathbf{B}(w^{-1})
&=
\tt_M\cdots\tt_1
\overleftarrow{\prod}_{u\in\inv_c(w)}
\uu.
\end{aligned}
\]
The underlying reflection word of this factorization is $\mathrm{rev}(\vect)\,
\mathrm{rev}(\mathrm{inv}_c(w))$, so \cref{cor:main} gives
\[
(q-1)^{-r}
[\T_{c^{-1}}]\H\!\left(
\Delta^{2m}
\mathbf{B}(w)\mathbf{B}(w^{-1})
\right)
\Big|_{q=1}=
\operatorname{sub}_T\!\left(
\operatorname{rev}(\vect)\,
\operatorname{rev}(\operatorname{inv}_c(w)),
c^{-1}
\right),  
\] as desired.   
\end{proof}

\begin{theorem}\label{thm:Fuss_parking_2}
Let $m$ be a positive integer, and let
$\vect=(t_1,\ldots,t_M)\in T^M$. If
$\tt_M\cdots\tt_1=\Delta^{2m-2}$,
then
\[
\sum_{w\in W}
\operatorname{sub}_T\!\left(
\operatorname{inv}_c(w)\,\vect\,\operatorname{inv}(c),
c^{-1}
\right)
=
\operatorname{Park}_{mh+1}(W).
\]    
\end{theorem}

\begin{proof}
Apply \cref{thm:GLTW_parking} with the standard
Coxeter element $c^{-1}$. We have
\[
\begin{aligned}
\mathbf{B}(c^{-1})^{-(mh+1)}
&=
\Delta^{-2m}\mathbf{B}(c^{-1})^{-1}\\
&=
\Delta^{-2m}
\left(
\mathbf{B}(c)\mathbf{B}(c^{-1})
\right)^{-1}
\mathbf{B}(c).
\end{aligned}
\]
Hence, \cref{lem:coeff-identity,thm:GLTW_parking} give
\[
\begin{aligned}
(-q)^{-r(mh+1)}(q-1)^r\!
\operatorname{Park}_{mh+1}(W;q)&=
[\T_e]\left(
\PF\,T_{c^{-1}}^{-(mh+1)}
\right)\\
&=
q^r\sum_{v\in W}q^{-\ell(v)}
[\T_{c^{-1}}]\H\!\left(
\mathbf{B}(v)\mathbf{B}(v^{-1})
\Delta^{-2m}
\left(
\mathbf{B}(c)\mathbf{B}(c^{-1})
\right)^{-1}
\right).
\end{aligned}
\]
For every $v\in W$, we have
$\B(w_\circ v^{-1})\B(v)=\B(v^{-1})\B(vw_\circ)=\Delta$. Consequently,
\[
\begin{aligned}
\mathbf{B}(v)\mathbf{B}(v^{-1})\Delta^{-2}
&=
\mathbf{B}(w_\circ v^{-1})^{-1}
\mathbf{B}(vw_\circ)^{-1}\\
&=
\left(
\mathbf{B}(vw_\circ)
\mathbf{B}(w_\circ v^{-1})
\right)^{-1}.
\end{aligned}
\]
Set $w=vw_\circ$ so that $w^{-1}=w_\circ v^{-1}$ and $\ell(v)=N-\ell(w)$. 
Reindexing the preceding sum therefore gives
\[
\begin{aligned}
&(-q)^{-r(mh+1)}(q-1)^r
\operatorname{Park}_{mh+1}(W;q)\\
&\qquad=
q^r\sum_{w\in W}q^{\ell(w)-N}
[T_{c^{-1}}]\mathrm{H}\!\left(
\left(
\mathbf{B}(w)\mathbf{B}(w^{-1})
\right)^{-1}
\Delta^{-2m+2}
\left(
\mathbf{B}(c)\mathbf{B}(c^{-1})
\right)^{-1}
\right).
\end{aligned}
\]

The hypothesis $\tt_M\cdots\tt_1=\Delta^{2m-2}$
implies that
$\Delta^{-2m+2}
=
\tt_1^{-1}\cdots\tt_M^{-1}$.
Moreover, \cref{lem:inversions} tells us
\[
\left(
\mathbf{B}(w)\mathbf{B}(w^{-1})
\right)^{-1}
=
\prod_{u\in\operatorname{inv}_c(w)}
\uu^{-1}
\quad\text{and}\quad
\left(
\mathbf{B}(c)\mathbf{B}(c^{-1})
\right)^{-1}
=
\prod_{u\in\operatorname{inv}(c)}
\uu^{-1}.
\]
It follows that
\[
\left(
\mathbf{B}(w)\mathbf{B}(w^{-1})
\right)^{-1}
\Delta^{-2m+2}
\left(
\mathbf{B}(c)\mathbf{B}(c^{-1})
\right)^{-1}=
\left(
\prod_{u\in\operatorname{inv}_c(w)}
\uu^{-1}
\right)
\left(
\prod_{i=1}^M\tt_i^{-1}
\right)
\left(
\prod_{u\in\operatorname{inv}(c)}
\uu^{-1}
\right).
\]
Thus, the underlying reflection word is
$\operatorname{inv}_c(w)\,\vect\,\operatorname{inv}(c)$,
and every exponent in the corresponding signed factorization is
$-1$. \cref{thm:main}
implies that
\[
\begin{aligned}
&\left.
(q-1)^{-r}
[T_{c^{-1}}]\H\!\left(
\left(
\mathbf{B}(w)\mathbf{B}(w^{-1})
\right)^{-1}
\Delta^{-2m+2}
\left(
\mathbf{B}(c)\mathbf{B}(c^{-1})
\right)^{-1}
\right)
\right|_{q=1}\\
&\qquad=
(-1)^r
\left|
\operatorname{Sub}_T\left(
\operatorname{inv}_c(w)\,
\vect\,
\operatorname{inv}(c),
c^{-1}
\right)
\right|
\end{aligned}
\]
because every reduced $T$-word for $c^{-1}$ contains exactly $r$
reflections.

Finally, $r(mh+1)\equiv r\pmod 2$
because $rh=2N$. Dividing the trace identity by $(q-1)^r$ and setting
$q=1$, we find that the sign on the left is $(-1)^r$, which cancels the
sign supplied by \cref{thm:main}. We deduce the desired identity. 
\end{proof}

Taking $\vect=\operatorname{inv}({\sf w}_\circ(c))^m$
in \cref{thm:Fuss_parking_1} and
$\vect=\operatorname{inv}({\sf w}_\circ(c))^{m-1}$
in \cref{thm:Fuss_parking_2} yields the canonical models
\begin{equation}\label{eq:Parkmh+1_1}
\sum_{w\in W}
\operatorname{sub}_T\!\left(
\operatorname{inv}({\sf w}_\circ(c^{-1}))^m\operatorname{inv}_{c^{-1}}(w)
,
c^{-1}
\right)
=
\operatorname{Park}_{mh+1}(W)
\end{equation} 
and
\begin{equation}\label{eq:Parkmh+1_2} 
\sum_{w\in W}
\operatorname{sub}_T\!\left(
\operatorname{inv}_c(w)
\operatorname{inv}({\sf w}_\circ(c))^{m-1}
\operatorname{inv}(c),
c^{-1}
\right)
=
\operatorname{Park}_{mh+1}(W).
\end{equation}

\begin{remark}
The formulas in \eqref{eq:Parkmh+1_1} and \eqref{eq:Parkmh+1_2} decompose
\[
\Park_{mh+1}(W)=(mh+1)^r
\]
into sums indexed by $W$. Since $(mh+1)^r$ is also the number of
regions of the $m$-Shi arrangement, it is tempting to regard the
summand indexed by $w$ as a noncrossing analogue of the collection
of $m$-Shi regions lying inside the finite Weyl chamber indexed by
$w$. One notable difference is that the summands in \eqref{eq:Parkmh+1_1} and \eqref{eq:Parkmh+1_2} can
depend on the standard Coxeter element $c$, whereas the $m$-Shi
arrangement does not involve such a choice. 
\end{remark}

\subsection{Fuss--Dogolon Parking Models} 
We now specialize \cref{thm:GLTW_parking} to the case when $p=mh-1$. Once again, positive and negative trace evaluations produce two different
reflection subword models. The positive trace gives a model analogous
to positive $m$-noncrossing partitions, while the negative trace gives a model
analogous to positive $m$-clusters. 

\begin{theorem}\label{thm:dogolon_parking1}
Let $m$ be a positive integer, and let $\vect=(t_1,\ldots,t_M)\in T^M$.
If \[\tt_M\cdots\tt_1=
\Delta^{2m}
\bigl(\B(c)\B(c^{-1})\bigr)^{-1},\]
then
\[
\sum_{w\in W}
\mathrm{sub}_T\!\left(
\mathrm{rev}(\vect)\,
\mathrm{rev}\bigl(\mathrm{inv}_c(w)\bigr),
c^{-1}
\right)
=
\mathrm{Park}_{mh-1}(W).
\]
\end{theorem}

\begin{proof}
Applying \cref{thm:GLTW_parking} with the
standard Coxeter element $c^{-1}$ gives
\[
(q-1)^r\,\mathrm{Park}_{mh-1}(W;q)
=
[\T_e]\left(
\PF\,\T_{c^{-1}}^{mh-1}
\right).
\]
Since $\B(c^{-1})^h=\Delta^2$, we have
\[
\B(c^{-1})^{mh-1}
=
\Delta^{2m}
\bigl(\B(c)\B(c^{-1})\bigr)^{-1}
\B(c).
\]
\cref{lem:coeff-identity} therefore yields
\[
\begin{aligned}
(q-1)^r\mathrm{Park}_{mh-1}(W;q)
&=
q^r[\T_{c^{-1}}]\left(
\PF\cdot
\H\!\left(
\Delta^{2m}
\bigl(\B(c)\B(c^{-1})\bigr)^{-1}
\right)
\right).
\end{aligned}
\]
Since $\PF$ is central by \cref{lem:PF_central}, this is equal to 
\[
\begin{aligned}
q^r
\sum_{w\in W}q^{-\ell(w)}
[\T_{c^{-1}}]\,
\H\!\left(
\Delta^{2m}
\bigl(\B(c)\B(c^{-1})\bigr)^{-1}
\B(w)\B(w^{-1})
\right).
\end{aligned}
\]

By hypothesis,
$\Delta^{2m}
\bigl(\B(c)\B(c^{-1})\bigr)^{-1}
=
\tt_M\cdots\tt_1$.
Moreover, \cref{lem:inversions} tells us that 
\[
\B(w)\B(w^{-1})
=
\overleftarrow{\prod}_{u\in\mathrm{inv}_c(w)}
\uu.
\]
Thus, the underlying reflection word of the pure braid
$\Delta^{2m}
\bigl(\B(c)\B(c^{-1})\bigr)^{-1}
\B(w)\B(w^{-1})$
is
\[
\mathrm{rev}(\vect)\,
\mathrm{rev}\bigl(\mathrm{inv}_c(w)\bigr).
\]
\cref{cor:main} now implies that 
\[
\begin{aligned}
&\left.
(q-1)^{-r}
[\T_{c^{-1}}]\,
\H\!\left(
\Delta^{2m}
\bigl(\B(c)\B(c^{-1})\bigr)^{-1}
\B(w)\B(w^{-1})
\right)
\right|_{q=1}\\
&\qquad=
\mathrm{sub}_T\!\left(
\mathrm{rev}(\vect)\,
\mathrm{rev}\bigl(\mathrm{inv}_c(w)\bigr),
c^{-1}
\right),
\end{aligned}
\]
as desired. 
\end{proof}

\begin{theorem}\label{thm:dogolon_parking2}
Let $m$ be a positive integer, and let $\vect=(t_1,\ldots,t_M)\in T^M$. If
$\tt_M\cdots\tt_1
=
\Delta^{2m-2}$,
then
\[
\sum_{w\in W}
\left|
\mathrm{Sub}_T\left(
\mathrm{inv}_c(w)\,\vect,
c^{-1}
\right)
\right|
=
\mathrm{Park}_{mh-1}(W).
\]    
\end{theorem}
\begin{proof}
\cref{thm:GLTW_parking} tells us that 
\[
(-q)^{-r(mh-1)}(q-1)^r
\mathrm{Park}_{mh-1}(W;q)=
[\T_e]\left(
\PF\,\T_c^{-(mh-1)}
\right).
\] 
Since $\B(c)^{-(mh-1)}
=
\Delta^{-2m}\B(c)$, it follows from \cref{lem:coeff-identity} that
\[
(-q)^{-r(mh-1)}(q-1)^r
\Park_{mh-1}(W;q)=
q^r
\sum_{v\in W}q^{-\ell(v)}
[\T_{c^{-1}}]\,
\H\!\left(
\B(v)\B(v^{-1})\Delta^{-2m}
\right).
\]
We wish to reindex the sum by setting $w=vw_\circ$. We have
$\Delta=\B(w^{-1})\B(v)=\B(v^{-1})\B(w)$, so we can use the centrality of $\Delta^2$ to find that 
\[
\begin{aligned}
\B(v)\B(v^{-1})\Delta^{-2m}
&=
\B(w^{-1})^{-1}
\Delta^2
\B(w)^{-1}
\Delta^{-2m}\\
&=
\B(w^{-1})^{-1}\B(w)^{-1}\Delta^{-2m+2}\\
&=
\bigl(\B(w)\B(w^{-1})\bigr)^{-1}\Delta^{-2m+2}.
\end{aligned}
\]
Since $q^{-\ell(v)}=q^{\ell(w)-N}$,
we can reindex the sum to find that 
\[
(-q)^{-r(mh-1)}(q-1)^r
\mathrm{Park}_{mh-1}(W;q)=
q^r
\sum_{w\in W}q^{\ell(w)-N}
[\T_{c^{-1}}]\,
\H\!\left(
\bigl(\B(w)\B(w^{-1})\bigr)^{-1}
\Delta^{-2m+2}
\right). 
\]

\cref{lem:inversions} tells us that
$\bigl(\B(w)\B(w^{-1})\bigr)^{-1}
=
\prod_{u\in\mathrm{inv}_c(w)}
\uu^{-1}$, and we know by hypothesis that $\Delta^{-2m+2}
=
\tt_1^{-1}\cdots
\tt_M^{-1}$.
It follows that
\[
\bigl(\B(w)\B(w^{-1})\bigr)^{-1}
\Delta^{-2m+2}=
\left(
\prod_{u\in \inv_c(w)}
\uu^{-1}
\right)
\tt_1^{-1}\cdots
\tt_M^{-1}.
\]
The underlying reflection word for this factorization is $\inv_c(w)\,\vect$,
and every exponent in the corresponding signed factorization is
$-1$. Therefore, since every reduced $T$-word for $c^{-1}$ uses exactly $r$ reflections, \cref{thm:main} implies that 
\[
(q-1)^{-r}
[\T_{c^{-1}}]\,
\H\!\left(
\bigl(\B(w)\B(w^{-1})\bigr)^{-1}
\Delta^{-2m+2}
\right)
\Big|_{q=1}=
(-1)^r
\Subb_T\!\left(
\inv_c(w)\,\vect,
c^{-1}
\right).
\]
Finally, $r(mh-1)\equiv r\pmod 2$ since $rh=2N$. 
After dividing the preceding trace identity by $(q-1)^r$ and setting
$q=1$, we obtain a factor of $(-1)^r$ on the left-hand side. This
cancels the factor $(-1)^r$ on the right-hand side, which yields the desired result. 
\end{proof} 

It follows from \eqref{eq:prop_6.6} that
\[\Delta^{2m}
\bigl(\B(c)\B(c^{-1})\bigr)^{-1}
=
\Delta^{2m-2}
\B(cw_\circ)\B(w_\circ c^{-1}).
\]
Consequently, \cref{lem:inversions,cor:wo} show that one may take
$\vect
=
\mathrm{inv}({\sf w}_\circ(c))^{m-1}
\mathrm{inv}_c(cw_\circ)$
in \cref{thm:dogolon_parking1} to obtain the model
\[
\sum_{w\in W}
\Subb_T\!\left(
\mathrm{inv}_{c^{-1}}(cw_\circ)\,
\mathrm{inv}({\sf w}_\circ(c^{-1}))^{m-1}\mathrm{inv}_{c^{-1}}(w)
,
c^{-1}
\right)
=
\mathrm{Park}_{mh-1}(W).
\]
Taking $\vect
=
\mathrm{inv}({\sf w}_\circ(c))^{m-1}$
in \cref{thm:dogolon_parking2} yields the model 
\[
\sum_{w\in W}
\Subb_T\!\left(
\mathrm{inv}_c(w)\,
\mathrm{inv}({\sf w}_\circ(c))^{m-1},
c^{-1}
\right)
=
\mathrm{Park}_{mh-1}(W).
\] 

\subsection{Rational Parking}
We have just seen how to obtain subword models for Fuss--Catalan parking numbers and Fuss--Dogolon parking numbers. Similar techniques also work for more general rational models. The following fact allows us to reuse our positive dual conjugators from~\Cref{sec:rational} without additional work.

\begin{proposition}\label{prop:8.9}  
For all $\bbeta,\mathbf{v} \in \B_W$, we have 
\[[\T_e]\H\!\left(\sum_{w \in W} q^{-\ell(w)} \B(w)\B(w^{-1}) \mathbf{v} \bbeta \mathbf{v}^{-1}\right)=[\T_e]\H\!\left(\sum_{w \in W} q^{-\ell(w)}\B(w)\B(w^{-1}) \bbeta\right).\]
\end{proposition}

\begin{proof}
We can use the fact that $\PF$ is central (\cref{lem:PF_central}) and the trace property to compute that 
\begin{align*}[\T_e]\left(\PF\cdot\H(\mathbf{v}\bbeta\mathbf{v}^{-1})\right)&=[\T_e]\left(\H(\mathbf{v})\cdot\PF \cdot\H(\bbeta \mathbf{v}^{-1})\right) \\ 
&= [\T_e]\left(\PF \cdot\H(\bbeta \mathbf{v}^{-1})\H(\mathbf{v})\right) \\
&= [\T_e]\left(\PF \cdot\H(\bbeta)\right). \end{align*} This is equivalent to the desired identity. 
\end{proof} 

In the next theorems, for $\protect\vv\epsilon=(\epsilon_1,\ldots,\epsilon_M)\in\{\pm 1\}^M$ and $k\geq 0$, let 
\[
(1^k,\protect\vv\epsilon)
=
(\underbrace{1,\ldots,1}_{k\text{ times}},
\epsilon_1,\ldots,\epsilon_M).
\] 

\begin{theorem}\label{thm:rat_parking}
Let $W$ be a Weyl group of rank $r$, and fix a standard Coxeter element $c$. Let
$p>0$ be coprime to $h$. Choose an element $g_+\in W$ such that
$g_+c^pg_+^{-1}=c$,
and choose a braid lift $\mathbf g_+\in\B_W$ of $g_+$. Define the
pure braid
\[
\mathbbm c_{p,\mathbf g_+}
=
\mathbf g_+\B(c)^p\mathbf g_+^{-1}\B(c)^{-1}.
\]
Suppose that $\vect=(t_1,\ldots,t_M)\in T^M$ and
$\protect\vv\epsilon=(\epsilon_1,\ldots,\epsilon_M)\in\{\pm1\}^M$
satisfy
\[
\mathbbm c_{p,\mathbf g_+}
=
\P_c(t_1)^{\epsilon_1}\cdots\P_c(t_M)^{\epsilon_M}.
\]
Then
\[
\sum_{w\in W}
\Subb_T^{(1^{\ell(w)},\protect\vv\epsilon)}
\bigl(\rev(\inv_c(w))\vect,c^{-1}\bigr)
=
\Park_p(W).
\]
\end{theorem}

\begin{proof}
By \cref{thm:GLTW_parking,prop:8.9}, we have
\begin{align*}
(q-1)^r\Park_p(W;q)
&=
[\T_e]\bigl(
\mathrm{PF}\,
\H(\mathbf g_+\B(c)^p\mathbf g_+^{-1})
\bigr)\\
&=
[\T_e]\bigl(
\mathrm{PF}\,
\H(\mathbbm c_{p,\mathbf g_+}\B(c))
\bigr)\\
&=
q^r[\T_{c^{-1}}]\bigl(
\mathrm{PF}\,\H(\mathbbm c_{p,\mathbf g_+})
\bigr)\\
&=
q^r\sum_{w\in W}q^{-\ell(w)}
[\T_{c^{-1}}]
\H\bigl(
\B(w)\B(w^{-1})\mathbbm c_{p,\mathbf g_+}
\bigr),
\end{align*}
where the third equality follows from \cref{lem:coeff-identity}. For each $w\in W$,
\cref{lem:inversions} and the assumed factorization of
$\mathbbm c_{p,\mathbf g_+}$ give a factorization of
$\B(w)\B(w^{-1})\mathbbm c_{p,\mathbf g_+}$ using elements of $\mathbb T$ and their inverses; the underlying reflection word is
$\mathrm{rev}(\inv_c(w))\vect$,
and the sign vector is $(1^{\ell(w)},\protect\vv\epsilon)$. Since
$\ell_T(c^{-1})=r$, \cref{thm:main} tells us that
\[
(q-1)^{-r}
[\T_{c^{-1}}]
\H\bigl(
\B(w)\B(w^{-1})\mathbbm c_{p,\mathbf g_+}
\bigr)
\Big|_{q=1}
=
\operatorname{sub}^{(1^{\ell(w)},\protect\vv\epsilon)}_T
\bigl(
\operatorname{rev}(\operatorname{inv}_c(w))\vec t,
c^{-1}
\bigr).
\]
Dividing the preceding trace identity by $(q-1)^r$ and specializing
$q=1$ proves the result.
\end{proof}

\begin{theorem}\label{thm:rat_parking_negative}
Let $W$ be a Weyl group of rank $r$, and fix a standard Coxeter element $c$. Let
$p>0$ be coprime to $h$. Choose an element $g_-\in W$ such that
\(
g_-c^{-p}g_-^{-1}=c,
\)
and choose a braid lift $\mathbf g_-\in\B_W$ of $g_-$. Define the
pure braid
\[
\mathbbm c_{-p,\mathbf g_-}
=
\mathbf g_-\B(c)^{-p}\mathbf g_-^{-1}\B(c)^{-1}.
\]
Suppose that $\vect=(t_1,\ldots,t_M)\in T^M$ and
$\protect\vv\epsilon=(\epsilon_1,\ldots,\epsilon_M)\in\{\pm1\}^M$
satisfy
\[
\mathbbm c_{-p,\mathbf g_-}
=
\P_c(t_1)^{\epsilon_1}\cdots\P_c(t_M)^{\epsilon_M}.
\]
Then
\[
(-1)^r
\sum_{w\in W}
\Subb_T^{(1^{\ell(w)},\protect\vv\epsilon)}
\bigl(\rev(\inv_c(w))\vect,c^{-1}\bigr)
=
\Park_p(W). 
\]
\end{theorem} 

\begin{proof}
Using the second identity in \cref{thm:GLTW_parking} together with \cref{prop:8.9}, we obtain
\begin{align*}
(-q)^{-rp}(q-1)^r\Park_p(W;q)
&=
[\T_e]\bigl(
\mathrm{PF}\,
\H(\mathbf g_-\B(c)^{-p}\mathbf g_-^{-1})
\bigr)\\
&=
[\T_e]\bigl(
\mathrm{PF}\,
\H(\mathbbm c_{-p,\mathbf g_-}\B(c))
\bigr)\\
&=
q^r\sum_{w\in W}q^{-\ell(w)}
[\T_{c^{-1}}]
\H\bigl(
\B(w)\B(w^{-1})\mathbbm c_{-p,\mathbf g_-}
\bigr).
\end{align*}
As in the proof of \cref{thm:rat_parking}, we can invoke \cref{lem:inversions,thm:main} to find that 
\[
(q-1)^{-r}
[\T_{c^{-1}}]
\H\bigl(
\B(w)\B(w^{-1})\mathbbm c_{-p,\mathbf g_-}
\bigr)
\Big|_{q=1}
=
\operatorname{sub}^{(1^{\ell(w)},\protect\vv\epsilon)}_T
\bigl(
\operatorname{rev}(\operatorname{inv}_c(w))\vect,
c^{-1}
\bigr).
\]
Dividing the trace identity by $(q-1)^r$ and setting $q=1$ then implies that
\[
\sum_{w\in W}
\operatorname{sub}^{(1^{\ell(w)},\protect\vv\epsilon)}_T
\bigl(
\operatorname{rev}(\operatorname{inv}_c(w))\vect,
c^{-1}
\bigr)
=
(-1)^{rp}\Park_p(W).
\]
Finally, $rp\equiv r\pmod 2$. Indeed, this is immediate if $r$ is
even; if $r$ is odd, then the identity $rh=2N$ forces $h$ to be even, so
the fact that $\gcd(p,h)=1$ forces $p$ to be odd. Hence $(-1)^{rp}=(-1)^r$, which
proves the result.
\end{proof}

\section{Applications to Knot Theory}\label{sec:knot}

Let $K$ be a knot. In this section, we show how to compute the
specialization $\HOMFLY(K;a,z)|_{z=0}$ using signed reflection
subwords. At the end of the section, we explain the corresponding
statement for a link $L$ with $d$ components; in that setting, the natural
specialization is
\[
\left.z^{d-1}\HOMFLY(L;a,z)\right|_{z=0}.
\]

We specialize to linear type~$A$. Let
$c=(12\cdots n)\in\mathfrak S_n$ be the usual long cycle.
Let $\s_1,\ldots,\s_{n-1}$ be the Artin generators of $\B_n$, and
write
\[
\t_{(ij)}
=
\s_i\cdots\s_{j-2}\s_{j-1}
\s_{j-2}^{-1}\cdots\s_i^{-1}
\]
for the $c$-dual lift of the transposition $(ij)$. Note that
$\t_{(ij)}$ is the classical Birman--Ko--Lee dual generator~\cite{birman1998new}.
Let
\[
\tt_{(ij)}
=
\t_{(ij)}^2
=
\s_i\s_{i+1}\cdots\s_{j-2}\s_{j-1}^2
\s_{j-2}^{-1}\cdots\s_i^{-1}.
\]
Geometrically, $\t_{(ij)}$ is the positive half twist interchanging
the $i$th and $j$th points, and it projects to the transposition
$(ij)$. Its square $\tt_{(ij)}$ is the corresponding pure full twist
around the hyperplane $x_i=x_j$. 

Despite an easy, explicit formula for the Hecke expansion of $\H(\tt_{(ij)})$, the multiplication and
re-expansion of the resulting terms in the Hecke algebra leaves $\NC(\mathfrak{S}_n,c^{-1})$ and so does not allow us to give the full $q$-expansion purely in terms of noncrossing combinatorics. We therefore content ourselves with the leading $q=1$ evaluation.

Code to compute the results of this section is available at~\cite{code}.

\subsection{Knots to Braids}
Let $K$ be a knot, and let $\bbeta_K = \prod_{k=1}^\ell \s_{i_k}^{\epsilon_k} \in \B_n$ be a braid whose
braid closure is equal to $K$;
write $\mathrm{wr}(\bbeta_K) = \sum_{k=1}^\ell \epsilon_k$ for the writhe of $\bbeta_K$. 

\begin{definition}
The \defn{HOMFLYPT polynomial} $\HOMFLY(L;a,z)\in\mathbb{Z}[a^{\pm1},z^{\pm1}]$ of an oriented link $L$
is determined by the skein recurrence
\[
  a\,\HOMFLY(L_+) \;-\; a^{-1}\,\HOMFLY(L_-) \;=\; z\,\HOMFLY(L_0),
\]
together with the normalization $\HOMFLY(\bigcirc)=1$ on the unknot, where $L_+,L_-,L_0$ differ
only inside a small disk as a positive crossing, a negative crossing, and the oriented smoothing,
respectively.  For a Laurent polynomial $F$ in $a$, we write $[a^m]F$ for the
coefficient of $a^m$ in $F$. Throughout this section, we use the
substitution
\(
z=q^{1/2}-q^{-1/2}.
\) 
\end{definition}

Write $\mathfrak{S}_{n,k} := \big\{u \in \mathfrak{S}_{n} : \Des(u)=\{s_1,\ldots,s_k\}\big\}$ (so $\mathfrak{S}_{n,0}=\{e\}$ and $\mathfrak{S}_{n,n-1}=\{w_\circ\}$), and for $u\in \mathfrak{S}_n$ let
$\ell(u)$ denote its Coxeter (inversion) length. The leading Hecke coefficients of a braid recover
the $a$-coefficients of $\HOMFLY$. 

\begin{theorem}[{\cite[Theorem~1.8]{trinh2026partial}}]
\label{thm:homfly-trace} 
Let $L$ be an oriented link with $d$ components, and let
$\boldsymbol{\beta}_L\in \B_n$ be a braid whose closure is $L$. For every
$0\leq k\leq n-1$, we have
\begin{align*}
&\,\,\,\,\,\,\,
\left.
[a^{-\operatorname{wr}(\boldsymbol{\beta}_L)-n+1+2k}]
\Bigl(z^{d-1}\homfly(L;a,z)\Bigr)
\right|_{z=0} \\ 
&=
(-1)^{n-1+k}
\left.
(q-1)^{-(n-d)}
[\T_e]\H\!\left(
\sum_{u\in\mathfrak S_{n,k}}
q^{-\ell(u)}
\B(u)\B(u^{-1})\boldsymbol{\beta}_L
\right)
\right|_{q=1}.
\end{align*}   
\end{theorem}

Recall that for $\protect\vv\epsilon=(\epsilon_1,\ldots,\epsilon_M)\in\{\pm 1\}^M$ and $k\geq 0$, we write  
\[
(1^k,\protect\vv\epsilon)
=
(\underbrace{1,\ldots,1}_{k\text{ times}},
\epsilon_1,\ldots,\epsilon_M).
\] 

\begin{theorem}\label{thm:knots}
Let $\bbeta_K\in\B_n$ be a braid whose closure is a knot $K$. Let $\beta_K$ be the image of $\bbeta_K$ in $\mathfrak S_n$. Choose $g\in\mathfrak S_n$
such that $g\beta_Kg^{-1}=c$,
and choose a braid lift $\mathbf g\in\B_n$ of $g$. Factor the pure braid $\brho_K
:=
\mathbf g\bbeta_K\mathbf g^{-1}\B(c)^{-1}$ as
\[
\brho_K
=
\P_c(t_1)^{\epsilon_1}\cdots\P_c(t_M)^{\epsilon_M}
\]
for some $\vect=(t_1,\ldots,t_M)\in T^M$ and
$\protect\vv\epsilon=(\epsilon_1,\ldots,\epsilon_M)
\in\{\pm1\}^M$.
We have
\[
(-1)^{n-1}
\Subb_T^{\protect\vv\epsilon}(\vect,c^{-1})
=
\left.[a^{-\mathrm{wr}(\bbeta_K)-n+1}]
\homfly(K;a,z)\right|_{z=0}.
\]
More generally,
\[
\homfly(K;a,z)\Big|_{z=0}
=
\sum_{k=0}^{n-1}
(-1)^{n-1+k}
\sum_{u\in \mathfrak S_{n,k}}
\Subb_T^{(1^{\ell(u)},\protect\vv\epsilon)}
\bigl(\rev(\inv_c(u))\vect,c^{-1}\bigr)
a^{-\mathrm{wr}(\bbeta_K)-n+1+2k}.
\]
\end{theorem}

\begin{proof}
For $0\leq k\leq n-1$, set
\[
Z_{n,k}
=
\sum_{u\in \mathfrak S_{n,k}}
q^{-\ell(u)}
\H\bigl(\B(u)\B(u^{-1})\bigr).
\]
The element $Z_{n,k}$ is central in $\H_n$ by the description of
these elements in~\cite[Section~7]{trinh2026partial}.
Therefore, using the trace property, we find that 
\begin{align*}
[\T_e]\bigl(Z_{n,k}\H(\bbeta_K)\bigr)
&=
[\T_e]\bigl(
Z_{n,k}\H(\mathbf g\bbeta_K\mathbf g^{-1})
\bigr)\\
&=
[\T_e]\bigl(Z_{n,k}\H(\brho_K\B(c))\bigr)\\
&=
q^{n-1}
[\T_{c^{-1}}]\bigl(Z_{n,k}\H(\brho_K)\bigr),
\end{align*}
where the last equality follows from \cref{lem:coeff-identity}.

Fix $u\in \mathfrak S_{n,k}$. By \cref{lem:inversions},
if
\(
\rev(\inv_c(u))=(v_1,\ldots,v_{\ell(u)}),
\)
then
\[
\B(u)\B(u^{-1})
=
\P_c(v_1)\cdots\P_c(v_{\ell(u)}).
\]
Consequently,
\[
\B(u)\B(u^{-1})\brho_K
=
\P_c(v_1)\cdots\P_c(v_{\ell(u)})
\P_c(t_1)^{\epsilon_1}\cdots\P_c(t_M)^{\epsilon_M}.
\]
Invoking \cref{thm:main} allows us to deduce that 
\[
(q-1)^{-(n-1)}
[\T_{c^{-1}}]
\H\bigl(\B(u)\B(u^{-1})\brho_K\bigr)
\Big|_{q=1}
=
\Subb_T^{(1^{\ell(u)},\epsilon)}
\bigl(\rev(\inv_c(u))\vect,c^{-1}\bigr).
\]
Since all powers of $q$ in the preceding expressions specialize to
$1$, it follows that
\[
(q-1)^{-(n-1)}
[\T_e]\bigl(Z_{n,k}\H(\bbeta_K)\bigr)
\Big|_{q=1}
=
\sum_{u\in \mathfrak S_{n,k}}
\Subb_T^{(1^{\ell(u)},\epsilon)}
\bigl(\rev(\inv_c(u))\vect,c^{-1}\bigr).
\]
Combining this with \cref{thm:homfly-trace} proves the coefficient
formula for every $k$. Summing these coefficient formulas over
$0\leq k\leq n-1$ proves the second identity. The first identity is
the special case $k=0$ since $\mathfrak S_{n,0}=\{e\}$.
\end{proof} 

\Cref{tab:knots} illustrates~\Cref{thm:knots} for all knots up to eight crossings.

\subsection{Links}

Let $L$ be an oriented link with $d$ components, and let
$\bbeta_L\in\B_n$ be a braid whose closure is $L$. The image
${\beta_L}\in\mathfrak S_n$ has $d$ cycles; let
$\lambda=(\lambda_1\geq\cdots\geq\lambda_d)$ be its cycle type. Set
$m_0=0$ and $m_j=\lambda_1+\cdots+\lambda_j$ for $1\leq j\leq d$,
and define
\[
c_\lambda
=
\prod_{j=1}^d
(m_{j-1}+1\ \ m_{j-1}+2\ \ \cdots\ \ m_j).
\]
Thus, we have $c_\lambda\in\NC(\mathfrak S_n,c)$ and
$c_\lambda^{-1}\in\NC(\mathfrak S_n,c^{-1})$.

Choose $g\in\mathfrak S_n$ such that
$g{\beta_L}g^{-1}=c_\lambda$,
and choose a braid lift $\mathbf g\in\B_n$ of $g$. Suppose
\[
\brho_L
:=
\mathbf g\beta_L\mathbf g^{-1}\B(c_\lambda)^{-1}
=
\P_c(t_1)^{\epsilon_1}\cdots\P_c(t_M)^{\epsilon_M},
\]
where
\[
\vect=(t_1,\ldots,t_M)\in T^M\quad\text{and}\quad
\protect\vv\epsilon=(\epsilon_1,\ldots,\epsilon_M)\in\{\pm 1\}^M.
\]

\begin{theorem}\label{thm:homfly-links}
With the notation above, we have
\[
\left.z^{d-1}\homfly(L;a,z)\right|_{z=0}
=
\sum_{k=0}^{n-1}
(-1)^{n-1+k}
\sum_{u\in \mathfrak S_{n,k}}
\Subb_T^{(1^{\ell(u)},\protect\vv\epsilon)}
\bigl(
\rev(\inv_c(u))\vect,
c_\lambda^{-1}
\bigr)
a^{-\mathrm{wr}(\bbeta_L)-n+1+2k}.
\]
\end{theorem}

\begin{proof}
By \cref{thm:homfly-trace}, for every $0\leq k\leq n-1$, we have
\begin{align*}
&[a^{-\mathrm{wr}(\bbeta_L)-n+1+2k}]
\left.z^{d-1}\homfly(L;a,z)\right|_{z=0}\\
&\qquad=
(-1)^{n-1+k}
(q-1)^{-(n-d)}
[\T_e]\H\!\left.\left(
\sum_{u\in \mathfrak S_{n,k}}
q^{-\ell(u)}
\B(u)\B(u^{-1})\bbeta_L
\right)\right|_{q=1}.
\end{align*}
In particular, the Laurent polynomial
\[
[\T_e]\H\!\left(
\sum_{u\in\mathfrak S_{n,k}}
q^{-\ell(u)}
\B(u)\B(u^{-1})\bbeta_L
\right)
\]
is divisible by $(q-1)^{n-d}$. 

As in the proof of \cref{thm:knots}, centrality allows us
to replace $\bbeta_L$ by
$\mathbf g\bbeta_L\mathbf g^{-1}
=
\brho_L\B(c_\lambda)$. 
Since
$\ell(c_\lambda)=\ell_T(c_\lambda)=n-d$,
\cref{lem:coeff-identity} converts the trace into the coefficient of
$\T_{c_\lambda^{-1}}$. Finally,
$c_\lambda^{-1}\in\NC(\mathfrak S_n,c^{-1})$, so
\cref{thm:main} applies to the signed factorization of $\B(u)\B(u^{-1})\brho_L$.
This gives
\[
(q-1)^{-(n-d)}
[\T_{c_\lambda^{-1}}]
\H\!\bigl(\B(u)\B(u^{-1})\brho_L\bigr)\Big|_{q=1}
=
\Subb_T^{(1^{\ell(u)},\protect\vv\epsilon)}
\bigl(
\rev(\inv_c(u))\vect,
c_\lambda^{-1}
\bigr).
\]
To complete the proof, we substitute this into the preceding trace formula and sum over
$k$. 
\end{proof}

\begin{landscape}
\renewcommand{\arraystretch}{1.4}
\begin{longtable}{@{} c c c c c c c c @{}}
\caption{For each knot $K$ up to eight crossings from the Rolfsen Knot Table, we use Sage's representative braid $\bbeta_K$ on $n$ strands with braid closure $K$~\cite{sage}, and we compute a conjugating braid $\mathbf{g}$ so that the projection of $\mathbf{g}\bbeta_K\mathbf{g}^{-1}$ to $\mathfrak S_n$ is the standard long cycle $c=(12\cdots n)$ in $\mathfrak S_n$.  We rewrite the resulting pure braid $\brho_K=\mathbf{g}\bbeta_K\mathbf{g}^{-1}\bc^{-1}$ as a product $\prod_{k=1}^m \tt_{(i_k\,j_k)}^{\epsilon_k}$.  The signed subword count given in~\Cref{thm:knots} computes the coefficients of the $(a,z)$-HOMFLYPT polynomial of $K$ evaluated at $z=0$. A list $[b_0,\ldots,b_{n-1}]$ in the final column denotes the
polynomial
$\sum_{k=0}^{n-1}
b_k a^{-\mathrm{wr}(\bbeta_K)-n+1+2k}$.} \\
\toprule
Knot $K$ & $\bbeta_K$ & $n$ & $\mathrm{wr}(\bbeta_K)$ & $\mathbf{g}$ & $\brho_K=\prod \tt_{(ij)}^{\pm 1}$ & $\HOMFLY(K;a,0)$ \\
\midrule
\endfirsthead

\multicolumn{8}{c}{{\tablename\ \thetable{} -- continued}} \\
\toprule
Knot $K$ & $\bbeta_K$ & $n$ & $\mathrm{wr}(\bbeta_K)$ & $\mathbf{g}$ & $\brho_K=\prod \tt_{(ij)}^{\pm 1}$ & $\HOMFLY(K;a,0)$ \\
\midrule
\endhead

\midrule
\endfoot

\bottomrule
\endlastfoot

$[3, 1]$ & $\mathbf{s}_1^{-3}$ & 2 & $-3$ & $e$ & 
$\tt_{(12)}^{-1}\tt_{(12)}^{-1}$ & 
$[2, -1]$ \\ \addlinespace

$[4, 1]$ & $(\mathbf{s}_1^{-1}\mathbf{s}_2)^2$ & 3 & 0 & $\mathbf{s}_2$ & 
$\tt_{(12)}^{-1}\tt_{(13)}^{-1}\tt_{(12)}$  & $[1, -1, 1]$ \\ \addlinespace

$[5, 1]$ &$\mathbf{s}_1^{-5}$ & 2 & $-5$ & $e$  & 
$\tt_{(12)}^{-1}\tt_{(12)}^{-1}\tt_{(12)}^{-1}$ & 
 $[3, -2]$ \\ \addlinespace

$[5, 2]$ &$\mathbf{s}_1^{-3}\mathbf{s}_2^{-1}\mathbf{s}_1\mathbf{s}_2^{-1}$ & 3 & $-4$ & $\mathbf{s}_2$  & 
$\tt_{(12)}^{-1}\tt_{(13)}^{-1}\tt_{(13)}^{-1}\tt_{(23)}\tt_{(13)}^{-1}$ & 
 $[1, 1, -1]$ \\ \addlinespace

$[6, 1]$ &$\mathbf{s}_1^{-2}\mathbf{s}_2^{-1}\mathbf{s}_1\mathbf{s}_3\mathbf{s}_2^{-1}\mathbf{s}_3$ & 4 & $-1$ & $\mathbf{s}_2$  & 
$ \tt_{(12)}^{-1}\tt_{(13)}^{-1}\tt_{(12)}\tt_{(13)}^{-1}\tt_{(14)}^{-1}\tt_{(13)}$ & 
 $[1, 0, -1, 1]$ \\ \addlinespace

$[6, 2]$ &$\mathbf{s}_1^{-3}\mathbf{s}_2\mathbf{s}_1^{-1}\mathbf{s}_2$ & 3 & $-2$ & $\mathbf{s}_2$  & 
$\tt_{(12)}^{-1}\tt_{(13)}^{-1}\tt_{(13)}^{-1}\tt_{(12)}$ & 
 $[2, -2, 1]$ \\ \addlinespace

$[6, 3]$ &$\mathbf{s}_1^{-2}\mathbf{s}_2\mathbf{s}_1^{-1}\mathbf{s}_2^2$ & 3 & 0 & $\mathbf{s}_2$  & 
$ \tt_{(12)}^{-1}\tt_{(13)}^{-1}\tt_{(12)}\tt_{(23)}\tt_{(12)}^{-1}$ & 
 $[-1, 3, -1]$ \\ \addlinespace

$[7, 1]$ &$\mathbf{s}_1^{-7}$ & 2 & $-7$ & $e$ & 
$ \tt_{(12)}^{-1}\tt_{(12)}^{-1}\tt_{(12)}^{-1}\tt_{(12)}^{-1}$ & 
 $[4, -3]$ \\ \addlinespace

$[7, 2]$ &$\mathbf{s}_1^{-3}\mathbf{s}_2^{-1}\mathbf{s}_1\mathbf{s}_2^{-1}\mathbf{s}_3^{-1}\mathbf{s}_2\mathbf{s}_3^{-1}$ & 4 & $-5$ & $\mathbf{s}_2$ & 
$\tt_{(12)}^{-1}\tt_{(13)}^{-1}\tt_{(13)}^{-1}\tt_{(23)}\tt_{(34)}^{-1}\tt_{(13)}^{-1}$ & 
 $[1, 0, 1, -1]$ \\ \addlinespace

$[7, 3]$ &$\mathbf{s}_1^5\mathbf{s}_2\mathbf{s}_1^{-1}\mathbf{s}_2$ & 3 & 6 & $\mathbf{s}_2$  & 
$ \tt_{(12)}^{-1}\tt_{(13)}\tt_{(13)}\tt_{(12)}$ & 
 $[-2, 2, 1]$ \\ \addlinespace

$[7, 4]$ &$\mathbf{s}_1^2\mathbf{s}_2\mathbf{s}_1^{-1}\mathbf{s}_2^{2}\mathbf{s}_3\mathbf{s}_2^{-1}\mathbf{s}_3$ & 4 & 5 & $\mathbf{s}_2$ & 
$\tt_{(23)}\tt_{(13)}\tt_{(12)}^{-1}\tt_{(14)}^{-1}\tt_{(13)}$ & 
 $[-1, 0, 2, 0]$ \\ \addlinespace

$[7, 5]$ &$\mathbf{s}_1^{-4}\mathbf{s}_2^{-1}\mathbf{s}_1\mathbf{s}_2^{-2}$ & 3 & $-6$ & $\mathbf{s}_2$  & 
$ \tt_{(12)}^{-1}\tt_{(13)}^{-1}\tt_{(13)}^{-1}\tt_{(12)}\tt_{(13)}^{-1}\tt_{(13)}^{-1}$ & 
 $[2, 0, -1]$ \\ \addlinespace

$[7, 6]$ &$\mathbf{s}_1^{-2}\mathbf{s}_2\mathbf{s}_1^{-1}\mathbf{s}_3^{-1}\mathbf{s}_2\mathbf{s}_3^{-1}$ & 4 & $-3$ & $\mathbf{s}_2$  & 
$ \tt_{(12)}^{-1}\tt_{(13)}^{-1}\tt_{(12)}\tt_{(23)}\tt_{(12)}^{-1}\tt_{(34)}^{-1}\tt_{(13)}^{-1}$ & 
 $[1, -1, 2, -1]$ \\ \addlinespace

$[7, 7]$ &$(\mathbf{s}_1\mathbf{s}_2^{-1})^2\mathbf{s}_3\mathbf{s}_2^{-1}\mathbf{s}_3$& 4 & 1 & $\mathbf{s}_2$  & 
$ \tt_{(12)}^{-1}\tt_{(23)}\tt_{(13)}^{-1}\tt_{(14)}^{-1}\tt_{(13)}$ & 
 $[1, -2, 2, 0]$ \\ \addlinespace

$[8, 1]$ &$\mathbf{s}_1^{-2}\mathbf{s}_2^{-1}\mathbf{s}_1\mathbf{s}_2^{-1}\mathbf{s}_3^{-1}\mathbf{s}_2\mathbf{s}_4\mathbf{s}_3^{-1}\mathbf{s}_4$& 5 & $-2$ & $\mathbf{s}_3\mathbf{s}_2\mathbf{s}_4$ & 
$ \tt_{(12)}^{-1}\tt_{(13)}^{-1}\tt_{(14)}^{-1}\tt_{(13)}\tt_{(12)}\tt_{(13)}^{-1}\tt_{(14)}^{-1}\tt_{(15)}^{-1}\tt_{(13)}$ & 
 $[1, 0, 0, -1, 1]$ \\ \addlinespace

$[8, 2]$ &$\mathbf{s}_1^{-4}(\mathbf{s}_1^{-1}\mathbf{s}_2)^2$& 3 & $-4$ & $\mathbf{s}_2$ & 
$ \tt_{(12)}^{-1}\tt_{(13)}^{-1}\tt_{(13)}^{-1}\tt_{(13)}^{-1}\tt_{(12)}$ & 
 $[3, -3, 1]$ \\ \addlinespace

$[8, 3]$ &$\mathbf{s}_1^{-2}\mathbf{s}_2^{-1}\mathbf{s}_1\mathbf{s}_3\mathbf{s}_2^{-1}\mathbf{s}_3\mathbf{s}_4\mathbf{s}_3^{-1}\mathbf{s}_4$& 5 & 0 & $\mathbf{s}_3\mathbf{s}_2\mathbf{s}_4$  & 
$ \tt_{(12)}^{-1}\tt_{(13)}^{-1}\tt_{(14)}^{-1}\tt_{(13)}\tt_{(12)}\tt_{(13)}^{-1}\tt_{(14)}^{-1}\tt_{(15)}^{-1}\tt_{(14)}\tt_{(13)}$ & 
 $[1, 0, -1, 0, 1]$ \\ \addlinespace

$[8, 4]$ &$\mathbf{s}_1^{-3}\mathbf{s}_2\mathbf{s}_1^{-1}\mathbf{s}_2\mathbf{s}_3\mathbf{s}_2^{-1}\mathbf{s}_3$& 4 & $-1$ & $\mathbf{s}_2$  & 
$ \tt_{(12)}^{-1}\tt_{(13)}^{-1}\tt_{(13)}^{-1}\tt_{(12)}\tt_{(14)}^{-1}\tt_{(13)}$ & 
 $[2, -2, 0, 1]$ \\ \addlinespace

$[8, 5]$ &$(\mathbf{s}_1^{3}\mathbf{s}_2^{-1})^2$& 3 & 4 & $\mathbf{s}_2$  & 
$\tt_{(12)}^{-1}\tt_{(13)}\tt_{(23)}\tt_{(23)}\tt_{(13)}^{-1}$ & 
 $[2, -5, 4]$ \\ \addlinespace

$[8, 6]$ &$\mathbf{s}_1^{-4}\mathbf{s}_2^{-1}\mathbf{s}_1\mathbf{s}_3\mathbf{s}_2^{-1}\mathbf{s}_3$& 4 & $-3$ & $\mathbf{s}_2$ & 
$\tt_{(12)}^{-1}\tt_{(13)}^{-1}\tt_{(13)}^{-1}\tt_{(12)}\tt_{(13)}^{-1}\tt_{(14)}^{-1}\tt_{(13)}$ & 
 $[2, -1, -1, 1]$ \\ \addlinespace

$[8, 7]$ &$\mathbf{s}_1^{3}(\mathbf{s}_1\mathbf{s}_2^{-1})^2\mathbf{s}_2^{-1}$& 3 & 2 & $\mathbf{s}_2$ & 
$\tt_{(12)}^{-1}\tt_{(13)}\tt_{(13)}\tt_{(12)}\tt_{(13)}^{-1}\tt_{(13)}^{-1}$ & 
 $[-2, 4, -1]$ \\ \addlinespace

$[8, 8]$ &$\mathbf{s}_1^{3}\mathbf{s}_2\mathbf{s}_1^{-1}\mathbf{s}_3^{-1}\mathbf{s}_2\mathbf{s}_3^{-2}$& 4 & 1 & $\mathbf{s}_2$ & 
$\tt_{(12)}^{-1}\tt_{(13)}\tt_{(12)}\tt_{(14)}^{-1}\tt_{(13)}^{-1}$ & 
 $[-1, 1, 2, -1]$ \\ \addlinespace

$[8, 9]$ &$\mathbf{s}_1^{-2}(\mathbf{s}_1^{-1}\mathbf{s}_2)^2\mathbf{s}_2^{2}$& 3 & 0 & $\mathbf{s}_2$ & 
$\tt_{(12)}^{-1}\tt_{(13)}^{-1}\tt_{(13)}^{-1}\tt_{(12)}\tt_{(13)}$ & 
$[2, -3, 2]$ \\ \addlinespace

$[8, 10]$ &$\mathbf{s}_1^{3}\mathbf{s}_2^{-1}\mathbf{s}_1^{2}\mathbf{s}_2^{-2}$& 3 & 2 & $e$ & 
$\tt_{(12)}\tt_{(13)}^{-1}\tt_{(23)}\tt_{(13)}^{-1}$ & 
 $[-3, 6, -2]$ \\ \addlinespace

$[8, 11]$ &$\mathbf{s}_1^{-2}\mathbf{s}_2^{-1}\mathbf{s}_1\mathbf{s}_2^{-2}\mathbf{s}_3\mathbf{s}_2^{-1}\mathbf{s}_3$& 4 & $-3$ & $\mathbf{s}_2$ & 
$ \tt_{(12)}^{-1}\tt_{(13)}^{-1}\tt_{(12)}\tt_{(13)}^{-1}\tt_{(13)}^{-1}\tt_{(14)}^{-1}\tt_{(13)}$ & 
 $[1, 1, -2, 1]$ \\ \addlinespace

$[8, 12]$ &$\mathbf{s}_1^{-1}\mathbf{s}_2\mathbf{s}_1^{-1}\mathbf{s}_3^{-1}\mathbf{s}_2\mathbf{s}_4\mathbf{s}_3^{-1}\mathbf{s}_4$& 5 & 0 & $\mathbf{s}_3\mathbf{s}_2\mathbf{s}_4$ & 
$ \tt_{(12)}^{-1}\tt_{(13)}^{-1}\tt_{(14)}^{-1}\tt_{(13)}\tt_{(12)}\tt_{(13)}^{-1}\tt_{(15)}^{-1}\tt_{(13)}$ & 
 $[1, -1, 1, -1, 1]$ \\ \addlinespace

$[8, 13]$ &$\mathbf{s}_1^{-1}(\mathbf{s}_1^{-1}\mathbf{s}_2)^2\mathbf{s}_2\mathbf{s}_3\mathbf{s}_2^{-1}\mathbf{s}_3$& 4 & 1 & $\mathbf{s}_2$ & 
$ \tt_{(12)}^{-1}\tt_{(13)}^{-1}\tt_{(12)}\tt_{(23)}\tt_{(12)}^{-1}\tt_{(14)}^{-1}\tt_{(13)}$ & 
$[-1, 2, 0, 0]$ \\ \addlinespace

$[8, 14]$ &$\mathbf{s}_1^{-3}\mathbf{s}_2^{-1}\mathbf{s}_1(\mathbf{s}_2^{-1}\mathbf{s}_3)^2$& 4 & $-3$ & $\mathbf{s}_2$ & 
$\tt_{(12)}^{-1}\tt_{(13)}^{-1}\tt_{(13)}^{-1}\tt_{(23)}\tt_{(13)}^{-1}\tt_{(14)}^{-1}\tt_{(13)}$ & 
 $[1, 0, 0, 0]$ \\ \addlinespace

$[8, 15]$ &$\mathbf{s}_1^{-2}\mathbf{s}_2\mathbf{s}_1^{-1}\mathbf{s}_3^{-1}\mathbf{s}_2^{-3}\mathbf{s}_3^{-1}$& 4 & $-7$ & $\mathbf{s}_2$ & 
$\tt_{(12)}^{-1}\tt_{(13)}^{-1}\tt_{(12)}\tt_{(23)}\tt_{(12)}^{-1}\tt_{(34)}^{-1}\tt_{(13)}^{-1}\tt_{(14)}^{-1}\tt_{(14)}^{-1}$ & 
 $[1, 3, -4, 1]$ \\ \addlinespace

$[8, 16]$ & $(\mathbf{s}_1^{-2}\mathbf{s}_2)^2\mathbf{s}_1^{-1}\mathbf{s}_2$& 3 & $-2$ & $e$ & 
$\tt_{(12)}^{-1}\tt_{(12)}^{-1}\tt_{(13)}^{-1}\tt_{(12)}\tt_{(23)}\tt_{(12)}^{-1}$ & 
 $[0, 2, -1]$ \\ \addlinespace

$[8, 17]$ &$\mathbf{s}_1^{-1}(\mathbf{s}_1^{-1}\mathbf{s}_2)^3\mathbf{s}_2$& 3 & 0 & $e$  & 
$\tt_{(12)}^{-1}\tt_{(12)}^{-1}\tt_{(13)}^{-1}\tt_{(12)}\tt_{(13)}$ & 
 $[1, -1, 1]$ \\ \addlinespace

$[8, 18]$ &$(\mathbf{s}_1^{-1}\mathbf{s}_2)^4$& 3 & 0 & $e$ & 
$\tt_{(12)}^{-1}\tt_{(23)}^{-1}\tt_{(12)}\tt_{(23)}\tt_{(12)}^{-1}$ & 
 $[-1, 3, -1]$ \\ \addlinespace

$[8, 19]$ &$\mathbf{s}_1^{3}\mathbf{s}_2\mathbf{s}_1^{3}\mathbf{s}_2$& 3 & 8 & $\mathbf{s}_2$  & 
$\tt_{(12)}^{-1}\tt_{(13)}\tt_{(12)}\tt_{(23)}\tt_{(23)}$ & 
 $[1, -5, 5]$ \\ \addlinespace

$[8, 20]$ & $\mathbf{s}_1^{3}\mathbf{s}_2^{-1}\mathbf{s}_1^{-3}\mathbf{s}_2^{-1}$& 3 & $-2$ & $\mathbf{s}_2$ & 
$\tt_{(12)}^{-1}\tt_{(13)}\tt_{(23)}^{-1}\tt_{(13)}^{-1}$ & 
 $[-1, 4, -2]$ \\ \addlinespace

$[8, 21]$ & $\mathbf{s}_1^{-3}\mathbf{s}_2^{-1}\mathbf{s}_1^2\mathbf{s}_2^{-2}$& 3 & $-4$ & $e$ & 
$\tt_{(12)}^{-1}\tt_{(12)}^{-1}\tt_{(13)}^{-1}\tt_{(23)}\tt_{(13)}^{-1}$ & 
 $[3, -3, 1]$ \\
\label{tab:knots}
\end{longtable}
\end{landscape}

\section{Conclusion}

The constructions developed in this paper suggest several natural
questions.  We collect them here as directions for future work.

\subsection{The Noncrossing Support of a Dual Pure Generator}
Let $t\in T$, and let $\tt=\P_c(t)$ be the square of the $c$-dual lift of $t$. \Cref{lem:dual_square_expansion} says that for each $\pi\in\NC(W,c^{-1})\setminus\{e,t\}$, the coefficient $[\T_\pi]\H(\P_c(t))$ is divisible by $(q-1)^{\ell_T(\pi)+1}$. We conjecture that this coefficient is in fact $0$.

\begin{conjecture}
For $t\in T$ and $\pi\in\NC(W,c^{-1})\setminus\{e,t\}$, we have 
\[
[\T_\pi]\H(\P_c(t))=0.
\]
\end{conjecture} 

\subsection{Rational Positive Factorizations}

The rational constructions in~\Cref{sec:rational} are currently
type-by-type.  In the classical types and in most exceptional cases, we
find positive dual $(c,p)$-conjugators, while the constructions presently
given in type $E_7$ for $p=5$ and in type $E_8$ involve signed
factorizations.

\begin{conjecture}
Let $W$ be a crystallographic Coxeter group, let $c$ be a standard Coxeter
element, and let $p>0$ be coprime to the Coxeter number.  Then there exists a
positive dual $(c,p)$-conjugator $\mathbf{g}\in\B_W$. Thus, the element $\XX=\mathbf g\B(c)^p\mathbf g^{-1}\B(c)^{-1}$ should admit a positive factorization of the form
\[
\XX
=
\P_c(t_1)\cdots\P_c(t_{r(p-1)/2}).
\]
\end{conjecture}

A type-uniform construction would explain the similarities among the
type $A$, $B$, and $D$ models, and it would determine whether the signs in
the exceptional tables are artifacts of the chosen lifts. 

\subsection{Flip Graphs for Reflection Subwords}

Suppose $\vect\in T^M$. Let $\pi\in W$. There is a natural flip graph whose vertex set is the set $\Sub_T(\vect,\pi)$ of subwords of $\vect$ that are reduced $T$-words for $\pi$.  In this graph, two distinct subwords are adjacent when their take sets differ by exchanging exactly one index. The subword sets
\[
\Sub_T\!\Bigl(\inv\bigl({\sf c}\,{\sf w}_\circ(c)\bigr),c^{-1}\Bigr)
\qquad\text{and}\qquad
\Sub_T\!\Bigl(\inv\bigl({\sf w}_\circ(c^{-1})\bigr)^2,c^{-1}\Bigr)
\]
recover the usual cluster and noncrossing-partition models, where the graph is the $1$-skeleton of the generalized associahedron. More
generally, subword sets such as
\[
\Sub_T\!\Bigl(\rev\bigl(\inv({\sf w}_\circ(c))\bigr)\,
\inv(c)\,
\inv({\sf w}_\circ(c)),c^{-1}\Bigr)
\]
give flip structures on all faces of the cluster complex, not
only on its facets.  In types $A$ and $B$, this gives little Schr\"oder/central Delannoy combinatorics (see sequences A001003 and A001850 in \cite{oeis}).  It would be interesting to determine which of these flip graphs are polytopal and how they depend on the chosen word for the full twist. 

\subsection{The $F$-triangle and the $M$-triangle}
Let $(s_1,\ldots,s_r)$ be a reduced word for $c$.  With some effort, we can encode F.~Chapoton's F-triangle $F(x,y)$ (see~\cite{ed2004enumerative,athanasiadis2007some} for definitions) in two ways using~\Cref{lem:coeff-delta} and a twisted version of $\Delta^2 \prod_{i=1}^r \s_i^2$, which---upon proper interpretation---we believe likely encodes the difference between the M-triangle and the F-triangle:
\begin{align*} 
F(x,y)&=[\T_e]\left.\left(\left(\frac{x}{q-1}\right)^r \H(\Delta^{4}) \prod_{i=1}^r \left((q-1)(y-x)/x+\T_{s_i}\right)  \sum_{w \in W} \left(\frac{q-1}{x}\right)^{\ell_T(w)}\T_{w}^{-1}\right)\right|_{q=1}\\
&=[\T_e]\left.\left(\left(\frac{x}{q-1}\right)^r \H(\Delta^{-2}) \prod_{i=r}^1 \left((q-1)y+(1-y)\T_{s_i}\right)  \sum_{w \in W} \left(\frac{(1-xy)(q-1)}{x(1-y)}\right)^{\ell_T(w)}\T_{w^{-1}}\right)\right|_{q=1}.
\end{align*}
It would be interesting to explore this further.

\subsection{Full Hecke Expansion via the Dual Braid Group?}

\Cref{thm:main} determines the leading term at $q=1$---it would be very desirable to extend this to the full Hecke expansion.  While individual
dual squares can be explicitly expanded in the Hecke algebra, products of these
expansions quickly leave the noncrossing partition lattice, and we lose control of them using only dual braid group combinatorics.

\section*{Acknowledgments}
Nathan Williams was partially supported by the National Science Foundation under Award No.~2246877, and is grateful to Jaeseong Oh for an invitation to talk about this work at a workshop on Diagrammatics for Combinatorics, Algebra, and Topology at Sungkyunkwan University.  We thank Minh-T\^am Trinh for many insightful conversations. 

We began this project about a year ago.  While we initially wrote all code by hand, as large language model capabilities increased we turned to ChatGPT~5.4,~5.5, and 5.6 Pro, Gemini~3.1 Pro, and Claude Opus~4.8 and~5 for assistance with code generation and editing.  

\bibliographystyle{plain}
\bibliography{bib.bib}

\end{document}